\documentclass[11pt,a4paper, oneside]{amsart}
\usepackage[dvipsnames]{xcolor}
\usepackage{braket,comment,soul}
\usepackage{hyperref}
\usepackage{pdfpages,faktor}
\usepackage{graphicx,mathabx,bbm,url}
\usepackage{caption}
\usepackage{subcaption}
\usepackage{mathrsfs} 
\usepackage{tikz-cd} 
\usepackage{bm}
\usepackage{enumitem}
\usepackage{mathtools}

\usepackage{geometry}
\usepackage{amssymb,amsmath,amsfonts,amsthm}

\newtheorem{thm}{Theorem}[section]
\newtheorem{cor}[thm]{Corollary}
\newtheorem{lemma}[thm]{Lemma}
\newtheorem{prop}[thm]{Proposition}

\newtheorem*{mainthm}{Main Theorem}

\theoremstyle{definition}
\newtheorem{defn}[thm]{Definition}
\newtheorem{question}[thm]{Question}

\newtheorem{rem}[thm]{Remark}

\theoremstyle{remark}

\DeclareMathOperator{\Int}{Int}

\DeclareMathOperator{\diam}{diam}
\DeclareMathOperator{\Aut}{Aut}
\DeclareMathOperator{\Crit}{Crit}

\numberwithin{equation}{section}

\newcommand{\A}{\mathbb{A}}
\DeclareMathOperator{\Id}{Id}
\newcommand{\Mod}{\operatorname{mod}}
\def\B{{\mathcal{B}}}
\def\D{{\mathbb{D}}}
\def\N{{\mathbb{N}}}
\def\Z{{\mathbb{Z}}}

\def\cD{\mathcal {D}}
\def\cV{\mathcal {V}}
\def\cW{\mathcal {W}}
\def\cE{\mathcal {E}}
\def\cR{\mathcal {R}}

\def\tG{\widetilde{\Gamma}}

\def\Ss{\mathcal{S}}

\renewcommand{\geq}{\geqslant}
\renewcommand{\leq}{\leqslant}
\renewcommand{\ge}{\geqslant}
\renewcommand{\le}{\leqslant}

\newcommand{\frakC}{\mathfrak{C}}

\renewcommand{\hat}{\widehat}

\newcommand{\eps}{\varepsilon}
\renewcommand{\epsilon}{\varepsilon}

\makeatletter
\DeclareFontFamily{U}{tipa}{}
\DeclareFontShape{U}{tipa}{m}{n}{<->tipa10}{}
\newcommand{\arc@char}{{\usefont{U}{tipa}{m}{n}\symbol{62}}}%

\newcommand{\arc}[1]{\mathpalette\arc@arc{#1}}

\newcommand{\arc@arc}[2]{%
  \sbox0{$\m@th#1#2$}%
  \vbox{
    \hbox{\resizebox{\wd0}{\height}{\arc@char}}
    \nointerlineskip
    \box0
  }%
}
\makeatother

\newcommand{\Kstar}{{\mathscr K_\star}}
\newcommand{\Gstar}{\mathscr G_\star}

\newcommand{\parab}{{\text{par}}}
\newcommand{\attr}{{\text{attr}}}
\newcommand{\ovl}{\overline}

\newcommand{\disk}{\mathbb{D}}
\newcommand{\C}{\mathbb{C}}
\newcommand{\K}{\mathscr{K}}
\newcommand{\Cc}{\widehat{{\C}}}
\newcommand{\Circle}{{{\mathbb S}^1}}
\newcommand{\sm}{\setminus}
\newcommand{\tT}{\widetilde{T}}
\renewcommand{\tilde}{\widetilde}
\newcommand{\clo}{{\text{cl}}\,\,}

\renewcommand{\phi}{\varphi}

\DeclareFontFamily{U}{mathb}{\hyphenchar\font45}
\DeclareFontShape{U}{mathb}{m}{n}{ <5> <6> <7> <8> <9> <10> gen * mathb <10.95> mathb10 <12> <14.4> <17.28> <20.74> <24.88> mathb12 }{}
\DeclareSymbolFont{mathb}{U}{mathb}{m}{n}
\DeclareFontSubstitution{U}{mathb}{m}{n}
\DeclareMathSymbol{\righttoleftarrow}{3}{mathb}{"FD}
\DeclareMathSymbol{\selfmap}{3}{mathb}{"FD}

\title[Transcendental correspondences along basins]
{Transcendental correspondences: when Fuchsian groups take over basins of entire maps}

\begin{author}[K.~Drach]{Kostiantyn Drach}
\address{Departament de Matem{\`a}tiques i Inform{\`a}tica, Universitat de Barcelona, Gran Via 585, 08007, Barcelona, Spain, and Centre de Recerca Matem{\`a}tica, Edifici C, Carrer de l'Albareda, Bellaterra, Barcelona, Spain}
\email{kostiantyn.drach@ub.edu}
\end{author}

\begin{author}[S.~Mukherjee]{Sabyasachi Mukherjee}
\address{School of Mathematics, Tata Institute of Fundamental Research, 1 Homi Bhabha Road, Mumbai 400005, India}
\email{sabya@math.tifr.res.in}
\end{author}

\begin{author}[L.~Pardo-Sim\'on]{Leticia Pardo-Sim\'on}
\address{Departament de Matem{\`a}tiques i Inform{\`a}tica, Universitat de Barcelona, Gran Via 585, 08007, Barcelona, Spain, and Centre de Recerca Matem{\`a}tica, Edifici C, Carrer de l'Albareda, Bellaterra, Barcelona, Spain}
\email{lpardosimon@ub.edu}
\end{author}

\thanks{The first and the third authors are partially supported by Agencia Estatal de Investigaci\'on grants PID2023-147252NB-I00, CNS2025-166633, and by the Severo Ochoa and Mar\'ia de Maeztu Program for Centers and Units of Excellence in R\&D (CEX2020-001084-M). The third author is a Serra Húnter fellow. The second author was partially supported by the Department of Atomic Energy, Government of India, under Project Identification No. RTI 4014, an endowment of the Infosys Foundation, and ANRF research project grant ANRF/ARGM/2025/000089/MTR. \\\indent Part of this work was done during the second author's visit to the Universitat de Barcelona, whose support is gratefully acknowledged.}

\date{\today}

\begin{document}

\begin{abstract}
  In this paper, we initiate a systematic study of $(\infty : \infty)$ holomorphic correspondences that naturally arise as conformal combinations (matings) of transcendental entire maps with Fuchsian groups. This construction parallels the recent theory of finite-degree algebraic correspondences associated with rational maps. 
  
  Our correspondence combines the dynamics of a transcendental entire function outside a distinguished attracting/parabolic basin with the action of a compatible Fuchsian group within it.  

We show that the resulting correspondence is the composition of a M{\"o}bius involution and the deleted covering correspondence of a meromorphic function having exactly one simple pole. When the transcendental entire function has finitely many singular values, so does this meromorphic function, and its line complex can be described explicitly. 
\end{abstract}

\maketitle

\setcounter{tocdepth}{1}
\tableofcontents

\section{Introduction}\label{intro_sec}
The celebrated Sullivan dictionary from the 1980s links the dynamics of rational maps in $\Cc$ with hyperbolic geometry, in particular with the action of Kleinian groups (discrete subgroups of $\mathrm{PSL}_2(\mathbb C)$) and Fuchsian groups (discrete subgroups of $\mathrm{PSL}_2(\mathbb R)$) \cite{Sul85}. Much earlier, Fatou suggested that the dynamics of rational maps and the action of Kleinian/Fuchsian groups could be understood together through the dynamics of a single holomorphic object \cite{Fat29}. Such an object should naturally be a many-to-many mapping, namely, a \emph{correspondence}. Indeed, a rational map is generically many-to-one: several points may have the same image. By contrast, the action of a finitely generated group, through its generators, is naturally one-to-many: a single point can be mapped by each generator. Thus, if a common holomorphic object combining these two types of dynamics exists, it should be a correspondence.

Typically, combining rational dynamics with the dynamics of Kleinian groups leads to an \emph{algebraic} holomorphic correspondence. The study of algebraic correspondences is now a rapidly developing area of research. It was initiated by Bullett and Penrose in the mid-1990s \cite{BP}. They constructed a $2$-to-$2$ correspondence by mating the dynamics of quadratic polynomials with the action of the modular group $\mathrm{PSL}_2(\mathbb Z)$.
For a detailed account on more recent results and systematic studies in the theory of algebraic (holomorphic and anti-holomorphic) correspondences, we refer the reader to the survey articles \cite{LM26b,ICMSurvey}.

A natural question is whether the theory of algebraic correspondences can be extended to include the dynamics of entire, or more generally meromorphic, functions. The resulting holomorphic correspondences would no longer be algebraic. The only available result in this direction is due to Bullett and Freiberger \cite{BF}, where the authors use a limiting argument and the construction of \cite{BP} to build an $\infty$-to-$\infty$ correspondence combining the action of a certain Fuchsian group with (a part of the non-escaping) dynamics of the entire map $z \mapsto \mu \left(\sin \sqrt{z}\right)^2$, for a certain chosen parameter $\mu \in \mathbb C \setminus \{0\}$. 

In this paper, we initiate a systematic study of \emph{transcendental} correspondences that naturally combine compatible Fuchsian group actions with the dynamics of \emph{transcendental entire} maps. The correspondences we construct arise as compositions of M{\"o}bius involutions and covering correspondences of meromorphic maps having a unique (simple) pole. Our construction is inherently different from that of \cite{BF}. We use the strategies developed in \cite{LLM24,BLLM} (also see \cite{MV25}); however, because of the additional difficulties inherent in transcendental dynamics, such as infinite degree and the presence of the essential singularity at $\infty$, several steps require new arguments and additional care.

\subsection{Dynamics of correspondences}

Let $\C_* = \C \sm \{0\}$ be a punctured plane. For us, a holomorphic correspondence will be a curve $\mathfrak{C} \subset \C_* \times \C_*$ given by an equation $P(z,w)=0$, where $P$ is holomorphic in both variables. In this way, 
\[
(z,w) \in \mathfrak C \quad \Longleftrightarrow \quad P(z,w)=0.
\]
For such an object, we have natural projection maps
\[
\pi_1 \colon \mathfrak{C} \to \mathbb C_*, \quad \pi_1(z,w)=z,\qquad \pi_2 \colon \mathfrak{C} \to \mathbb C_*, \quad \pi_2(z,w)=w.
\]

Since the correspondence lives in the product of $\C_*$ with itself, we can study the dynamics of $\frakC$ by iterating the correspondence in the following way. We say that a bi-infinite sequence $\ldots \mapsto z_{-1} \mapsto z_0 \mapsto z_1 \mapsto z_2 \mapsto\ldots\mapsto z_n \mapsto z_{n+1} \mapsto \ldots$ in $\C_*$ is a \emph{(full) orbit of $z_0$ under $\frakC$} if $(z_n, z_{n+1}) \in \frakC$ for all $n \in \mathbb Z$. Note that for a given $z_n$, the point $z_{n+1}$ belongs to the set $\pi_2\left(\pi_1^{-1}(z_n)\right)$.

We will assume that, except for finitely many points $z_{n+1}$ (called \emph{singular values}), the mapping $\pi_2 \circ \pi_1^{-1} \colon z_n \mapsto z_{n+1}$ locally extends to a conformal map in some small neighborhood of $z_{n}$ in $\C_*$. This extension is called a \emph{forward branch of the correspondence}. Similarly, a local conformal extension of the mapping $\pi_1 \circ \pi_2^{-1} \colon z_{n+1} \mapsto z_n$ is called a \emph{backward branch of the correspondence}. In this way, the iteration of the correspondence should be understood as compositions of forward and backward local holomorphic branches. 

We say that a set $\mathcal K \subset \C_*$ is \emph{fully invariant} under $\frakC$ if for $(z,w) \in \frakC$,
\[
\quad z \in \mathcal K \, \Longleftrightarrow \,  w \in \mathcal K;
\]
if only the implication $\Rightarrow$ holds, then $\mathcal K$ is called \emph{forward invariant}, and if only the implication $\Leftarrow$ holds, it is called \emph{backward invariant}. Note that a forward (resp., backward) invariant set is preserved under the dynamics of forward (resp., backward) branches of the correspondence.

The number of forward branches of the correspondence around a generic point is the \emph{forward degree}; and the number of backward branches (again, around a generic point) is called the \emph{backward degree}. Both of these numbers (finite or infinite) are well-defined.

Note that one can compose correspondences to obtain a correspondence. Namely, if $\frakC_1$ and $\frakC_2$ are correspondences, then their composition $\frakC_1\circ\frakC_2$ is the correspondence such that
\[
(z,w) \in \frakC_1 \circ \frakC_2 \quad \Longleftrightarrow \quad \exists u \in \C_* \quad \text{such that } \quad (z,u) \in \frakC_1 \quad \text{and} \quad (u,w)\in \frakC_2.
\]
In this way, iterating a correspondence $\frakC$ can be viewed as repeatedly composing it with itself. Finally, since every holomorphic map $f$ can be viewed as a correspondence $w-f(z)=0$, we can compose correspondences and maps.

A special class of correspondences is \emph{covering correspondences}. We say that $\mathfrak C$ is a {covering correspondence} if 
\[
(z,w) \in \frakC \quad \Longleftrightarrow \quad h(z)=h(w),
\]
where $h \colon \C \to \Cc$ is a meromorphic map. For a covering correspondence, the forward and backward branches are given by local deck transformations $h^{-1} \circ h$ of the map $h$. Finally, a \emph{deleted} covering correspondence is a correspondence of the form
\[
(z,w) \in \frakC \quad \Longleftrightarrow \quad \frac{h(z)-h(w)}{z-w}=0.
\]
In this case, the correspondence contains all branches of local deck transformations except the identity branch (as it has been ``deleted"). 

The dynamics of a covering, or a deleted covering, correspondence is just a permutation of points in fibers, and hence it is quite trivial. On the contrary, the dynamics of covering correspondences composed with other correspondences or maps can be highly non-trivial, as we will see in our Main Theorem. 

\subsection{Statement of the main result}

To state the result, we need to introduce two classes of maps and groups that can be combined to produce a correspondence.

On the map side, we introduce classes $\Kstar(d)$, $d\ge2$, $\star\in\{\attr,\parab\}$, of transcendental entire maps with a marked fixed immediate $\star$-basin $U$, where $d=\deg(f\colon U\to U)$. In the parabolic case, these are strongly geometrically finite maps with bounded Fatou components that are Jordan disks; in the attracting case, they are hyperbolic maps with bounded Fatou components that are quasidisks. See Definitions~\ref{Def:Kp} and~\ref{Def:Kattr}.

As an example, the map $z \mapsto \sin(z)$ belongs to $\K_\parab(2)$, while the map $z \mapsto \frac{\pi}{2}\sin(z)$ belongs to $\K_\attr(2)$ (see Figure~\ref{Fig:Sins}). 

\begin{figure}[t]
\captionsetup{width=0.98\linewidth}
    \captionsetup{width=0.98\linewidth}
    \centering
    \includegraphics[width=0.7\linewidth]{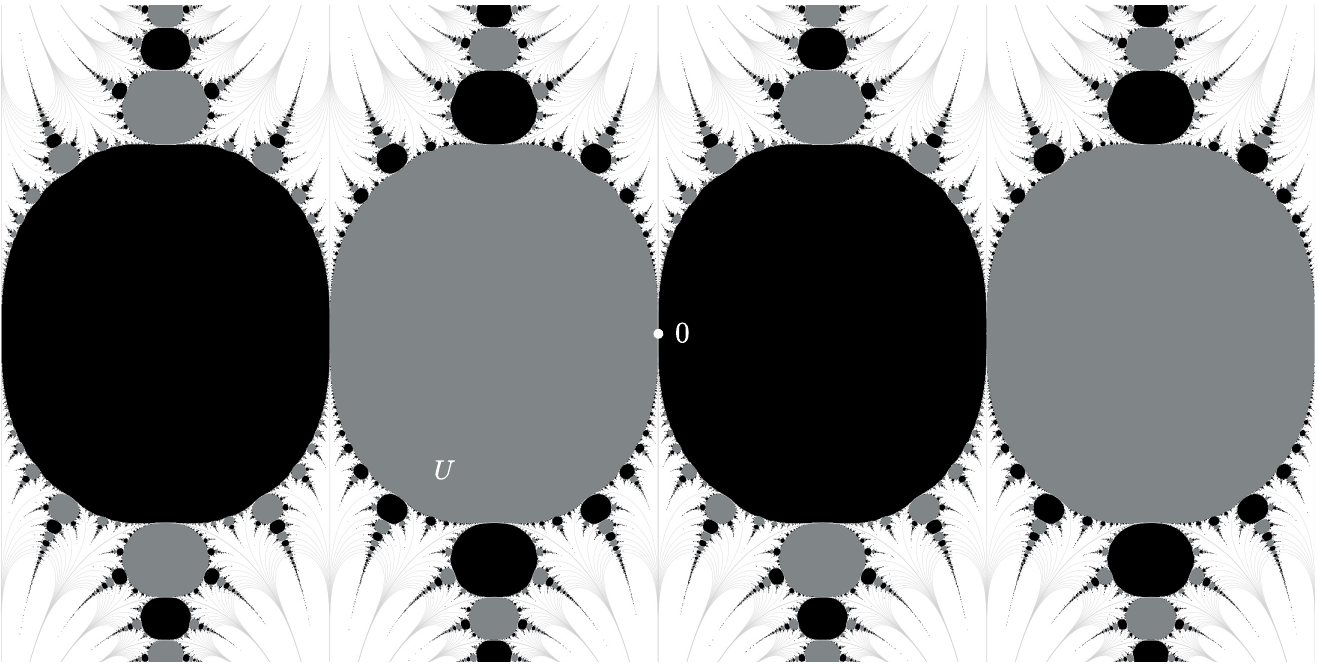}
\vspace{-3mm}
    
    \rule{0.85\linewidth}{0.1pt}

\vspace{1mm}    

\hspace{-2.5mm}
    \includegraphics[width=0.7\linewidth]{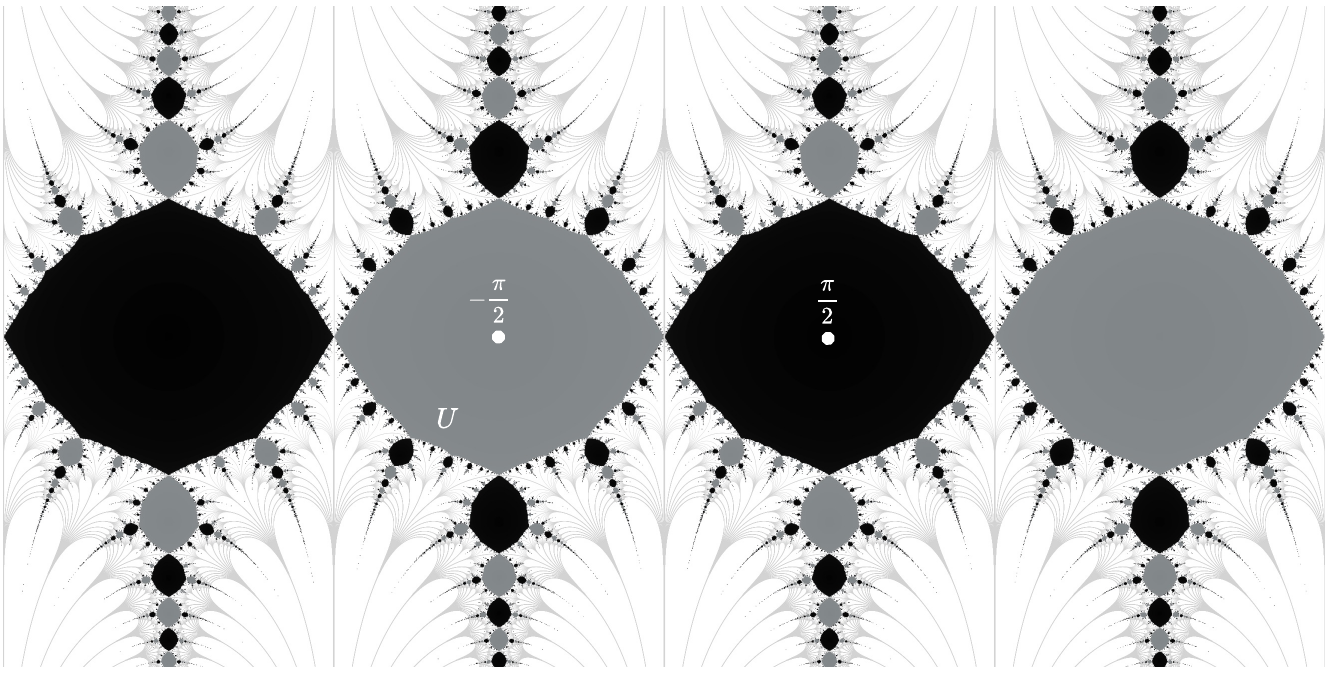}
    
    \caption{The degree-two examples on the entire-map side. Top: $f(z)=\sin z\in\mathscr K_{\parab}(2)$. The unique fixed point $0$ is parabolic and has two immediate parabolic basins, shown in black and gray. Bottom: $f(z)=\frac{\pi}{2}\sin z\in\mathscr K_{\attr}(2)$; the unique fixed points $\pm\pi/2$ are superattracting, and their immediate basins are shown in black and gray. In both examples, the Fatou set consists precisely of the grand orbits of the two indicated immediate basins. Either component may be chosen as the marked basin $U$. Our construction of the correspondence keeps the dynamics of $f$ outside the grand orbit of $U$ and replaces the compatible group dynamics from $\mathscr G_{\star}(2)$ (see Figure~\ref{Fig:Groups}), with $\star=\parab$ or $\attr$ respectively.} 
    \label{Fig:Sins} 
\end{figure}

For each $f \in \Kstar(d)$ with the marked $\star$-basin $U$, we can consider the \emph{dynamics of $f$ outside the grand orbit of $U$}, i.e., the restriction $f \colon \C \sm \bigcup_{n \ge 0} f^{-n}(U) \to \C \sm \bigcup_{n \ge 0} f^{-n}(U)$. This is the part of the dynamics of $f$ that we want to keep in our correspondence. In particular, we keep the dynamics of $f$ on its Julia set. On the other hand, we want to replace the dynamics of $f$ on the grand orbit of $U$ with a compatible group dynamics.

For this, on the group side, we introduce a class of \emph{compatible} Fuchsian groups $\Gstar(d)$, $d\ge 2$, $\star \in \{\attr, \parab\}$ consisting of groups $G$ acting on $\disk$ that are associated to finite-area zero-genus orbifolds with finitely many punctures and with at most two orbifold points so that at most one of them is of order greater than $2$. For each $G \in \Gstar(d)$, we can construct the corresponding factor Bowen--Series map $F_G\colon \ovl\disk \sm T \to \ovl\disk$ defined on a subset $\ovl\disk \sm T$ of the unit disk. The restriction $F_G\vert_{\Circle}$ is an expansive covering map of degree $d \ge 2$, which makes this action compatible with the action of $f\in\K_*(d)$ on $\partial U$. The precise conditions on the groups in $\Gstar(d)$ will be given in Definitions~\ref{Def:Gattr} and \ref{Def:Gparab}. 

As examples, the classical \emph{modular group} generated by $z \mapsto z+1$ and $z \mapsto -\frac{1}{z}$ in the upper half-plane $\mathbb H$ belongs to $\Gstar(2)$ for both choices of $\star$. On the other hand, the groups associated to punctured spheres possibly with an order two orbifold point belong to $\mathscr{G}_{\attr}$ (see Figure~\ref{Fig:Groups}).

The main result of the paper is the following theorem (see Theorem~\ref{corr_mating_thm} for the more precise version).

\begin{mainthm}\label{main_thm}
For every $d \ge 2$, $\star \in \{\attr, \parab\}$, and for every choice of a map $f \in \Kstar(d)$ with a marked $\star$-basin $U$ and a Fuchsian group $G \in \Gstar(d)$, there exists a holomorphic correspondence $\frakC \subset \C_* \times \C_*$ and a pair of sets $\pmb{\cR}_0 \subset \C_*$ and $\pmb{\cE} \subset \C_*$ \footnote{Here, ``R'' stands for \emph{r}esident set, and ``E" stands for the \emph{e}xpulsion set; this notation will be further explained in Section~\ref{corr_sec}.} such that:

\begin{enumerate}[leftmargin=8mm]
    \item 
    $\pmb{\cR}_0$ is closed and forward invariant under $\frakC$, while $\pmb{\cE}$ is open and fully invariant under $\frakC$.
    \item 
    If $\pmb{\cR} := \pmb{\cR}_0 \cup \eta(\pmb{\cR}_0)$, where $\eta(z) = 1/z$ is the involution, then $\pmb{\cR}$ is fully invariant under $\frakC$ and the sets $\pmb{\cR}$ and $\pmb{\cE}$ provide a partition of $\C_*$, i.e., $\C_*= \pmb{\cR} \sqcup \pmb{\cE}$.
    \item 
    There exists a forward branch of the correspondence $\frakC$ that preserves $\pmb{\cR}_0$, and the restriction of this branch $\pmb{\cR}_0$ is conformally conjugate to the dynamics of $f$ outside the grand orbit of $U$.
    \item 
    \label{Main:Item:4}
     There exists a connected component $\pmb{\cW}_0$ of $\pmb{\cE}$ such that $\pmb{\cW}_0$ is a topological disk and all the forward branches of the correspondence $\frakC$ preserving $\pmb{\cW}_0$ are univalent and generate a group that is conformally conjugate to the action of $G$ on $\disk$.
\end{enumerate}

Furthermore, $\frakC$ is a correspondence of bi-degree $(\infty:\infty)$ given as a composition of a deleted covering correspondence and a M{\"o}bius involution. Explicitly, there exists a meromorphic map $\mathfrak{h} \colon \C \to \Cc$ with a unique simple pole at $0$ so that the correspondence $\frakC$ can be written as 
\begin{equation}
(z, w) \in \mathfrak{C} \iff \frac{\mathfrak{h}(w)-\mathfrak{h}(\eta(z))}{w- \eta(z)}=0.
\label{corr_eqn}
\end{equation}
\end{mainthm}

\begin{figure}[th]
\captionsetup{width=0.98\linewidth}
    \centering
    \includegraphics[width=0.37\linewidth]{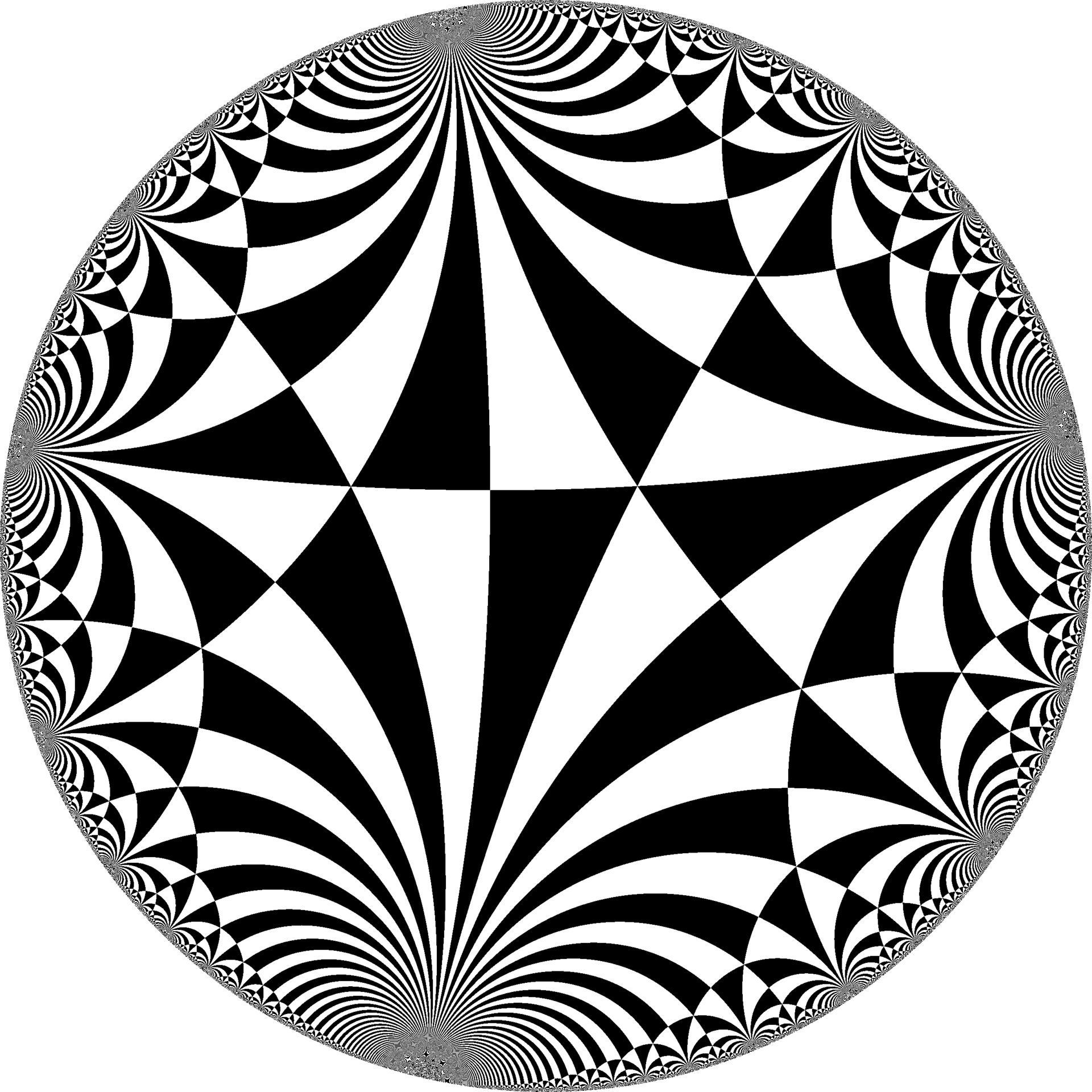}\quad \includegraphics[width=0.37\linewidth]{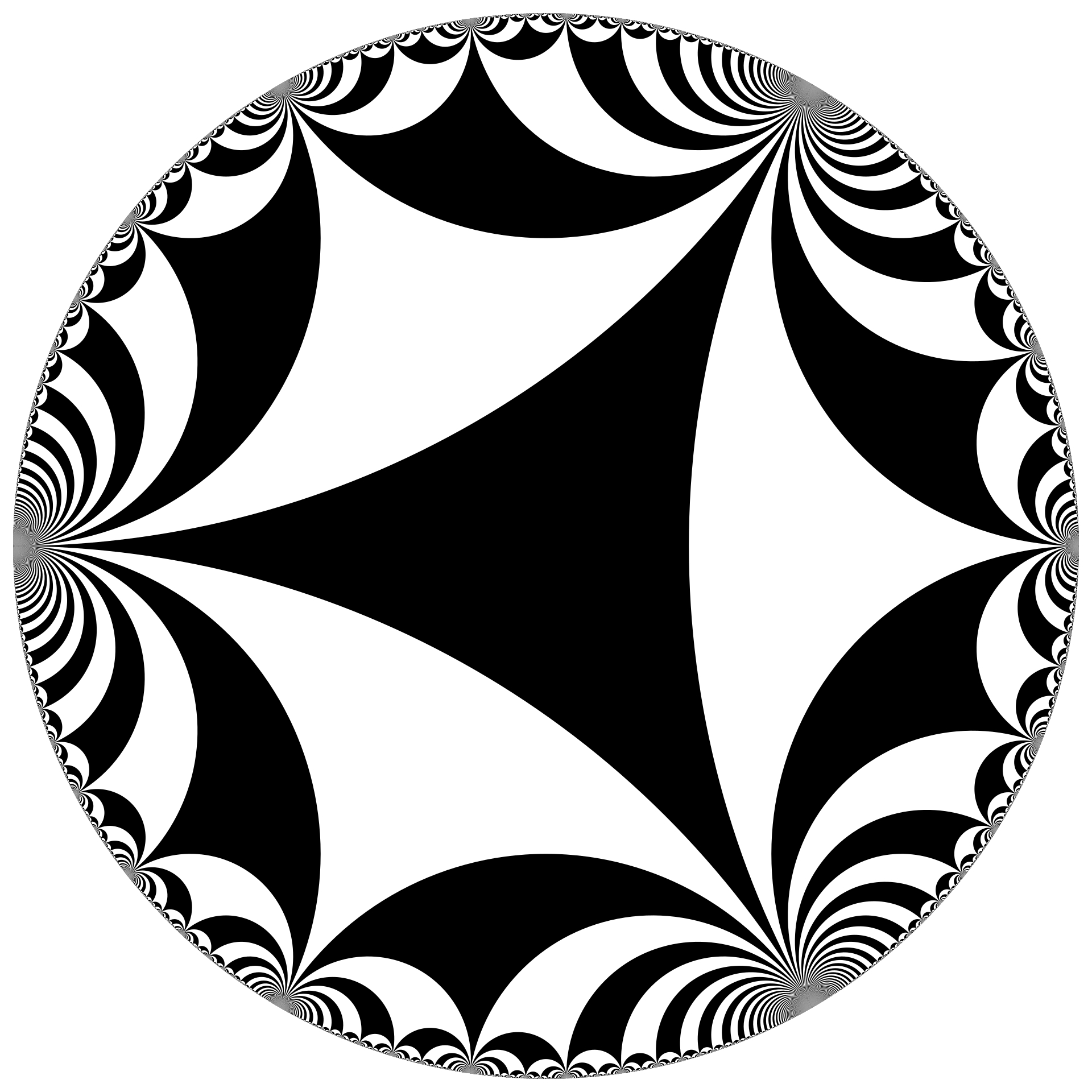}
    \caption{Depicted are tessellations of the unit disk under the actions of the modular group (left) and the Fuchsian group uniformizing a sphere with two punctures and one orbifold point of order $2$ (right). These are precisely the two groups in $\mathscr{G}_\attr(2)$. On the other hand, $\mathscr{G}_{\parab}(2)$ consists only of the modular group.}
    \label{Fig:Groups}
\end{figure}

\begin{rem}
We note that in the case when $f$ has finitely many singular values (i.e., $f$ is in the Speiser class), one can explicitly construct $\mathfrak{h}$ (up to pre- and post-composition with quasiconformal homeomorphisms so that the resulting map is holomorphic) based on the line complex (Speiser graph) of $f$ and the properties of $G$. We will explain this in Section~\ref{speiser_subsec}.
\end{rem}

\subsection{Connections with matings between rational maps and Fuchsian groups}

In the literature, most mating constructions between polynomial maps (respectively, parabolic rational maps) and Fuchsian groups replace the action of the polynomial on the basin of infinity (or that of the parabolic rational map on its marked fully invariant immediate parabolic basin) with the action of the group (see \cite{LLM24,BLLM}). The situation for transcendental entire functions is rather different, as there is no genuine analog to the basin of infinity (or a fully invariant immediate parabolic basin) for such maps.

In this paper, we take a different approach. Namely, we replace the action of a transcendental entire map $f$ on an invariant (attracting/parabolic) Fatou component $U$. Note that $U$ cannot be fully invariant. This creates an interesting distinction between the polynomial case and our current setup. Because the map $f$ may have critical points within the iterated preimages of $U$, we can only expect a group structure on the specific `copy' of $U$ where a group is inserted, but not necessarily on its iterated preimages. Specifically, this copy of $U$ corresponds exactly to the simply connected domain $\pmb{\cW}_0$ appearing in Item~\eqref{Main:Item:4} of the Main Theorem.

It is also worth remarking that the construction of conformal combinations carried out in this paper is quite general. Instead of replacing the dynamics of $f$ only on a single invariant Fatou component, one could replace its action on several such components with compatible Fuchsian groups. This would result in a transcendental correspondence that combines the dynamics of a transcendental entire map (at least on its Julia set) with the actions of a finite collection of Fuchsian groups (cf. \cite{MV25}). This should give rise to embeddings of products of Teichm{\"u}ller spaces of several genus zero orbifolds in the space of transcendental correspondences (or to the family of meromorphic maps that define these correspondences via Equation~\eqref{corr_eqn}). It would be interesting to know pre--compactness / boundary behavior of such embeddings in the spirit of the degeneration results for algebraic correspondences established in \cite{LMM26}.

\subsection{Organization of the paper} 

The paper is organized as follows. In Section~\ref{Sec:mateable_maps}, we describe the class of transcendental entire maps $\Kstar(d)$ used in the construction of our correspondences. In Section~\ref{Sec:MateableGroups}, we introduce the compatible Fuchsian groups $\Gstar(d)$ and their associated factor Bowen--Series maps. In Section~\ref{Sec:Combination}, we prove the combination theorem for maps in $\Kstar(d)$ and the factor Bowen--Series maps of groups in $\Gstar(d)$. The resulting object is a partially defined map $g \colon \C \sm \tilde T \to \C$. The construction in that section relies on quasiconformal or David surgery, depending on $\star$. In Section~\ref{analytic_description_sec}, we give an analytic description of the conformal combination map $g$ in terms of a meromorphic map $\mathfrak{h} \colon \C \to \Cc$. In particular, under certain assumptions on $f$, we explain how to construct the line complex of $\mathfrak{h}$ from the line complex of $f$ and the properties of $G$. Finally, in Section~\ref{corr_sec}, we use the map $\mathfrak{h}$ to define the required correspondence explicitly; see \eqref{corr_eqn}. The remainder of that section is devoted to verifying the properties of the correspondence asserted in the Main Theorem.

\section{Mateable entire maps}\label{Sec:mateable_maps}
The purpose of this section is to describe the entire maps that will be used on the holomorphic side of the combination construction.  The hypotheses below ensure that the Fatou components in which we perform the surgery have controlled boundary geometry, and that pullbacks of these components do not develop uncontrolled branching.  We shall use two closely related classes.  The first is a parabolic class, denoted   $\mathscr K_\parab$, including a subclass of strongly geometrically finite entire maps.  The second is an attracting class, $\mathscr K_\attr$,  consisting of hyperbolic entire maps whose Fatou components are bounded quasidisks.

Recall that for a transcendental entire map $f$, the singular set, $S(f)$, is defined as the closure of the set of finite asymptotic and critical values of $f$. Equivalently, it consists of the set of singularities of the inverse function $f^{-1}$.  We say that $f$ is in the \textit{Eremenko--Lyubich class $\B$} if $S(f)$ is bounded. We denote its postsingular set by  $P(f):=\overline{\bigcup_{n\ge 0} f^n(S(f))}$. The maps we consider are subclasses of functions in $\B$, that we describe in what follows.

\subsection{Strongly geometrically finite maps with parabolic basins}

Geometrically finite polynomials are those for which each critical point
in the Julia set has a finite orbit. The natural extension of this definition to transcendental entire functions is the following, first given in \cite{MB10_landing_geomfinite}, further studied in \cite{ARS22_geomfinite}.
\begin{defn} A transcendental entire function $f$ is \emph{geometrically finite} if $F(f)\cap S(f)$ is compact and $J(f)\cap P(f)$ is finite. 
	We say that $f$ is \emph{strongly geometrically finite} if it is geometrically finite with \textit{bounded criticality} on the Julia set, the latter meaning that $J(f)$ contains no finite asymptotic values of $f$ and the local degree of $f$ at points of $J(f)$ is uniformly bounded.
\end{defn}
If $f$ is geometrically finite, then $F(f)$ is either empty or consists of the basins of attraction of finitely many attracting or parabolic cycles, \cite[Proposition 2.6]{MB10_landing_geomfinite}. The class of strongly geometrically finite entire functions provides a natural transcendental analogue of geometrically finite rational maps. Results of \cite{ARS22_geomfinite} show that these maps admit an expanding structure near the Julia set and a well-developed theory of dynamic rays, closely paralleling the polynomial setting. See \cite[Section 11]{ARS22_geomfinite} for several examples of strongly geometrically finite entire functions, including suitable sine maps, like $z\mapsto \sin(z)$.

For our parabolic construction, the point of the following theorem is that a criticality condition is equivalent to a topological statement about all Fatou components.  In particular, the relevant parabolic basins are bounded Jordan domains, which is the amount of boundary regularity we need for our combination result.

\begin{thm}[{\cite[Theorem 1.8]{ARS22_geomfinite}}]\label{thm:bounded_geomfinite} Suppose that $f$ is strongly geometrically finite. Then the following are equivalent:
	\begin{enumerate}[leftmargin=8mm]
		\item every component of $F(f)$ is a bounded Jordan domain;
		\item the map $f$ has no asymptotic values, and every component of $F(f)$ contains at most finitely many critical points.
	\end{enumerate}	
\end{thm}

\begin{defn}[Class $\mathscr K_\parab$]
	\label{Def:Kp}
	We say that $f\in \mathscr K_\parab$ if $f$ is strongly geometrically finite, satisfies one of the conditions in Theorem~\ref{thm:bounded_geomfinite}, and has at least one parabolic fixed point. For $d\ge2$, define $\mathscr K_{\parab}(d):=\{f\in\mathscr K_{\parab}:\deg(f\colon U\to U)=d\}$, where $U$ is the marked invariant immediate parabolic basin.
\end{defn}

\subsection{Hyperbolic maps}

Although strongly geometrically finite maps form a broader class, the attracting construction requires uniform quasidisk geometry of Fatou components. For this reason, we restrict attention to hyperbolic entire maps, the natural transcendental analogue of hyperbolic rational maps; see \cite{RS_hyperbolicity} for different characterizations of these maps.

\begin{defn}[Hyperbolicity]
	\label{Def:Ka}
	A transcendental entire function $f$ is \emph{hyperbolic} if $f\in \B$ and every element of $S(f)$ belongs to the basin of some attracting periodic cycle of $f$.
\end{defn}
If $f$ is hyperbolic, then $F(f)$ is the union of the attracting basins of finitely many attracting cycles, and every connected component of $F(f)$ is simply connected. Moreover, there exists a compact set $K\subset F(f)$ such that
\begin{equation}\label{eq:compact_hyperbolic}
f(K)\subset \operatorname{int}(K)
\qquad\text{and}\qquad
S(f)\subset \operatorname{int}(K),
\end{equation}
see \cite[Proposition~2.1]{BFR15_hyperbolic}. The following theorem characterizes those hyperbolic maps whose Fatou components are bounded quasidisks.

\begin{thm}[{\cite[Theorems 1.2 and 1.10]{BFR15_hyperbolic}}]\label{thm:bounded_hyperbolic} Let $f\in \B$ be hyperbolic. Then the following are equivalent:
	\begin{enumerate}[leftmargin=8mm]
		\item every component of $F(f)$ is a bounded quasidisk;
		\item the map $f$ has no asymptotic values, and every component of $F(f)$ contains at most finitely many critical points.
	\end{enumerate}	
\end{thm}

Motivated by this result, we introduce the class of hyperbolic maps that will be used in the attracting construction.
\begin{defn}[Class $\mathscr K_\attr$]
	\label{Def:Kattr}
	We say that $f\in\mathscr K_\attr$ if $f$ is hyperbolic, has an attracting fixed point, satisfies one of the equivalent conditions of Theorem~\ref{thm:bounded_hyperbolic}, and there exists $N\in\mathbb N$ such that every component of $F(f)$ contains at most $N$ critical points, counted with multiplicity. For $d\ge2$, define $\mathscr K_{\attr}(d):=\{f\in\mathscr K_{\attr}:\deg(f\colon U\to U)=d\}$, where $U$ is the marked invariant immediate attracting basin.
\end{defn}

\begin{lemma}[Uniform degree bound]
	\label{Lem:UnivComp}
	Let $f\in\mathscr K_\attr$. Then there exists a constant $L\geq1$ such that
	the following holds. If $U$ is a periodic Fatou component and $V$ is an
	iterated preimage component of $U$, and if $k\geq0$ is minimal with
	$f^k(V)=U$, then
	\[
	\deg(f^k\colon V\to U)\leq L.
	\]
\end{lemma}
\begin{proof}
Since $f$ is hyperbolic, by \eqref{eq:compact_hyperbolic}, $S(f)$ is compactly contained in $F(f)$ and so intersects only finitely many Fatou components. Let $\mathcal S$ be the finite collection of Fatou components that meet $S(f)$. Let $U$ be a periodic Fatou component, and let 
	\[
	V=V_k \xrightarrow{f} V_{k-1}\xrightarrow{f}\cdots
	\xrightarrow{f}V_1\xrightarrow{f}V_0=U
	\]
	be the chain of Fatou components, where $k$ is minimal. For each $j$, the map $f\colon V_j\to V_{j-1}$ is proper of finite degree by \cite[Proposition~2.8]{BFR15_hyperbolic}. If
	$V_{j-1}\cap S(f)=\emptyset$, then this map has no singular values over $V_{j-1}$; since $V_{j-1}$ is simply connected, it is univalent. Thus non-univalent steps can occur only when $V_{j-1}\in\mathcal S$. Since $k$ is minimal, the Fatou components
	$V_0,\ldots,V_k$ are pairwise distinct. Hence each component of
	$\mathcal S$ occurs at most once in the chain. Moreover, by the definition of $\mathscr K_\attr$ and the
	Riemann--Hurwitz formula, $\deg(f\colon V_j\to V_{j-1})\leq N+1$
	for every $j$. Hence $\deg(f^k\colon V\to U) \leq (N+1)^{\#\mathcal S}.$ Taking $L=(N+1)^{\#\mathcal S}$ proves the claim.
\end{proof}

The previous lemma provides uniform control on the degrees of pullbacks of periodic Fatou components. Combined with the quasidisk characterization of Theorem~\ref{thm:bounded_hyperbolic}, this allows us to obtain uniform geometric control on all Fatou components.

\begin{prop}[Uniform quasidisks] 
	\label{Prop:4David}
	Let $f\in \mathscr K_\attr$. Then, the following hold:
	\begin{enumerate}[leftmargin=8mm]
		\item \label{item:bounded_diam} For every positive $\epsilon>0$ the number of Fatou components whose diameter with respect to the spherical distance exceeds $\epsilon$ is finite.
		\item \label{item:uniform_qdisks} All Fatou components of $f$ are \textit{uniform quasidisks}, that is, there exists a constant $K\geq 1$ such that the boundary of each Fatou component of $f$ is a $K-$quasicircle.
	\end{enumerate}		
\end{prop}

\begin{proof}
Item~\eqref{item:bounded_diam} follows from the proof of \cite[Theorem~2.5]{BFR15_hyperbolic}, which verifies condition~(b) of \cite[Lemma~2.3]{BFR15_hyperbolic}; compare \cite{Bergweiler_morosawa} for similar results for \textit{semihyperbolic} maps.
		
To prove \eqref{item:uniform_qdisks}, note that since $f$ is hyperbolic, there are only finitely many periodic
Fatou components. Let $\mathcal U_0$ denote this finite collection. By
Theorem~\ref{thm:bounded_hyperbolic}, each $U \in \mathcal U_0$ is a quasidisk, and so there is a quasiconformal map $\phi_U\colon\C\to\C$ such that $\phi_U(\D)=U$. Since $S(f)\Subset F(f)$ and $\partial U\subset J(f)$, we may choose
numbers $0<r_U<1<R_U$ such that the annulus
$A_U:=\phi_U(\A(r_U,R_U))$
contains $\partial U$ and satisfies $\overline{A_U}\cap S(f)=\emptyset.$

Let $V$ be a component of some iterated preimage of $U\in\mathcal U_0$, and let $k\geq0$ be minimal such that $f^k(V)=U$. By
Lemma~\ref{Lem:UnivComp}, the degree $d:=\deg(f^k\colon V\to U)$
is bounded above by a constant $L$ independent of $V$, $U$, and $k$. Let $\widetilde A_V$ be the component of $f^{-k}(A_U)$ that contains
$\partial V$. Since $A_U\cap S(f)=\emptyset$, the restriction $f^k:\widetilde A_V\to A_U$ is a covering map. Since $A_U$ is doubly connected and $\deg(f^k|_{\widetilde A_V})=d<\infty$, Lemma~2.7 of
\cite{BFR15_hyperbolic} implies that $\widetilde A_V$ is also doubly
connected. Hence $f^k:\widetilde A_V\to A_U$ is a covering map of degree $d$ between annuli. Comparing this covering with
\[
P_d(z)=z^d\colon \A(r_U^{1/d},R_U^{1/d})\to \A(r_U,R_U),
\]
we can lift $\phi_U$ to a quasiconformal map $h_V\colon \A(r_U^{1/d},R_U^{1/d})\to \widetilde A_V$
satisfying $f^k\circ h_V=\phi_U\circ P_d.$ By \cite[Lemma~3.8]{EFGP24}, the map $h_V$ extends to a quasiconformal
self-map of $\C$ whose dilatation depends only on the quasiconformal constant
of $\phi_U$, on $r_U,R_U$, and on $d$.

Since $U$ ranges over the finite family $\mathcal U_0$ and $d\leq L$, the extension constants are uniformly bounded. Hence
there exists $K\geq1$ such that every Fatou component which eventually maps to
a periodic Fatou component is a $K-$quasidisk. Since every Fatou component of
a hyperbolic entire map eventually maps to a periodic Fatou component, all
Fatou components are $K-$quasidisks.
\end{proof}

\section{Mateable Fuchsian groups and factor Bowen--Series maps}
\label{Sec:MateableGroups}

The purpose of this section is to collect the group--side results that will be
used in the construction of the correspondences.  We do not prove any new
results here.  Rather, we recall the relevant notions and state the known
theorems in a form adapted to our setting.  

The point is to encode the boundary
dynamics of the Fuchsian group by a continuous expansive circle map. These circle maps play the role of external maps on the group side: they are the analogues of the map $z\mapsto z^d$ for a polynomial with the connected Julia set. The difference is that
the group--side external maps are usually only piecewise analytic and have
finitely many singular points where the local dynamics is parabolic. The circle maps we need come from the factor Bowen--Series construction of \cite{MM, MM_survey}.  

Throughout this section, the integer $d$ denotes the degree of the circle map.

\subsection{Mateable circle maps}
\label{SSec:MateableCircleMaps}

The circle maps that enter our mating constructions belong to a class
introduced in \cite[\S2]{MM_survey}.  For $A:\Circle\to\Circle$ a continuous
map, the \emph{grand orbit} of a point $x\in\Circle$ is
\[
\operatorname{GO}_A(x)
=
\{y\in\Circle:\ A^m(x)=A^n(y)\text{ for some }m,n\ge0\}.
\]
If $\Gamma<\Aut(\D)$ is a Fuchsian group with limit set
$\Lambda(\Gamma)=\Circle$, we say that $A$ is \emph{orbit equivalent} to
$\Gamma$ on $\Lambda(\Gamma)$ if
\[
\operatorname{GO}_A(x)=\Gamma\cdot x,
\qquad
x\in\Lambda(\Gamma).
\]
That is, the dynamics of $A$ records the same boundary orbit relation as the
action of $\Gamma$. The following definition is taken from \cite[Definition~2.1]{MM_survey}, \cite[Definition~2.10]{MM24}.

\begin{defn}[Mateable and virtually mateable maps]
	\label{Def:MateableCircleMap}
\leavevmode\par
\begin{itemize}	[leftmargin=8mm]
\item	A continuous map $A:\Circle\to\Circle$ is called a \emph{mateable map
		associated with a Fuchsian group $\Gamma$} if:
	\begin{enumerate}[leftmargin=6mm]
		\item $A$ is orbit equivalent to $\Gamma$;
		\item $A$ is piecewise analytic on $\Circle$;
		\item $A$ is an expansive covering map of degree greater than one;
		\item $A$ is Markov;
		\item no periodic break point of $A$ is asymmetrically hyperbolic.
	\end{enumerate}
    \smallskip
    
\item A continuous map $A:\Circle\to\Circle$ is said to be a \emph{virtually mateable map associated with a Fuchsian group $\Gamma$} if the following hold:
	\begin{enumerate}[leftmargin=6mm]
		\item $A$ is \emph{virtually orbit equivalent} to $\Gamma$; i.e., $A$ is a factor of a possibly discontinuous circle map $\widetilde{A}$ such that $\widetilde{A}$ is orbit equivalent to a finite index subgroup of $\Gamma$;
        \item $A$ is piecewise analytic on $\mathbb{S}^1$;
		\item $A$ is an expansive covering map of degree greater than one.
		\item $A$ is virtually Markov; 
		\item No periodic break point of $A$ is asymmetrically hyperbolic.
\end{enumerate}
\end{itemize}
\end{defn}

We briefly explain the terminology appearing in the definition.  A map is
\emph{piecewise analytic} if $\Circle$ can be decomposed into finitely many
closed arcs such that the restriction of the map to each arc extends
analytically to a neighbourhood of that arc.  The endpoints of the maximal such
arcs are called the \emph{break points} of the map.  The \emph{Markov}
condition requires that the break points belong to a finite forward--invariant
set, yielding a Markov partition of the circle. On the other hand, being virtually Markov means that there exists $n \in \N$ such that the  $n-$fold preimages of the maximal connected subsets of $\mathbb{S}^1$ on which $A$ is genuinely analytic give a Markov partition of $\mathbb{S}^1$ for $A$. The final condition excludes a
certain pathology at periodic break points, namely that the two one-sided
analytic germs have incompatible hyperbolic behavior. We refer to
\cite[\S2]{MM_survey} for the precise formulation.

A map $A:\Circle\to\Circle$ is called \emph{piecewise M\"obius}
if there are closed arcs $I_1,\ldots,I_k\subset\Circle$, with pairwise disjoint
interiors and union equal to $\Circle$, and elements
$g_1,\ldots,g_k\in\Aut(\D)$ such that
\[
A|_{I_j}=g_j|_{I_j},
\qquad
j=1,\ldots,k.
\]
If the group generated by these pieces,
$\Gamma_A:=\langle g_1,\ldots,g_k\rangle<\Aut(\D)$, is discrete, then $A$ is
called \emph{piecewise Fuchsian}.

\subsection{Factor Bowen--Series maps}
\label{SSec:FactorBowenSeriesMaps}

The main source of examples of mateable maps is provided by the factor
Bowen--Series construction of \cite{MM_survey,MM,MV25}; see also \cite[\S 13]{LLM24} for the formulation most relevant to the present paper.  We first recall the
ordinary Bowen--Series construction, and then explain how the factor
construction produces the group--side systems used later in
Sections~\ref{Sec:Combination} and~\ref{corr_sec}.

Let $\Gamma<\Aut(\D)$ be a finitely generated Fuchsian group with $\Lambda(\Gamma)=\mathbb{S}^1$,
and let $R\subset\D$ be a possibly ideal polygonal fundamental domain for
$\Gamma$. Denote the sides of $R$,
in cyclic order, by $s_1,\ldots,s_k$. Each side $s_i$ is paired with some side $s_j$ (where $i$ and $j$ may be the same) of $R$ by a unique element $h(s_i)\in\Gamma$. The set
of side-pairing transformations $h(s_1),\ldots,h(s_k)$ generate
$\Gamma$. For each side $s_i$, let $C(s_i)$ be the complete hyperbolic geodesic containing $s_i$. Let $P_i,Q_{i+1}\in\mathbb S^1$ be the endpoints of $C(s_i)$, with $P_i$ preceding $Q_{i+1}$ in the counter--clockwise order. Thus the points occur on $\mathbb S^1$ in the cyclic order $P_1,Q_1,P_2,Q_2,\ldots,P_k,Q_k$. Here, the indices are taken
modulo $k$.

\begin{defn}[Bowen--Series map]
	\label{Def:BowenSeriesMap}
	The \emph{Bowen--Series map} associated with the Fuchsian group
	$\Gamma$ and the chosen fundamental domain $R$ is the piecewise M\"obius
	map $A_{\Gamma,R}^{\mathrm{BS}}:\mathbb S^1\to\mathbb S^1$ defined by
	$A_{\Gamma,R}^{\mathrm{BS}}(x)=h(s_i)(x)$ for
	$x\in [P_i,P_{i+1})$, where $[P_i,P_{i+1})\subset\mathbb S^1$ denotes
	the counter-clockwise arc from $P_i$ to $P_{i+1}$. When the fundamental
	domain is fixed, we write simply $A_{\Gamma}^{\mathrm{BS}}$, or
	$A^{\mathrm{BS}}$.
\end{defn}

\noindent According to \cite{BowenSeries79}, the map $A_{\Gamma}^{\mathrm{BS}}$ is orbit equivalent to the Fuchsian group $\Gamma$, except possibly at finitely many points	modulo the action of $\Gamma$. We refer the reader to \cite{BowenSeries79} for more details.

Note that $A_{\Gamma,R}^{\mathrm{BS}}$ admits a canonical extension to a subset of the closed unit disk. Namely, if $D_i$ denotes the closed region bounded by the arc $[P_i,P_{i+1}]\subset\mathbb S^1$ and the hyperbolic geodesic connecting $P_i, P_{i+1}$, then the extension is defined as 
$$
\widehat A_{\Gamma,R}^{\mathrm{BS}}|_{D_i\setminus\{P_{i+1}\}}:=h(s_i)|_{D_i\setminus\{P_{i+1}\}}.
$$ 
This yields a piecewise M\"obius extension of $A_{\Gamma,R}^{\mathrm{BS}}$ to the domain $\mathcal D_{\Gamma,R}:=\bigcup_i D_i$.

Ordinary Bowen--Series maps are generally not continuous at the endpoints of
the partition arcs. However, it is continuous for suitable choices of fundamental domains for Fuchsian groups that uniformize punctured spheres with up to two order $2$ orbifold points. In order to allow for an orbifold point of arbitrary order (that puts Hecke groups in the mating framework), one designs a factoring construction that
removes such jump discontinuities. The idea is to pass first to a
cyclic cover on which the Bowen--Series polygon has a rotational symmetry, and
then to quotient the Bowen--Series map by this symmetry.  In the symmetric
model, the possible jumps occur at the endpoints of the rotational sectors; on
the quotient these endpoints are identified, and the descended boundary map is
continuous.

We now describe the orbifolds for which this construction
will be used.

\begin{defn}[Class of orbifolds for factor Bowen--Series maps]
	\label{Def:OrbifoldClassForFBS}
	Let $\mathscr S$ denote the collection of finite area hyperbolic surfaces/orbifolds $\Sigma$
	satisfying all of the following properties:
	\begin{enumerate}
		\item $\Sigma$ has genus zero;
		\item $\Sigma$ has at least one puncture;
		\item $\Sigma$ has at most two orbifold points of order $2$;
		\item $\Sigma$ has at most one orbifold point of order $\nu\ge3$.
	\end{enumerate}
Note that $\mathscr{S}$ contains all punctured spheres.
	For $\Sigma\in\mathscr S$, set
	\[
	n(\Sigma)=
	\begin{cases}
		\nu, & \text{if $\Sigma$ has an orbifold point of order $\nu\ge3$,}\\
		1, & \text{otherwise.}
	\end{cases}
	\]
	Let $\delta_1(\Sigma)$ be the number of punctures of $\Sigma$, and let
	$\delta_2(\Sigma)\in\{0,1,2\}$ be the number of order--two orbifold points.  Set
	\[
	p(\Sigma)=
	\begin{cases}
		2(\delta_1(\Sigma)-1), & \text{if } \delta_2(\Sigma)=0,\\
		2\delta_1(\Sigma)-1, & \text{if } \delta_2(\Sigma)=1,\\
        2\delta_1(\Sigma), & \text{if } \delta_2(\Sigma)=2.
	\end{cases}
	\]
	Finally, set
	\[
	m(\Sigma):=n(\Sigma)p(\Sigma),
	\qquad
	d(\Sigma):=m(\Sigma)-1.
	\]
\end{defn}

The meaning of these numbers is as follows.  The integer $n(\Sigma)$ is the
order of the cyclic symmetry used in the construction.  If $n(\Sigma)=1$, no
covering is needed and we set $\widetilde\Sigma=\Sigma$.  If $n(\Sigma)\ge3$,
the construction of \cite[\S13]{LLM24} first replaces $\Sigma$ by an
$n(\Sigma)-$fold cyclic cover
$\widetilde\Sigma\to\Sigma$.  Topologically, this cover is obtained by cutting
$\Sigma$ along a geodesic joining the order$-n(\Sigma)$ orbifold point to a cusp
and gluing $n(\Sigma)$ copies cyclically.

In either case, the covering orbifold $\widetilde\Sigma$ is uniformized by a
Fuchsian group admitting a preferred ideal polygonal fundamental domain with
$m(\Sigma)=n(\Sigma)p(\Sigma)$ sides.  The polygon is divided by the cyclic
symmetry into $n(\Sigma)$ congruent sectors, and each sector contains
$p(\Sigma)$ sides.  Thus $p(\Sigma)$ is the number of sides, equivalently ideal
vertices, that remain after passing to the quotient by the cyclic symmetry.  The
boundary degree of the resulting factor Bowen--Series map is
$d(\Sigma)=n(\Sigma)p(\Sigma)-1$ (see below). The hyperbolicity of $\Sigma$ implies that the preferred polygon on the cyclic cover has at least three sides,
so $n(\Sigma)p(\Sigma)\ge3$.

We now pass from the orbifold notation to the group notation that will be used
later.  Fix $\Sigma\in\mathscr S$, and let $\Gamma<\Aut(\D)$ be a Fuchsian
group uniformizing $\widetilde\Sigma$, so that
$\widetilde\Sigma\simeq\D/\Gamma$.  The symmetric Bowen--Series construction
can be carried out so that $\Gamma$ admits a closed ideal polygon
$\Pi\subset\overline\D$ as a fundamental domain, where $\Pi$ has
$m(\Sigma)=n(\Sigma)p(\Sigma)$ sides and is invariant under the rotation
\[
M_\omega(z)=\omega z,
\qquad
\omega=e^{2\pi i/n(\Sigma)}.
\]
Here \emph{symmetric} means that both the polygon and its side-pairing pattern
are preserved by the cyclic group $\langle M_\omega\rangle$.  Choose one
rotational sector and label the $p(\Sigma)$ sides of $\Pi$ in that sector by
$s_1,\ldots,s_{p(\Sigma)}$.  If $h_\ell\in\Gamma$ pairs the side $s_\ell$, then
the rotated side $M_\omega^j(s_\ell)$ is paired by
$M_\omega^j\circ h_\ell\circ M_\omega^{-j}$, for
$j=0,\ldots,n(\Sigma)-1$.  Equivalently, the Bowen--Series map
$A^{\mathrm{BS}}_\Gamma:
\overline{\D}\setminus \operatorname{Int}\Pi
\longrightarrow
\overline{\D}$
commutes with $M_\omega$, and $M_\omega$ normalizes $\Gamma$, i.e.
$M_\omega\Gamma M_\omega^{-1}=\Gamma$.  Hence
\[
G:=\langle \Gamma,M_\omega\rangle
\]
is a Fuchsian group containing $\Gamma$ as a normal subgroup of index
$n(\Sigma)$.  The group $\Gamma$ uniformizes the cyclic cover
$\widetilde\Sigma$, while $G$ uniformizes the original orbifold $\Sigma$.  A
closed sector
\[
\Pi_G:=\Pi\cap \{0\leq \arg z\leq 2\pi/n(\Sigma)\}
\]
is a fundamental domain for the action of $G$.  When $n(\Sigma)=1$, no
extension is needed; in this case $G=\Gamma$ and $\Pi_G=\Pi$. Identifying points in the same $\langle M_\omega\rangle-$orbit gives a quotient
map
$
q_\omega:\overline{\D}\longrightarrow\overline{\D},
$
which in standard coordinates is $q_\omega(z)=z^{n(\Sigma)}$.  If
$n(\Sigma)=1$, then $q_\omega=\Id$.

\begin{figure}[ht]
	\captionsetup{width=0.98\linewidth}
	\begin{tikzpicture}
		\node[anchor=south west,inner sep=0] at (0.5,0) {\includegraphics[width=0.4\linewidth]{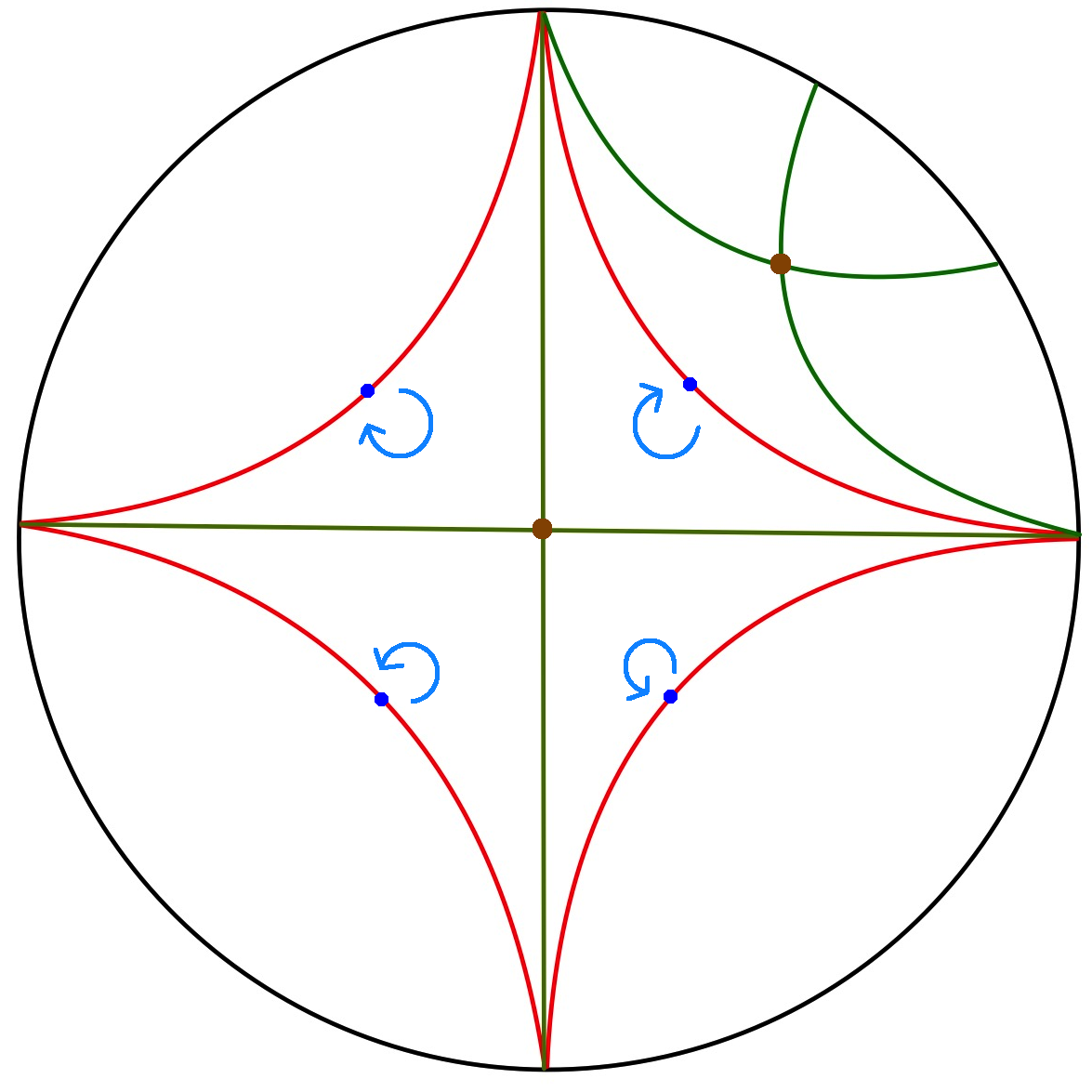}};
		\node[anchor=south west,inner sep=0] at (7,0) {\includegraphics[width=0.4\linewidth]{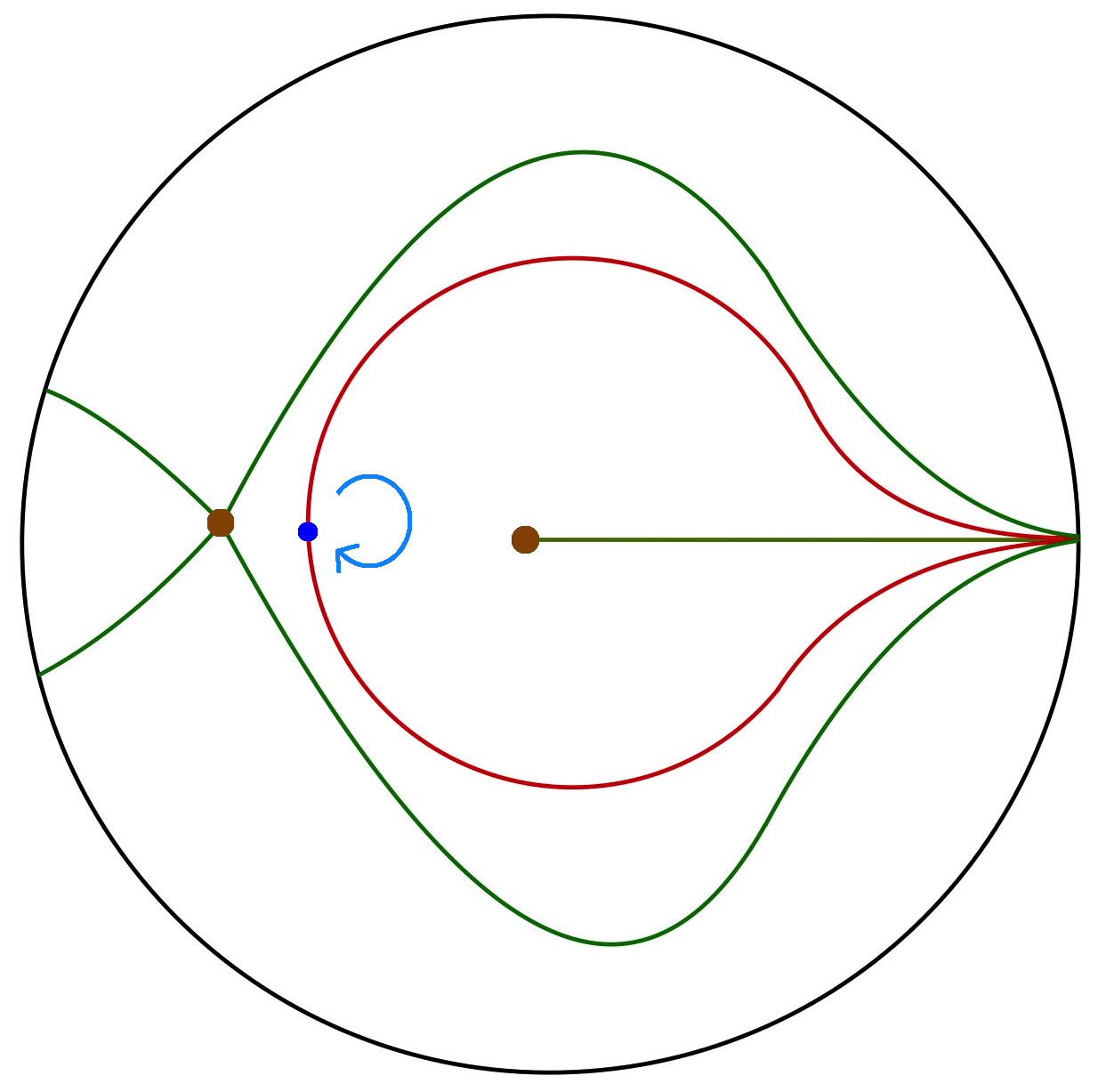}};
		\node at (4.2,3.25) {\begin{small}$h_{1}$\end{small}};
		\node at (2.7,3.25) {\begin{small}$h_{2}$\end{small}};
		\node at (2.7,2.6) {\begin{small}$h_{3}$\end{small}};
		\node at (4.2,2.6) {\begin{small}$h_{4}$\end{small}};
		\node at (3.7,4) {$\Pi$};
		\node at (10,2.2) {$T_G$};
		\node[fill=white,inner sep=1pt] at (4.55,4.2) {$\widetilde c$};
		\node[fill=white,inner sep=1pt] at (3.22,2.72) {$0$};
		\node[fill=white,inner sep=1pt] at (8.15,3.35) {$c$};
		\node[fill=white,inner sep=1pt] at (9.8,3.2) {$v$};
	\end{tikzpicture}
	\caption{Construction of a factor Bowen--Series map in the fully ramified
		Hecke example with $n=4$ and $p=1$.  Left: the preferred fundamental domain
		$\Pi$ for the covering orbifold $\widetilde\Sigma$ and the associated
		Bowen--Series map $A^{\mathrm{BS}}_\Gamma$.  The four sides of $\Pi$ are
		paired with themselves by the transformations $h_1,\ldots,h_4$, obtained from one another by
		rotation through powers of $M_i(z)=iz$.  The blue arrows indicate the
		corresponding Bowen--Series branches.  The straight green cross is the union
		$P=\bigcup_{j=0}^3\{r i^j:0\le r\le1\}$ of radial lines.  In the branch region where the Bowen--Series
		map is $h_1$, the curved green arcs shown form the preimage
		$h_1^{-1}(P)$; these arcs meet at
		$\widetilde c=h_1^{-1}(0)$.  Right: after quotienting by
		$\langle M_i\rangle$, the four rotational sectors are identified and the
		factor tile $T_G=q_i(\operatorname{Int}\Pi)$, 
		where $q_i(z)=z^4$, is an ideal monogon.  The
		factor Bowen--Series map
		$F_G:\overline{\D}\setminus T_G\to\overline{\D}$ satisfies
		$F_G\circ q_i=q_i\circ A^{\mathrm{BS}}_\Gamma$.  The point
		$c=q_i(\widetilde c)$ is the unique critical point of $F_G$; its
		multiplicity is three.  The point $v=q_i(0)=0$ is the corresponding
		critical value.}
	\label{hecke_factor_bs_fig}
\end{figure}
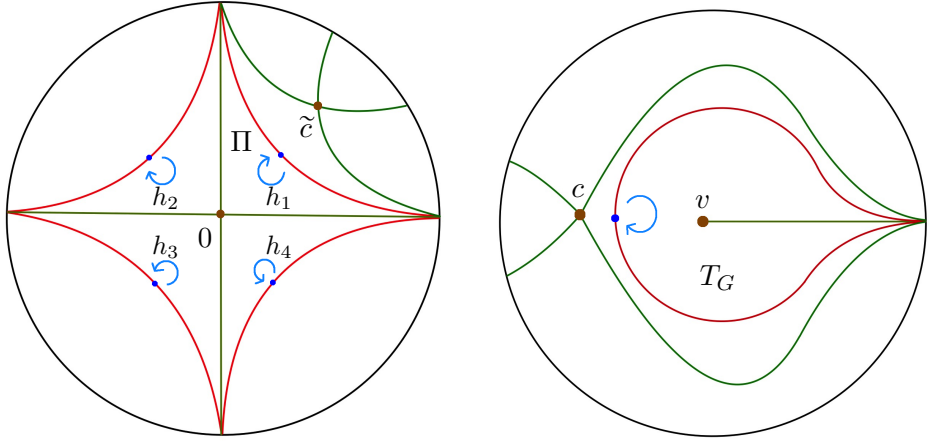

\begin{defn}[Factor Bowen--Series map]
	\label{Def:FactorBowenSeries}
	Let $G=\langle \Gamma,M_\omega\rangle$ be obtained from an orbifold
	$\Sigma\in\mathscr S$ by the construction above. Let
	$
	T_G:=q_\omega(\operatorname{Int}\Pi).
	$
	The \emph{factor Bowen--Series map} associated with $G$ is the unique map
	$
	F_G:\overline{\D}\setminus T_G\longrightarrow\overline{\D}
	$
	such that
	\[
	F_G\circ q_\omega
	=
	q_\omega\circ A^{\mathrm{BS}}_\Gamma
	\quad
	\text{on }
	\overline{\D}\setminus \operatorname{Int}\Pi.
	\]
	We call $T_G$ the \emph{factor tile}.  (See Figure~\ref{hecke_factor_bs_fig}.)
\end{defn}

The following theorem collects the properties of factor Bowen--Series maps that
will be used later.

\begin{thm}[Factor Bowen--Series maps;
	{\cite[Proposition~2.5]{MM},
		\cite[Theorem~2.6]{MM_survey},
		\cite[\S13]{LLM24}}]
	\label{Thm:FBSProperties}
	Let $G=\langle \Gamma,M_\omega\rangle$ be obtained from an orbifold
	$\Sigma\in\mathscr S$ by the construction above, and let
	$F_G:\overline{\D}\setminus T_G\to\overline{\D}$ be the associated factor
	Bowen--Series map.  Set $n=n(\Sigma)$, $p=p(\Sigma)$, and $d=np-1$.
	Then the following properties hold.
	\begin{enumerate}
		\item The map $F_G$ is holomorphic on the interior of its domain and
		extends continuously to each boundary arc of $\partial T_G$.  Its boundary
		map $F_G|_{\Circle}:\Circle\to\Circle$ is a virtually mateable circle map (in fact, a mateable circle map if $n=1$).  In
		particular, it is a continuous piecewise analytic, orientation--preserving, expansive covering map of degree $d$.  Hence, it is topologically
		conjugate to $z\mapsto z^d$ on $\Circle$.
		
		\item The boundary map $F_G|_{\Circle}$ is virtually orbit equivalent to the
		Fuchsian group $G$. 
		
		\item The factor tile $T_G\subset\D$ is simply connected, and
		the closure $\overline{T_G}=q_\omega(\Pi)$ in $\ovl{\disk}$ is a closed ideal
		polygonal disk.  Its singular boundary points are precisely
		$\Ss_G:=\partial T_G\cap\Circle$, and $\#\Ss_G=p$.
		
		\item The restriction $F_G \colon \partial T_G\to\partial T_G$ is an
		orientation-reversing involution.  In particular, $F_G$
		preserves $\Ss_G$, and each point of $\Ss_G$ is a parabolic point of
		period one or two for the boundary dynamics.
		
		\item The quotient map $q_\omega$ satisfies
		$q_\omega^{-1}(T_G)=\operatorname{Int}\Pi$, and the semiconjugacy relation
		$F_G\circ q_\omega=q_\omega\circ A^{\mathrm{BS}}_\Gamma$ holds on
		$\overline{\D}\setminus\operatorname{Int}\Pi$.
 		\item If $n=1$, no critical points are created by the quotient.  If
		$n\ge3$, then $F_G$ has $p$ critical points in the interior of its domain,
		each of local degree $n$, equivalently of multiplicity $n-1$. All these critical points
		are mapped to the same critical value $q_\omega(0)=0$. 
	\end{enumerate}
\end{thm}

\begin{rem}
	The notation for these maps varies in the literature.  In \cite{MM}, the
	factor Bowen--Series map is denoted $A^{\mathrm{fBS}}_\Gamma$, while in
	\cite[\S13]{LLM24} it is denoted $A^f_\Sigma$.  We write $F_G$ because in
	this paper it plays the role of the external map on the group side.
\end{rem}

We collect the group--side data into notation that will be used throughout the
paper.

\begin{defn}[The class $\mathscr G_{\attr}$]
	\label{Def:Gattr}
	Let $\mathscr G_{\attr}$ be the set of Fuchsian groups associated to the orbifolds in $\mathscr S$ and hence admitting the factor Bowen--Series construction described above. Thus an element $G\in\mathscr G_{\attr}$ consists of a Fuchsian group
	$G=\langle\Gamma,M_\omega\rangle$, together with the quotient data and factor
	tile $T_G$ used to construct its factor Bowen--Series map
	$F_G:\overline{\D}\setminus T_G\to\overline{\D}$.
	
	If $G$ is obtained from the orbifold $\Sigma_G\in\mathscr S$, we set
	\[
	n(G):=n(\Sigma_G),
	\qquad
	p(G):=p(\Sigma_G),
	\qquad
	d(G):=\deg(F_G|_{\Circle})=n(G)p(G)-1.
	\]
	For $d\ge2$, let
	\[
	\mathscr G_{\attr}(d)
	=
	\{G\in\mathscr G_{\attr}:d(G)=d\}.
	\]
\end{defn}

\subsection{Hecke groups and Farey maps}
\label{SSec:HeckeFarey}

The parabolic construction uses a distinguished one-parameter subfamily of
factor Bowen--Series maps, namely the Farey maps associated with Hecke groups.
The relevant feature of this subfamily is that the factor tile has a single
ideal vertex, and hence the boundary map has a unique parabolic fixed point.

Among the orbifolds in $\mathscr S$, the distinguished family is given by the genus zero orbifolds with one puncture, one order 2 orbifold point, and one orbifold point of
order $q\ge3$.  Their uniformizing groups are the classical Hecke groups,
generated in the upper half-plane by
\[
z \mapsto -\frac1z,
\qquad
z \mapsto z+2\cos\!\left(\frac{\pi}{q}\right).
\]
To match the degree of the external map, we write $H_d$ for the Hecke group with $q = d+1$. In our notation, $n(H_d)=d+1$, $p(H_d)=1$, and $d(H_d)=d$.

We denote the associated factor Bowen--Series map, called the \emph{Farey map},
by
\[
F_d:=F_{H_d}:\overline{\D}\setminus T_d\to\overline{\D},
\]
where $T_d = T_{H_d}$, and we write $\Ss_d:=\partial T_d\cap\Circle$.  With the normalization above, $\Ss_d=\{1\}$.

The properties of Farey maps that will be used later are summarized below.

\begin{prop}[Farey maps; {\cite[\S2.4]{MM}, \cite[\S4.2.3, \S5.1.2]{BLLM}}]
	\label{Prop:FareyProperties}
	The degree-$d$ Farey map $F_d$ satisfies the following properties.
	\begin{enumerate}
		\item The boundary map $F_d|_{\Circle}:\Circle\to\Circle$ is a virtually mateable
		circle map of degree $d$ in the sense of Definition~\ref{Def:MateableCircleMap}.  
        
		\item The point $1\in\Circle$ is the unique neutral periodic point of
		$F_d|_{\Circle}$.  It is fixed, and each of the two local branches of $F_d$ at $1$ extends analytically in a neighborhood of $1$.

		\item All periodic points of $F_d|_{\Circle}$ other than $1$ are repelling.
	\end{enumerate}
\end{prop}

We collect the parabolic group--side data in notation parallel to
Definition~\ref{Def:Gattr}.

\begin{defn}[The class $\mathscr G_{\parab}$]
	\label{Def:Gparab}
	Let $\mathscr G_{\parab}$ be the set of the Hecke groups
	$H_d$, $d\ge2$; to each of these groups we associate its standard Farey map.  For
	$G=H_d\in\mathscr G_{\parab}$, we set
	\[
	F_G:=F_{H_d}:=F_d,
	\qquad
	T_G:=T_{H_d}:=T_d,
	\qquad
	\Ss_G:=\Ss_{H_d}:=\Ss_d=\{1\}.
	\]
	For $d\ge2$, define
	\[
	\mathscr G_{\parab}(d)
	=
	\{G\in\mathscr G_{\parab}:d(G)=d\}
	=
	\{H_d\}.
	\]
\end{defn}

\subsection{Compatibility with basin boundaries' maps}
\label{SSec:BoundaryCompatibility}

The combination construction in Section~\ref{Sec:Combination} requires a
boundary conjugacy between the dynamics on the distinguished Fatou component of
the entire map and the corresponding group--side external map.

On the entire map side, the dynamics in the immediate basin of degree $d$ is
represented, after choosing a Riemann map, by a degree-$d$ Blaschke product.  On
the group side, the dynamics is represented by a factor Bowen--Series map: in
the attracting case by a map $F_G$ with $G\in\mathscr G_{\attr}(d)$, and in the
parabolic case by the Farey map $F_d$ with $H_d\in\mathscr G_{\parab}(d)$.

The required compatibility statements are different in the two settings.  In the
attracting case, the boundary conjugacy has a David extension to the disk (see \cite{David} for background on David homeomorphisms).  In
the parabolic case, it has a quasiconformal extension which is equivariant only
in a pinched neighbourhood of the circle.

\begin{thm}[Attracting compatibility;
	{\cite[Theorem~4.9]{LMMN}, \cite[Lemma~3.4]{MM}}]
	\label{Thm:AttractingCompatibility}
	Let $G\in\mathscr G_{\attr}(d)$, and let
	$F_G:\overline{\D}\setminus T_G\to\overline{\D}$ be the associated factor
	Bowen--Series map.  Let $B:\D\to\D$ be a Blaschke product whose
	restriction to $\Circle$ is an expanding degree$-d$ covering map.  Then there
	exists an orientation-preserving homeomorphism $\psi:\Circle\to\Circle$ such
	that $F_G\circ\psi=\psi\circ B$ on $\Circle$.  Moreover, $\psi$ extends continuously to $\D$ as a David homeomorphism.
\end{thm}

We shall use the following terminology. A one-sided pinched neighbourhood of
$\Circle$ is a collar of $\Circle$ inside $\overline\D$, allowed to pinch at
finitely many prescribed boundary points. In particular, it is a subset of
$\overline\D$.

\begin{thm}[Parabolic compatibility, {\cite[Lemma~5.2]{BLLM}}]
	\label{Thm:ParabolicCompatibility}
	Let $B:\D\to\D$ be a Blaschke product of degree $d\ge2$ such that $1\in\Circle$ is a triple fixed point of $B$. 
	Let $F_d:\overline\D\setminus T_d\to\overline\D$ be the degree$-d$ Farey map,
	normalized so that its unique parabolic fixed point on $\Circle$ is $1$. Set $A_B:=(B|_{\Circle})^{-1}(1)$ and $A_{F_d}:=(F_d|_{\Circle})^{-1}(1)$. There exists a homeomorphism
	$\psi:\overline\D\to\overline\D$, quasiconformal on $\D$, with $\psi(1)=1$,
	such that $\psi|_{\Circle}$ conjugates $B|_{\Circle}$ to $F_d|_{\Circle}$.
	In particular, $\psi(A_B)=A_{F_d}$.
	
	Moreover, there exist closed one-sided pinched neighbourhoods
	$X_B^-\subset X_B^+\subset\overline\D$ and
	$X_{F_d}^-\subset X_{F_d}^+\subset\overline\D$ such that
\begin{itemize}
    \item $X_B^-, X_{F_d}^-$ are pinched exactly at	the points of $A_B, A_{F_d}$, respectively;
    \item $X_B^+$ and $X_{F_d}^+$ are pinched only at $1$;
    \item  $X_{F_d}^-\subset\overline\D\setminus
	T_d$;
    \item $\psi(X_B^\pm)=X_{F_d}^\pm$; and
    \item $B:X_B^-\to X_B^+$,
	$F_d:X_{F_d}^-\to X_{F_d}^+$ satisfy $F_d\circ\psi=\psi\circ B$ on $X_B^-$.
    \end{itemize}
\end{thm}

\begin{proof}
	Let 
    \[
    B_d(z):=\frac{z^d+\frac{d-1}{d+1}}{1+\frac{d-1}{d+1}z^d}
    \]
    be the standard unicritical parabolic Blaschke product of degree $d$. It has a triple parabolic fixed point at $1$,
	and $(B_d|_{\Circle})^{-1}(1)$ is the set of $d-$th roots of unity.
	
	We first use the standard compatibility between $F_d$ and $B_d$. Since $F_d$ is
	Farey-like, \cite[Lemma~5.2]{BLLM} gives a homeomorphism
	$\mathfrak g:\overline\D\to\overline\D$, quasiconformal on $\D$, with
	$\mathfrak g(1)=1$, which conjugates $F_d$ to $B_d$ on suitable one-sided
	pinched neighbourhoods of $\Circle$. Equivalently, for
	$\psi:=\mathfrak g^{-1}$, there are closed one-sided pinched
	neighbourhoods $X_{B_d}^-\subset X_{B_d}^+$ and $X_{F_d}^-\subset X_{F_d}^+$, where $X_{B_d}^-$ is
	pinched at $A_B$, $X_{F_d}^-$ is pinched at $A_{F_d}$, and
	$X_{B_d}^+$ and $X_{F_d}^+$ are pinched only at $1$, such that
	$\psi(X_{B_d}^\pm)=X_{F_d}^\pm$ and $F_d\circ\psi=\psi\circ B_d$ on $X_{B_d}^-$.

    It remains to replace $B_d$ by the given Blaschke product $B$. The existence of such a quasiconformal conjugacy between $B_d$ and $B$ was proved in \cite[Proposition~6.8]{McM88}. 
    \end{proof}

\subsection{Local dynamics near parabolic points}
Consider a factor Bowen--Series map $F:\cD=:\overline{\D}\,\sm T\to\overline{\D}$. 

\subsubsection{The case of a fixed point}\label{fbs_para_fixed_local_dyn_subsec} 
Let $x$ be an ideal boundary point of $T$ with $F(x)=x$. We refer the reader to Figure~\ref{fbs_local_fixed_fig} for an illustration of the following analysis. 

Choose points $b_0, c_0\in\partial T$ sufficiently close to $x$ such that $F(b_0)=c_0$. Note that this implies $F(c_0)=b_0$, and $F:\arc{x,b_0}\to\arc{x,c_0}$ is an orientation-reversing involution, where $\arc{x,b_0}\subset\partial T$ (respectively, $\arc{x,c_0}\subset\partial T$) is a simple arc connecting $x, b_0$ (respectively, $x,c_0$). Connect $b_0, c_0$ by a simple arc $\arc{b_0,c_0}\subset \overline{T}$. Denote by $Y$ the topological triangle in $\overline{\D}$ whose sides are given by $\arc{b_0,c_0}$, $\arc{x,b_0}$, and $\arc{x,c_0}$.
By construction, $Y\subset\overline{T}$. 
\begin{figure}[h!]
\captionsetup{width=0.98\linewidth}
\begin{tikzpicture}
\node[anchor=south west,inner sep=0] at (0,0) {\includegraphics[width=0.33\textwidth,  trim=15 20 15 20, clip]{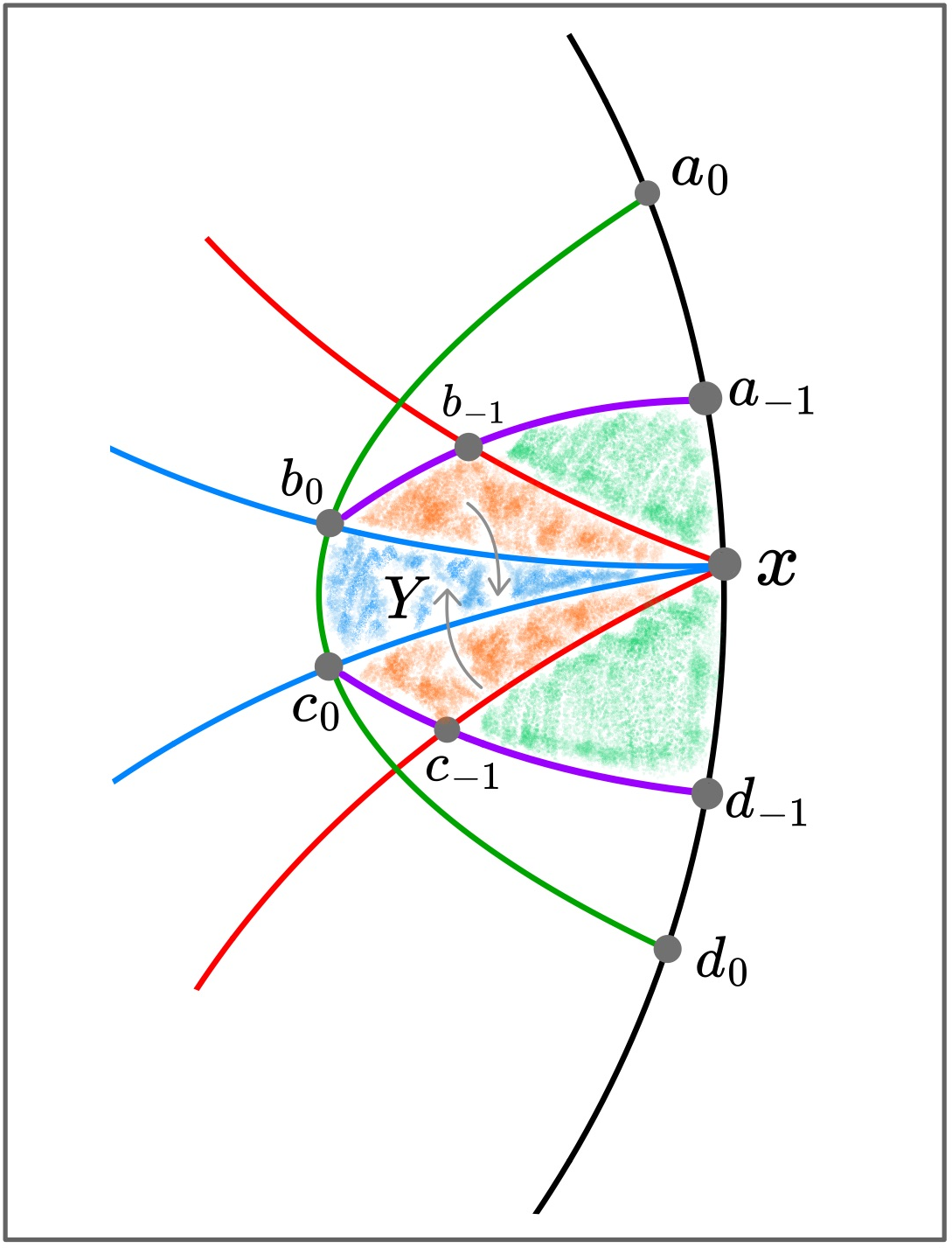}};
\end{tikzpicture}
\caption{The local dynamics of a factor Bowen--Series map near a parabolic fixed point is depicted.}
\label{fbs_local_fixed_fig}
\end{figure}

The parabolic dynamics of $F$ at $x$ implies that there are inverse orbits
$$
\cdots\xrightarrow{F} b_{-n}\xrightarrow{F} b_{-(n-1)}\xrightarrow{F}\cdots \xrightarrow{F} b_{-1}\xrightarrow{F} b_0
$$
and
$$
\cdots\xrightarrow{F} c_{-n}\xrightarrow{F} c_{-(n-1)}\xrightarrow{F}\cdots \xrightarrow{F} c_{-1}\xrightarrow{F} c_0
$$
with $\displaystyle\lim_{n\to+\infty} b_{-n}=\displaystyle\lim_{n\to+\infty} c_{-n}=x$.

Now pick points $a_0, d_0\in\mathbb{S}^1$ close to $x$ and on different sides of $x$, and connect $a_0, b_0$ (respectively, $c_0, d_0$) by a simple arc $\arc{a_0,b_0}\subset\overline{\D}\,\sm T$ (respectively, $\arc{c_0,d_0}\subset\overline{\D}\,\sm T$) such that the union of $\arc{a_0,b_0}$, $\arc{b_0,c_0}$, and $\arc{c_0,d_0}$ is a simple arc $\gamma$ connecting $a_0, d_0$ in $\overline{\D}$. Since $F$ is locally injective on $\mathbb{S}^1$, there exist unique points $a_{-1}, d_{-1}$ near $x$ such that $F(a_{-1})=a_0$ and $F(d_{-1})=d_0$. Further, there are simple arcs $\arc{b_0,a_{-1}}$ and $\arc{c_0,d_{-1}}$ that are mapped homeomorphically onto $\arc{a_0,b_0}\cup\arc{b_0,c_0}$ and $\arc{d_0,c_0}\cup\arc{c_0,b_0}$ (respectively) under $F$. Since $F$ has repelling directions on $\mathbb{S}^1$ at the parabolic fixed point $x$, we can assume that the arcs $\arc{b_0,a_{-1}}$ and $\arc{c_0,d_{-1}}$ are contained in the connected component $U$ of $\D\setminus\gamma$ having $x$ on its boundary.
The following properties are consequences of the above discussion. 
\begin{enumerate}[leftmargin=8mm]
    \item Each of the arcs $\arc{b_0,b_{-1}}$($\subset\arc{b_0,a_{-1}}$) and $\arc{c_0,c_{-1}}$ ($\subset\arc{c_0,d_{-1}}$) maps homeomorphically onto $\arc{b_0,c_0}$ under $F$.
    \item There are simple arcs $\arc{x,b_{-1}}$, $\arc{x,c_{-1}}$ in $U$ that map homeomorphically onto $\arc{x,b_0}$, $\arc{x,c_0}$ (respectively) under $F$.
    \item Each of the topological triangles $\Delta (x,b_0,b_{-1})$ and $\Delta (x,c_0,c_{-1})$ maps univalently onto the triangle $Y=\Delta (x,b_0,c_0)$ under $F$.
    \item The union of the topological triangles $\Delta (x,b_{-1},a_{-1})$ and $\Delta (x,c_{-1},d_{-1})$ maps univalently onto $\overline{U\setminus Y}$ under $F$.
\end{enumerate}

\subsubsection{The case of a two-cycle}\label{fbs_para_twocycle_local_dyn_subsec}
Now let $\{x,F(x)\}$ be a parabolic $2-$cycle of $F$, where $x, F(x)$ are ideal vertices of $T$. We record the covering properties of $F$ near this $2-$cycle for future reference. See Figure~\ref{fbs_local_per_two_fig} for an illustration.

We choose points $b_0, c_0\in\partial T$ close to $x$ and $b_0', c_0'\in\partial T$ close to $F(x)$ such that $F(b_0)=c_0'$ and $F(c_0)=b_0'$. Hence, $F(c_0')=b_0$ and $F(b_0')=c_0$, and $F:\arc{b_0,x}\cup\arc{x,c_0}\to\arc{c_0',F(x)}\cup\arc{F(x),b_0'}$ is an orientation-reversing involution, where $\arc{b_0,x}, \arc{x,c_0}, \arc{c_0',F(x)}, \arc{F(x),b_0'}$ are simple arcs contained in $\partial T$.
Connect $b_0, c_0$ (respectively, $b_0', c_0'$) by a simple arc $\arc{b_0,c_0}\subset \overline{T}$ (respectively, $\arc{b_0',c_0'}\subset\overline{T}$), defining the triangles $Y:=\Delta (x, b_0, c_0)\subset\overline{T}$ and $Y':=\Delta (F(x),b_0',c_0')\subset\overline{T}$.
\begin{figure}[h!]
\captionsetup{width=0.98\linewidth}
\begin{tikzpicture}
\node[anchor=south west,inner sep=0] at (0,0) {\includegraphics[width=0.5\textwidth, trim=15 15 15 15, clip]{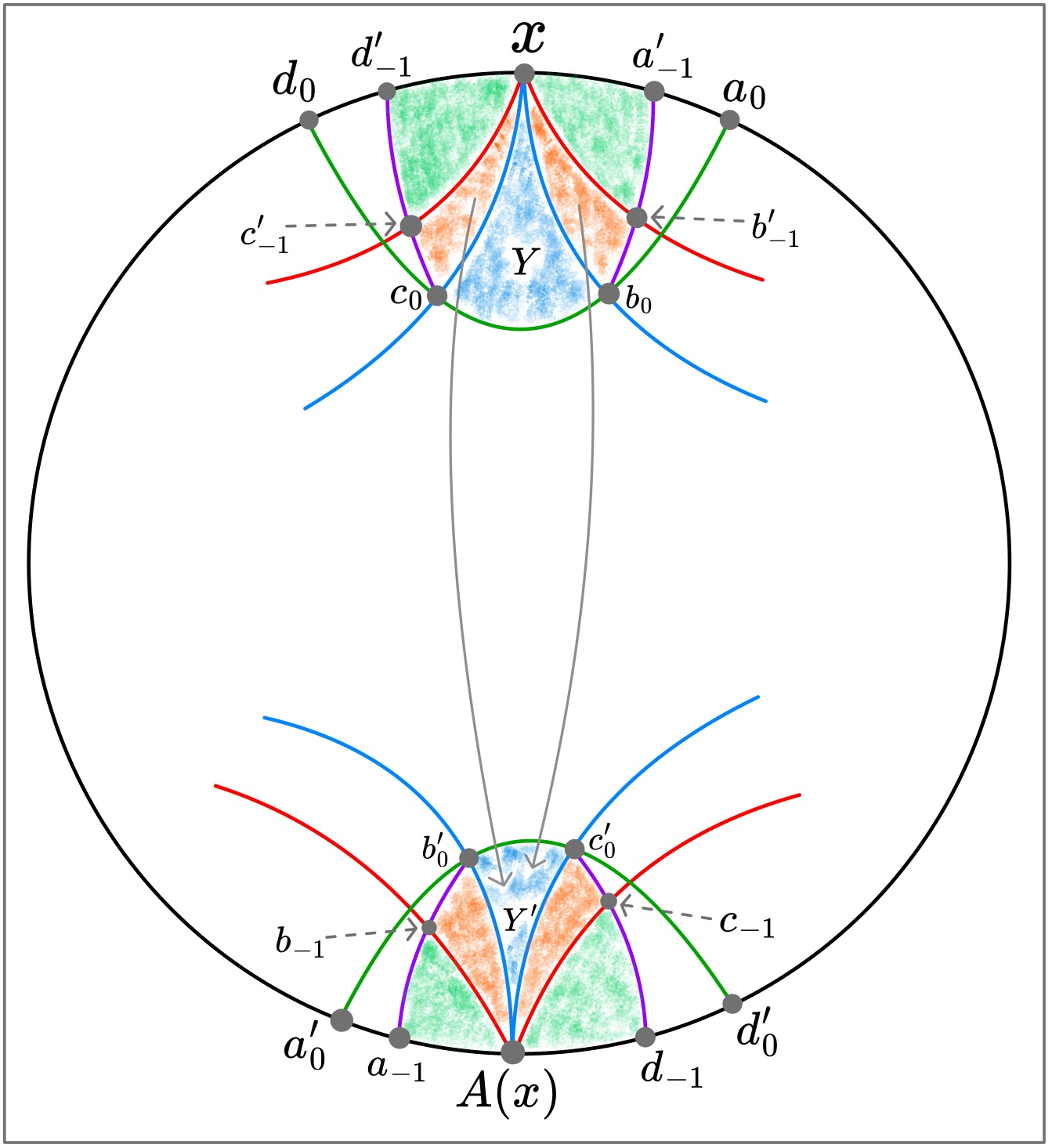}};
\end{tikzpicture}
\caption{The local dynamics of a factor Bowen--Series map near a parabolic $2-$cycle is depicted.}
\label{fbs_local_per_two_fig}
\end{figure}

As in the previous case, pick points $a_0, d_0\in\mathbb{S}^1$ close to $x$ and on different sides of $x$ (respectively, $a_0', d_0'\in\mathbb{S}^1$ close to $F(x)$ and on different sides of $F(x)$). Then connect $a_0, b_0$ (respectively, $c_0, d_0$) by a simple arc $\arc{a_0,b_0}\subset\overline{\D}\,\sm T$ (respectively, $\arc{c_0,d_0}\subset\overline{\D}\,\sm T$) such that the union of $\arc{a_0,b_0}$, $\arc{b_0,c_0}$, and $\arc{c_0,d_0}$ is a simple arc $\gamma$ connecting $a_0, d_0$ in $\overline{\D}$. Let us also perform the corresponding construction for the objects with $'$, and call the resulting simple arc $\gamma'$.
Local injectivity of $F$ on $\mathbb{S}^1$ provides us with unique points $a_{-1}, d_{-1}$ near $F(x)$ and $a_{-1}', d_{-1}'$ near $x$ such that $F(a_{-1})=a_0$, $F(d_{-1})=d_0$, $F(a_{-1}')=a_0'$, and $F(d_{-1}')=d_0'$. 

As before, there are simple arcs $\arc{b_0,a_{-1}'}$ and $\arc{c_0,d_{-1}'}$ that are mapped homeomorphically onto $\arc{a_0',b_0'}\cup\arc{b_0',c_0'}$ and $\arc{d_0',c_0'}\cup\arc{c_0',b_0'}$, respectively, under $F$. 
Similarly, there are simple arcs $\arc{b_0',a_{-1}}$ and $\arc{c_0',d_{-1}}$ that map homeomorphically onto $\arc{a_0,b_0}\cup\arc{b_0,c_0}$ and $\arc{d_0,c_0}\cup\arc{c_0,b_0}$, respectively, under $F$. 
We can assume that the arcs $\arc{b_0,a_{-1}'}$ and $\arc{c_0,d_{-1}'}$ (respectively, the arcs $\arc{b_0',a_{-1}}$ and $\arc{c_0',d_{-1}}$) are contained in the connected component $U$ of $\D\setminus\gamma$ having $x$ on its boundary (respectively, the connected component $U'$ of $\D\setminus\gamma'$ having $F(x)$ on its boundary). We also denote the $F-$preimages of $b_0', c_0'$ on the arcs $\arc{b_0,a_{-1}'}$ and $\arc{c_0,d_{-1}'}$ by $b_{-1}',c_{-1}'$. Similarly, we denote the $F-$preimages of $b_0, c_0$ on the arcs $\arc{b_0',a_{-1}}$ and $\arc{c_0',d_{-1}}$ by $b_{-1},c_{-1}$.

With notation as above, we have the following mapping properties of $F$. 
\begin{enumerate}[leftmargin=8mm]
    \item Each of the arcs $\arc{b_0,b_{-1}'}$ and $\arc{c_0,c_{-1}'}$ maps homeomorphically onto $\arc{b_0',c_0'}$ under $F$. Analogously, each of the arcs $\arc{b_0',b_{-1}}$ and $\arc{c_0',c_{-1}}$ maps homeomorphically onto $\arc{b_0,c_0}$ under $F$.
    \item There are simple arcs $\arc{x,b_{-1}'}$, $\arc{x,c_{-1}'}$ in $U$ that map homeomorphically onto $\arc{F(x),b_0'}$, $\arc{F(x),c_0'}$, respectively, under $F$. Similarly, there are simple arcs $\arc{F(x),b_{-1}}$, $\arc{F(x),c_{-1}}$ in $U'$ that map homeomorphically onto $\arc{x,b_0}$, $\arc{x,c_0}$, respectively, under $F$
    \item Each of the topological triangles $\Delta (x,b_0,b_{-1}')$ and $\Delta (x,c_0,c_{-1}')$ maps univalently onto the triangle $Y'=\Delta (x,b_0',c_0')$ under $F$.
    Similarly, each of the topological triangles $\Delta (F(x),b_0',b_{-1})$ and $\Delta (F(x),c_0',c_{-1})$ maps univalently onto the triangle $Y=\Delta (x,b_0,c_0)$ under $F$.
    \item The union of the topological triangles $\Delta (x,b_{-1}',a_{-1}')$ and $\Delta (x,c_{-1}',d_{-1}')$ maps univalently onto $\overline{U'\setminus Y'}$ under $F$. Similarly, the union of the topological triangles $\Delta (F(x),b_{-1},a_{-1})$ and $\Delta (F(x),c_{-1},d_{-1})$ maps univalently onto $\overline{U\setminus Y}$ under $F$.
\end{enumerate}

\section{Combination theorems for maps and groups}
\label{Sec:Combination}

We now combine the two inputs introduced in Sections~\ref{Sec:mateable_maps} and~\ref{Sec:MateableGroups}. On the holomorphic side we take $f\in\Kstar$, $\star\in\{\parab,\attr\}$, with $U$ an invariant immediate basin of the distinguished parabolic, respectively attracting, fixed point. Here $\mathscr K_\parab$ is the class of Definition~\ref{Def:Kp}, while $\mathscr K_\attr$ is the class of Definition~\ref{Def:Kattr}. By Theorems~\ref{thm:bounded_geomfinite} and~\ref{thm:bounded_hyperbolic}, the basin $U$ is a bounded Jordan domain in the parabolic case and a bounded quasidisk in the attracting case; in particular, $U$ is simply connected and every Riemann map $U\to\disk$ extends homeomorphically to the boundary. Recall that we write $\Kstar(d)$ for the marked maps for which $\deg(f \colon U\to U)=d\ge2$.

On the group side, we use the factor Bowen--Series maps defined in Section~\ref{Sec:MateableGroups}. For $\star=\attr$, the class $\mathscr G_\attr(d)$ is defined in Definition~\ref{Def:Gattr}, and Theorem~\ref{Thm:FBSProperties} gives the associated map $F_G \colon \ovl\disk\sm T_G\to\ovl\disk$, whose boundary restriction is a degree-$d$ mateable circle map. For $\star=\parab$, Definition~\ref{Def:Gparab} gives $\mathscr G_\parab(d)=\{H_d\}$, and the corresponding map is the degree-$d$ Farey map $F_d:\ovl\disk\sm T_d\to\ovl\disk$ of Proposition~\ref{Prop:FareyProperties}. When no confusion is possible, we write $T:=T_G$ and $\Ss:=\partial T_G\cap\Circle$. Thus throughout the section the degrees are matched: $\deg(f \colon U\to U)=\deg(F_G|_\Circle)=d$.

The aim is to replace the dynamics of $f$ inside the basin $U$ by the group--side dynamics of $F_G$, while keeping the dynamics of $f$ conformally unchanged away from the grand orbit of $U$. The boundary matching needed for this replacement is provided by the compatibility results of Subsection~\ref{SSec:BoundaryCompatibility}: Theorem~\ref{Thm:ParabolicCompatibility} in the parabolic case and Theorem~\ref{Thm:AttractingCompatibility} in the attracting case. The resulting surgery is quasiconformal for $\star=\parab$ and David for $\star=\attr$.

The main theorem of this section is the following combination result. Before stating it, recall that two partially defined holomorphic maps $h_1 \colon Y_1\sm X_1 \to Y_1$, $h_2 \colon Y_2\sm X_2 \to Y_2$, with $X_i \subset Y_i \subseteq \C$ being open domains, are called \emph{conformally conjugate} if there exists a homeomorphism $\phi \colon Y_1\to Y_2$ such that $\phi(X_1) = X_2$, $\phi \colon Y_1\sm \ovl X_1 \to Y_2\sm \ovl X_2$ is conformal, and 
\[
h_2\circ\phi(z) = \phi \circ h_1(z) \quad\text{ for all }\quad z\in Y_1\sm X_1. 
\]

\begin{thm}[Conformal combination of maps and groups]
\label{Thm:mating}
Fix $\star\in\{\parab,\attr\}$ and $d\ge2$. For every pair $f \in \Kstar(d)$ and $G \in \Gstar(d)$ there exist a pair of Jordan disks $\tT\subset \tilde U \subset \C$ and  a holomorphic map $g \colon \C \sm \tT\to \C$ that provides a conformal combination (``mating'') between $f$ and $G$ in the following sense:
\begin{enumerate}[leftmargin=8mm]
    \item 
    \label{It:Map}
    \emph{On the map side:} $g\big(\clo\big(\tilde U\big)\,\sm\tilde T\big) = \clo \big(\tilde U\big)$, and the maps $g$ and $f$ are conformally conjugate away from the grand orbits of $\tilde U$ and $U$; i.e., $g \colon \C \,\sm \bigcup_{k=0}^\infty \ovl{g^{-k}\big(\tilde U\big)} \to \C$ is conformally conjugate to $f \colon \C \,\sm \bigcup_{k=0}^\infty \ovl{f^{-k}(U)} \to \C$;
    \item 
    \label{It:Group}
    \emph{On the group side:} the map $g \colon {\clo\big(\tilde U\big) \sm \tT} \to \clo\big(\tilde U\big)$ is conformally conjugate to the factor Bowen--Series map $F_G \colon \ovl \disk \sm T \to \ovl \disk$ associated to the group $G$.
\end{enumerate}
    
\end{thm}

The technical details of the proof of Theorem~\ref{Thm:mating} will depend on whether $\star=\attr$ or $\star=\parab$. We will spend the next two subsections dealing with these cases separately.

\subsection{Proof of Theorem~\ref{Thm:mating}: the parabolic case ($\star=\parab$)}
\label{SSec:ParabProof}

We first treat the parabolic case. Thus $G=H_d$ and the associated factor
	Bowen--Series map is the degree-$d$ Farey map
	$F_d:\ovl\disk\sm T_d\to\ovl\disk$. We use the normalization of
	Proposition~\ref{Prop:FareyProperties}, so that the unique parabolic fixed
	point of $F_d|_\Circle$ is $1$.

    Let $f\in\mathscr K_\parab(d)$, and let $U$ be the invariant immediate
	basin of its distinguished parabolic fixed point. By
	Theorem~\ref{thm:bounded_geomfinite}, $U$ is a bounded Jordan domain. Let
	$p\in\partial U$ be the corresponding parabolic boundary fixed point, and
	choose a Riemann map $\phi \colon U\to\disk$ which extends homeomorphically to
	$\ovl U$ and satisfies $\phi(p)=1$. The map
	$B:=\phi\circ f\circ\phi^{-1} \colon \disk\to\disk$ is then a finite Blaschke product
	of degree $d$ ; equivalently, $B\circ\phi=\phi\circ f$ on $U$. Since the Blaschke product $B$ extends analytically by the Schwarz reflection across the boundary $\partial \disk$, it follows that $1$ is a triple
	fixed point for $B$.

We now apply Theorem~\ref{Thm:ParabolicCompatibility} to this Blaschke product $B$ and to the Farey map $F_d$. Set $A_B:=(B|_\Circle)^{-1}(1)$ and $A_F:=(F_d|_\Circle)^{-1}(1)$. The theorem gives a homeomorphism $\psi \colon \ovl\disk\to\ovl\disk$, quasiconformal on $\disk$, with $\psi(1)=1$ and $\psi(A_B)=A_F$, such that $\psi|_\Circle$ conjugates $B|_\Circle$ to $F_d|_\Circle$. Moreover, it gives closed one-sided	pinched neighbourhoods $X_B^-\subset X_B^+$ and $X_F^-\subset X_F^+$, with $X_F^-\subset\ovl\disk\sm T_d$ and $\psi(X_B^\pm)=X_F^\pm$ so that $F_d\circ\psi=\psi\circ B$ holds on $X_B^-$.

Define $\chi:=\psi\circ\phi \colon \ovl U\to\ovl\disk$. Then $\chi$ conjugates
	$f$ to $F_d$ on the boundary, and if
	$Y_B^\pm:=\phi^{-1}(X_B^\pm)$, then $Y_B^\pm$ are one-sided pinched
	neighbourhoods of $\partial U$ inside $\ovl U$, with $Y_B^-$ pinched precisely
	at the $d$ points $\phi^{-1}(A_B)$. On $Y_B^-\cap U$ we have
	$F_d\circ\chi=\chi\circ f$.

Let us define several objects: a domain
$\hat T := \chi^{-1}(T_d) \subset U$
and a quasi-regular map
\begin{equation*}
	\widecheck f \colon U \sm \hat T \to U, \quad\quad
	\widecheck f := \chi^{-1} \circ F_d \circ \chi.
\end{equation*}

Since $\chi(Y_B^-)=X_F^-\subset\ovl\disk\setminus T_d$, the set
	$Y_B^-\cap U$ lies in the domain of $\widecheck f$. Moreover, using
	$F_d\circ\psi=\psi\circ B$ on $X_B^-$ and $B\circ\phi=\phi\circ f$ on $U$, we get
	\begin{equation}
		\label{Eq:Parab:CollarEquality}
		\widecheck f=\chi^{-1}\circ F_d\circ\chi=f
		\quad\text{on }Y_B^-\cap U.
	\end{equation}
	Thus the replacement dynamics agrees with the original map on a one-sided
	pinched neighbourhood of $\partial U$ inside $U$.

Using these two objects, we define a semi-global continuous map $\hat f \colon \C \sm \hat T \to \C$ by setting
\begin{equation}
	\label{Eq:Parab:Tilde}
	\hat f(z) := \left\{
	\begin{aligned}
		&\widecheck f(z), &\quad \quad z \in U \sm \hat T\\
		&f(z),&\quad \quad z \in \C \sm U.
	\end{aligned} 
	\right.
\end{equation}

We now check that $\hat f$ is quasiregular. Away from $\partial U$ this is
	immediate: on $\C\sm\ovl U$ the map is holomorphic, since it coincides with
	$f$, while on $U\sm\hat T$ it is quasiregular, since
	$\hat f=\chi^{-1}\circ F_d\circ\chi$, the map $F_d$ is holomorphic on
	$\disk\sm T_d$, and $\chi$ is quasiconformal on $U$. It remains to consider
	points of $\partial U$. Put $P:=\phi^{-1}(A_B)$. If
	$\xi\in\partial U\sm P$, then the pinched collar $Y_B^-$ has positive width near
	$\xi$; hence, after shrinking a neighbourhood $N_\xi$ of $\xi$, we have
	$N_\xi\cap U\subset Y_B^-\cap U\subset U\sm\hat T$. By
	\eqref{Eq:Parab:CollarEquality}, the inside branch of $\hat f$ agrees with $f$
	on $N_\xi\cap U$, while the outside branch is $f$ by definition. Thus $\hat f$
	is holomorphic near every point of $\partial U\sm P$. The only remaining
	possible exceptional points are the finite set $P$; they are removable for
	quasiregular maps, since $\hat f$ is quasiregular in punctured neighbourhoods of
	these points and locally bounded. Therefore $\hat f$ is quasiregular. This is
	the local gluing argument underlying McMullen's Gluing Lemma; compare
	\cite[Proposition~7.29]{branner_fagella_2014}.

It also follows by construction that $\hat f$ admits an invariant Beltrami coefficient $\mu$ defined on $\C$ as follows (here, $\mu_0$ is the standard coefficient that gives the standard complex structure):
\begin{equation}
	\label{Eq:Mu:Parab}
	\mu(z) = \left\{
	\begin{aligned}
		&\mu_0(z), &\quad\quad
		&z \in \C \, \sm \bigcup_{k=0}^\infty f^{-k}(U),\\
		&{\chi}^* (\mu_0)(z), &\quad\quad
		&z\in U,\\
		&{(\chi} \circ f^k)^*(\mu_0)(z), &\quad\quad
		&z \in f^{-k}(U)\setminus f^{-(k-1)}(U), \,\, k \in \mathbb N_{\ge 1}.
	\end{aligned}
	\right.
\end{equation}
In this way, $\mu$ is a Beltrami coefficient with bounded dilatation. 

Finally, let $\Phi \colon \C \to \C$ be the solution of the Beltrami equation with the coefficient $\mu$. The map 
\[
g := \Phi \circ \hat f \circ \Phi^{-1} \colon \C \sm \widetilde T \to \C, \quad\quad \widetilde T := \Phi\big(\hat T\big), \quad \tilde U := \Phi(U).
\]
is the required map. Indeed, since $\Phi$ integrates the standard Beltrami form away from the grand orbit of $U$, and on that set $\hat f$ acts as $f$, we obtain that $g$ and $f$ are conformally conjugate away from the grand orbit of $U$ and $\tilde U$. This establishes~\eqref{It:Map}. On the other hand, on $U$, the map $\Phi$ integrates the Beltrami form $\chi^* (\mu_0)$. As $\mu_0\vert_{\D}$ is $F_d-$invariant, it follows from the definition of $\hat f$ and $g$ that $g$ restricted to $\Phi\big(\clo\big(U\big) \sm \hat T\big) = \clo\big(\tilde U\big) \sm \tilde T$ is conjugate to $F_d |_{\ovl \disk \sm T}$ via the conformal map $\chi\circ\Phi^{-1} \colon \clo\tilde U\to\overline{\D}$. This
establishes \eqref{It:Group}.
\qed
\subsection{Proof of Theorem~\ref{Thm:mating}: the attracting case ($\star=\attr$)}
Let $G\in\mathscr G_{\attr}(d)$ and let	$F=F_G \colon \ovl\disk\sm T\to\ovl\disk$ be the associated factor Bowen--Series map. Let $f\in\mathscr K_{\attr}(d)$, and let $U$ be the invariant immediate basin of the distinguished attracting fixed point, $p$, with $\deg(f \colon U\to U)=d$. Choose a Riemann map $\phi \colon U\to\disk$ so that $\phi(p)=0$. By Proposition~\ref{Prop:4David}, $U$ is a quasidisk, so $\phi$ extends homeomorphically to $\ovl U$. The map $B:=\phi\circ f\circ\phi^{-1} \colon \disk\to\disk$ is a finite Blaschke product of degree $d$ satisfying $B(0)=0$ and  $B|_\Circle$ is an expanding covering of degree $d$.

The boundary map $F|_\Circle$ is a degree-$d$ expansive covering,
	but it is not expanding: by
	Theorem~\ref{Thm:FBSProperties}, the singular boundary points of the
	factor tile are parabolic points of period one or two. Hence the
	boundary conjugacy between $B|_\Circle$ and $F|_\Circle$ cannot be quasisymmetric, which makes it impossible to use in this attracting case the classical quasiconformal surgery as in Subsection~\ref{SSec:ParabProof}. We will use the David surgery techniques instead, following the strategy of \cite{LMMN} (see, e.g., \cite[Theorem~7.1]{LMMN}). By Theorem~\ref{Thm:AttractingCompatibility} there exists an
	orientation--preserving homeomorphism $\psi \colon \Circle\to\Circle$ satisfying
	$F\circ\psi=\psi\circ B$ on $\Circle$, that admits a David
	extension $\psi:\disk\to\disk$.

As in the parabolic case, define a domain $\hat T := \phi^{-1} \circ \psi^{-1}(T)\subset U$ and a continuous map $\hat f \colon \C \sm \hat T \to \C$ given by
\begin{equation}
\label{f_tilde_def_eqn}
    \hat f(z) := \left\{
\begin{aligned}
    &\phi^{-1} \circ \psi^{-1} \circ F \circ \psi \circ \phi (z), &\quad &z \in U \sm \hat T,\\
    &f(z), &\quad &z \in \C \sm U.
\end{aligned}
    \right.
\end{equation}
We then define a Beltrami coefficient $\mu(z)$ by 
\begin{equation}
\label{Eq:Mu:Parab1}
    \mu(z) = \left\{
    \begin{aligned}
        &\mu_0(z), &\quad\quad &z \in \Cc \, \sm \bigcup_{k=0}^\infty f^{-k}(U),\\
        &(\psi \circ \phi)^* (\mu_0)(z), &\quad\quad &z\in U,\\
        &(\psi \circ \phi \circ f^k)^*(\mu_0)(z), &\quad\quad &z \in f^{-k}(U){\setminus f^{-(k-1)}(U)}, \,\, k \in \mathbb N_{\ge 1}.
    \end{aligned}
    \right.
\end{equation}
In this way, $\mu$ is invariant under $\hat f$. 

\noindent
\textsc{Claim I:} \textit{The Beltrami coefficient $\mu(z)$ is, in fact, a David coefficient in $\Cc$. This means that there exist constants $C \ge 1$, $\alpha > 0$, $\epsilon_0 > 0$ so that for all $\epsilon < \epsilon_0$, the spherical measure of points with large dilatation is exponentially small; specifically,
\begin{equation}
    \label{Eq:DavidCond}
    \sigma\big(\{z \in \Cc : |\mu(z)| \ge 1-\epsilon\}\big) \le C e^{-\frac{\alpha}{\epsilon}},
\end{equation}
where $\sigma$ is the spherical measure.}
\medskip 

\begin{proof}[Proof of \textsc{Claim I}]
In order to see this, we will use Proposition~\ref{Prop:4David}, Lemma~\ref{Lem:UnivComp}, and the properties of the class $\mathscr K_{\attr}(d)$ given in Definition~\ref{Def:Kattr}. By \eqref{Eq:Mu:Parab1}, $\mu$ is equal to the standard complex structure everywhere except on the grand orbit of $U$. Hence, we need to check the David condition \eqref{Eq:DavidCond} only on $U$ and its iterated preimages. 

By Proposition~\ref{Prop:4David}, the immediate basin $U$ is a
	quasidisk, hence a John domain. Since $\psi \colon \D\to\D$ is a David
	homeomorphism and $\phi \colon U\to\D$ is conformal, 
	\cite[Proposition~2.5(iv)]{LMMN} implies that
	$\psi\circ\phi \colon U\to\D$ is a David map. Therefore
	$\mu|_U=(\psi\circ\phi)^*(\mu_0)$ is a David coefficient on $U$. Thus
	there exist constants $C_0\ge1$, $\alpha_0>0$, and $\eps_0>0$ such that,
	for $0<\eps\le\eps_0$,
\[
{
	A=A(\eps):=\{z\in U:|\mu(z)|\ge1-\eps\}
	\quad\text{satisfies}\quad
	\sigma(A)\le C_0 e^{-\alpha_0/\eps}.}
\]

We will need the following lemma:
\begin{lemma}
\label{Lem:SphEuc}
        {For $s\ge1$, let $\,\mathcal U_s$ denote the collection of connected components of $f^{-s}(U)\setminus f^{-(s-1)}(U),$ and set}
    	\[
    	{
    		A_s:=f^{-s}(A)\cap\big(f^{-s}(U)\setminus f^{-(s-1)}(U)\big).}
    	\]
    Let $N \ge 1$ be the smallest integer such that for all $n \ge N$ each component $W\in \mathcal U_{n+1}$ is mapped univalently by $f^{n+1-N}$ onto a component of $\,\mathcal U_N$. Then there exists $C_1 \ge C_0$ and $\alpha_0 \ge \alpha_1>0$  such that
    \[
    \sum \limits_{s=1}^N \sigma\big({A_s}\big) \le C_1 \cdot e^{-\frac{\alpha_1}{\eps}}.
    \]
    for every $\eps_0 \ge \eps > 0$.

    Furthermore, there exists a constant $C_2 \ge C_1$ such that
    \[
    \sum \limits_{s \ge N+1} \sigma\big( {A_s}\big) \le C_2 \cdot e^{-\frac{\alpha_1}{\eps}}.
    \]    
\end{lemma}

\begin{proof}[Proof of Lemma~\ref{Lem:SphEuc}]
	{We first prove the estimate for $s=1$. The same argument, applied to $f^s$ and to the components of exact rank $s$, gives the corresponding estimate for each fixed $1<s\le N$; since there are only finitely many such levels, the constants can then be adjusted.} Let us label the connected components of $f^{-1}(U)$ as $U_i$, $i \in I_1$, and let $\mathcal{B}_i := f^{-1}(A) \cap U_i$ be the preimage of $A$ in $U_i$.

To control the interplay between the spherical and the Euclidean geometry, let us cover $\Cc$ with two open topological disks $\hat X_1$ and $\hat X_2$ that together provide an atlas of $\Cc$. We assume that $U \subset \hat X_1$ and $U \subset \Cc \sm \hat X_2$. 

Furthermore, we can assume that each $U_i$ lies either in $\hat X_1$, or in $\hat X_2$. Pick points $x_1 \in U \subset \hat X_1$ and $x_2 \in \hat X_2 \sm \hat X_1$, and uniformize each $\Cc \sm \{x_1\}$ and $\Cc \sm \{x_2\}$ by the complex plane $\C$. The images of $\hat X_1$ and $\hat X_2$ will be some topological disks $X_1$ and $X_2$ (and we will keep the notation $U_i$ for the images of the domains $U_i$ viewed as lying in $X_1$ or $X_2$). 

In this way, each Borel set $B \subset \Cc$ acquires the Euclidean area $\sigma_{Euc}(B)$ measured in the respective $\C$, and furthermore, there exists a constant $M$ that depends only on $\hat X_1$ and $\hat X_2$, such that for each such $B$,  
\begin{equation}
\label{Eq:Comp}
M^{-1} \cdot \sigma(B) \le \sigma_{Euc}(B) \le M \cdot \sigma(B).
\end{equation}

Choose a topological disk $V \subset \hat X_1$ such that $U \Subset V$ and $V\sm U$ is disjoint from the postcritical set of $f$. This is possible because $f$, by assumption, is a hyperbolic map. For every $U_i$, let $V_i$ be the component of $f^{-1}(V)$ containing $U_i$. In this way, $f \colon V_i\to V$ is a proper holomorphic map between two topological disks and 
\[
\deg(f \colon V_i  \to V) = \deg(f \colon U_i  \to U)
\]

By Lemma~\ref{Lem:UnivComp}, the degrees of all such maps $f \colon V_i \to V$ are uniformly bounded over all preimage components $U_i$. By \cite[Lemma A.7]{CDKS} (see also \cite[Fact 6.2]{KSS}), each $U_i$ has $\eta-$bounded geometry, where $\eta$ depends only on the bound on the degrees, $\Mod(V\sm \ovl U)$, and the bound on the geometry of $U$, and hence is uniform over all $i \in I_1$. Recall that $\eta-$bounded geometry means that $U_i$ contains a Euclidean round disk of radius $\eta \cdot \diam_{Euc}(U_i)$.  In particular, there exists a constant $M' \ge 1$ (independent of $i$) such that the Euclidean diameter $\diam_{Euc}(U_i)$ satisfies 
\begin{equation}
\label{Eq:Comp2}
\diam^2_{Euc}(U_i) \le M'\cdot \sigma_{Euc}(U_i).
\end{equation}

By \cite[Lemma 6.4]{Zha16}, there exist some constants $C', \alpha_1$ (that depend on $C_0$, $\alpha_0$, and the uniform bound on the degrees of $f|U_i$) such that
\begin{equation}
\label{eq:1}
\sigma_{Euc}(\mathcal{B}_i) \le C' \diam_{Euc}^2(U_i) \cdot e^{-\frac{\alpha_1}{\eps}}.
\end{equation} 
Therefore, by \eqref{Eq:Comp}, \eqref{Eq:Comp2}, and \eqref{eq:1},
\begin{equation*}
    \begin{aligned}
\sigma(\mathcal{B}_i) &\le M \cdot \sigma_{Euc}(\mathcal{B}_i) \le M \cdot C' \diam_{Euc}^2(U_i) \cdot e^{-\frac{\alpha_1}{\eps}} \le M \cdot C' \cdot M' \cdot \sigma_{Euc}(U_i) \cdot e^{-\frac{\alpha_1}{\eps}} \\ &\le M^2 \cdot C' \cdot M' \cdot \sigma(U_i) \cdot e^{-\frac{\alpha_1}{\eps}} = C'' \sigma(U_i) \cdot e^{-\frac{\alpha_1}{\eps}},
    \end{aligned}
\end{equation*} 
where $C''$ is the combined uniform constant. We conclude that 
\begin{equation}
\label{Eq:Conclusion}
{\sigma(A_1)} = \sum_{i \in I_1} \sigma(\mathcal{B}_i) \le C'' e^{-\frac{\alpha_1}{\eps}} \sum_{i\in I_1} \sigma(U_i) \le C'' \cdot \sigma(\Cc) \cdot e^{-\frac{\alpha_1}{\eps}}.
\end{equation}
{This proves the desired estimate for $s=1$. Repeating the same argument for each $s=2,\ldots,N$, and then increasing the constant and decreasing the exponent if necessary, gives $\sum_{s=1}^{N}\sigma(A_s)\le C_1 e^{-\alpha_1/\eps}$. We shall also use the corresponding componentwise estimate at level $N$: if $W\in\mathcal U_N$ and $A_W:=f^{-N}(A)\cap W$, then $\sigma(A_W)\le C_1\,\sigma(W)\,e^{-\alpha_1/\eps}$ for every $W\in\mathcal U_N.$}

Let us now prove the second claim. {For this part of the proof, we relabel the components of 
	$f^{-N}(U)\setminus f^{-(N-1)}(U)$ as $U_i$, $i\in I_N$, and re-define $\mathcal{B}_i$
	as $f^{-N}(A)\cap U_i$. By the componentwise estimate obtained above at level $N$,}
\begin{equation}\label{eq:Ai}
{\sigma(\mathcal{B}_i)\le C_1\,\sigma(U_i)\,e^{-\alpha_1/\eps}
	\quad\text{for all }i\in I_N.}
\end{equation}

Let $U_i'$ be a connected component of $f^{-k}(U_i)$ with $k\ge1$. 
{By the choice of $N$, the map} $f^k \colon U_i' \to U_i$ is univalent. 
Furthermore, {let $V_i$ be the component of $f^{-N}(V)$ containing $U_i$.} If $V_i'$ is the component of $f^{-k}(V_i)$ containing $U_i'$, then $f^k \colon V_i' \to V_i$ is univalent and 
\[
\Mod\big(V_i'\sm\ovl{U'_i}\big) = \Mod\big(V_i\sm\ovl{U_i}\big) \ge \frac{1}{D}\Mod\big(V \sm \ovl{U}\big),
\]
where $D$ is the uniform bound on the degrees of the landing maps $f^N \colon U_i \to U$, $i\in I_N$. Therefore, all moduli $\Mod(V_i' \sm \ovl{U_i'})$ are bounded below uniformly over $i$ and $k$. Thus, each map $f^k \colon U_i' \to U_i$ admits a Koebe extension $f^k \colon V_i' \to V_i$ with Koebe space uniformly bounded below.  By the Koebe distortion theorem, together with the spherical--Euclidean comparison~\eqref{Eq:Comp} {and the uniform bounded geometry of the Fatou components,} there exists a constant $L$, uniform over $i$ and $k$, such that for every Borel set $E\subset U_i$,
\[
{
	\sigma\left((f^k|_{U_i'})^{-1}(E)\right)
	\le L\cdot \sigma(U_i')\cdot \frac{\sigma(E)}{\sigma(U_i)}.}
\]
{Taking $E=\mathcal{B}_i$, and using \eqref{eq:Ai}, we get}
\begin{equation*}
		\begin{aligned}
			\sigma\big(\{z \in U_i' : |\mu(z)|\ge 1-\eps\}\big)
			&= \sigma\left(f^{-k}(\mathcal{B}_i) \cap U_i'\right)	{
			\le L\cdot \sigma(U_i')\cdot \frac{\sigma(\mathcal{B}_i)}{\sigma(U_i)}}\\
			&\le {L\cdot C_1\cdot e^{-\alpha_1/\eps}\cdot \sigma(U_i').}
	\end{aligned}
\end{equation*}
Hence, {summing over all components whose first landing time to $U$ is at least $N+1$,}
\begin{equation*}
	{
		\begin{aligned}
			\sum_{s \ge N+1} \sigma(A_s)
			&= \sum_{i\in I_N} \, \sum_{k \ge 1} \, 
			\sum_{\substack{U_i'\\ \text{ c.c. of }f^{-k}(U_i)}}
			\sigma\left(f^{-k}(\mathcal{B}_i) \cap U_i'\right)\\
			&\le L\cdot C_1\cdot e^{-\alpha_1/\eps}
			\sum_{i\in I_N} \, \sum_{k \ge 1} \,
			\sum_{\substack{U_i'\\ \text{ c.c. of }f^{-k}(U_i)}} \sigma(U_i')\\
			&\le L\cdot C_1\cdot \sigma(\Cc)\cdot e^{-\alpha_1/\eps}.
	\end{aligned}}
\end{equation*}
(where the notation c.c. stands for ‘connected components of’) and the second claim of the lemma follows with {$C_2:=\max\{C_1,LC_1\sigma(\Cc)\}$.}
\end{proof}
We are now ready to combine all the ingredients. Since, by construction, the dilatation of $\mu$ can be large on {$A$ and its pullbacks $A_s$,} using Lemma~\ref{Lem:SphEuc}, we estimate:
\begin{equation}
    \begin{aligned}
\sigma\big(\{z \in \Cc : |\mu(z)|\ge 1-\eps\}\big) &= \sigma(A) + \sum_{s = 1}^{N} \sigma\left({A_s}\right) + \sum_{s \ge N+1} \sigma\left({A_s}\right) \\ 
        &\le C_0 \cdot e^{-\frac{\alpha_0}{\eps}} + C_1 \cdot e^{-\frac{\alpha_1}{\eps}} + C_2 \cdot e^{-\frac{\alpha_1}{\eps}} \le (C_0+C_1+C_2) \cdot e^{-\frac{\alpha_1}{\eps}}    
    \end{aligned}
\end{equation}
This concludes the proof of \textsc{Claim I} with $C := C_0+C_1+C_2$ and $\alpha := \alpha_1$.
\end{proof}

Let us conclude the proof of the theorem using the David coefficient $\mu(z)$. By the David Integrability Theorem \cite[Theorem 20.6.2]{David}, there exists a unique (up to M\"obius transformation in $\Cc$) David homeomorphism $\Psi \colon \Cc \to \Cc$, $\Psi(\infty) = \infty$, that solves the Beltrami equation
\[
\frac{\partial \Psi}{\partial \ovl z} = \mu \frac{\partial \Psi}{\partial z}.
\]

\noindent
\textsc{Claim II:} {The map 
\[
g:= \Psi \circ \hat f \circ \Psi^{-1} \colon \C \sm \tilde T \to \C, \quad \tilde T:=\Psi(\hat T), \quad \tilde U := \Psi(U)
\]
is analytic, and it is the required map.}

\begin{proof}[Proof of \textsc{Claim II}]
Let $W$ be an open set in $\C\setminus\widehat{T}$ such that $\hat f \vert_W$ is a homeomorphism. Since $\Psi\circ\hat f=g\circ \Psi$, it suffices to show that $\Psi\circ\hat f$ is a David homeomorphism on $W$. Indeed, this would imply that both $\Psi$ and $\Psi\circ\hat f$ are David homeomorphisms on $W$ integrating $\mu$ (note that $\mu$ is $\hat f-$invariant), and hence by \cite[Theorem~20.4.19, p.~565]{David}, the map $g$ is analytic on {$\Psi(W)$}.
Further, as $U$ is a quasi-disk, its boundary is removable for $W^{1,1}-$functions and hence $\Psi(\partial U)$ is conformally removable (cf. \cite[\S 2]{LMMN}). Thus, we may assume that $W$ is disjoint from $\partial U$.

Suppose first that $W\subset \C\setminus\overline{U}$. By construction, $\hat f=f$ on $\C\setminus\overline{U}$, and {on the chosen set $W$ this branch is conformal}. Hence, by \cite[Proposition~2.5, part (ii)]{LMMN}, $\Psi\circ\hat f$ is a David homeomorphism on $W$.

Now let $W\subset U$. In this case, we have the following factorization of $\Psi\circ\hat f$ on $W$ (see~\eqref{f_tilde_def_eqn}):
$$
\Psi\circ\hat f = \left(\Psi\circ\phi^{-1}\circ\psi^{-1}\right)\circ\left(F\circ \psi\circ\phi\right).
$$

As $\phi$ is conformal, $\psi$ is a David homeomorphism, and {$F$ is conformal on  $\psi\circ\phi(W)$}, it again follows from \cite[Proposition~2.5, parts (i) and (ii)]{LMMN} that $F\circ \psi\circ\phi$ is a David homeomorphism on $W$. We also know that $\psi\circ\phi$ is David, and hence both the David homeomorphisms $\Psi$ and $\psi\circ\phi$ straighten {$\mu|_U$}. Thus, \cite[Theorem~20.4.19, p.~565]{David} implies that $\Psi\circ\phi^{-1}\circ\psi^{-1}$ is conformal. It now follows from \cite[Proposition~2.5, part (i)]{LMMN} and the above factorization that $\Psi\circ\hat f$ is a David homeomorphism on $W$. 

{Since the critical points of the holomorphic branches of $\hat f$ form a discrete set, the preceding argument shows that $g$ is analytic off a discrete subset of $\C\setminus \tilde T$. As $g$ is continuous, the removable singularity theorem gives analyticity across these points. Together with the removability of $\Psi(\partial U)$, this proves that $g$ is analytic on all of $\C\setminus \tilde T$.}

Note that the arguments of the previous paragraph show that $\Psi\circ\phi^{-1}\circ\psi^{-1}$ is a conformal conjugacy between $F\vert_{\D\setminus T}$ and $g\vert_{\widetilde{U}\setminus\widetilde{T}}$ (where $\tilde U=\Psi(U)$). Moreover, as $\mu=\mu_0$ outside the grand orbit of $U$, we have that both $\Psi$ and the identity map integrate the trivial Beltrami form outside the grand orbit of $U$. Hence, $\Psi$ is conformal on this set {and, in particular, on $\C \,\sm \bigcup_{k=0}^\infty \ovl{f^{-k}(U)}.$}
Finally, as $\hat f=f$ outside $U$, {and since the removed set contains $U$ and all its preimages, the complements of the grand orbits of $U$ for $f$ and for $\hat f$ agree. Hence, using $g=\Psi\circ\hat f\circ\Psi^{-1}$ and $\tilde U=\Psi(U)$,} we conclude that $\Psi$ conformally conjugates $f \colon \C \,\sm \bigcup_{k=0}^\infty \ovl{f^{-k}(U)} \to \C$ to $g \colon \C \,\sm \bigcup_{k=0}^\infty \ovl{g^{-k}\big(\tilde U\big)} \to \C$. This completes the proof of \textsc{Claim II}, and thus of Theorem~\ref{Thm:mating} in the attracting case.
\end{proof}
\section{Analytic description of conformal combinations}\label{analytic_description_sec}

Let $g:\C\setminus\widetilde{T}\to\C$ be a conformal combination between a map $f\in\Kstar(d)$ and a group $G\in\Gstar(d)$ (more precisely, the factor Bowen--Series map $F_G$ associated with $G$) constructed in Theorem~\ref{Thm:mating}.
The goal of this section is to provide an analytic non-dynamical description of $g$ in terms of the M{\"o}bius involution $\eta(z)=1/z$ and a meromorphic map with a simple pole at the origin. The main result of the section is the following theorem.

\begin{thm}[Analytic description of conformal combinations]
\label{Thm:Analytic}
    Given $d \ge 2$, a map $f \in \Kstar(d)$, a group $G \in \Gstar(d)$, and their conformal combination $g \colon \C \,\sm \tilde T \to \C$ as in Theorem~\ref{Thm:mating}, there exist a meromorphic map $\mathfrak{h} \colon\C\to \Cc$ and a pair of open topological disks $\mathcal V^0, \mathcal V^\infty \subset \Cc$ with the following properties:
    \begin{enumerate}[leftmargin=8mm]
  	     \item $\partial \mathcal V^0 = \partial \mathcal V^\infty$ is a piecewise analytic Jordan curve, $0 \in \mathcal V^0$, $\infty\in\mathcal V^\infty$, and we have the decomposition $\widehat{\C} = \mathcal V^0 \sqcup \partial \mathcal V^0 \sqcup \mathcal V^\infty$;
        \item the boundary $\partial \mathcal V^0$ is invariant under the involution $\eta(z) = 1/z$, and hence $\eta (\mathcal V^0) = \mathcal V^\infty$;
        \item $\mathfrak{h}$ has a unique simple pole at $0$;
        \item $\mathfrak{h}|_{\overline{\cV^0}}$ sends $(\overline{\cV^0},0)$ homeomorphically onto $(\widehat{\C} \sm \tilde T,\infty)$ and is conformal on $\cV^0$;
        \item $\mathfrak{h} \colon {\overline{\cV^\infty}\,\sm \{\infty\}} \to \C$ is equal to $g \colon \C \, \sm \tilde T \to \C$ up to pre-composition with a map conformal on $\cV^\infty$, i.e., there exists a homeomorphism $\psi_{\infty} \colon (\overline{\cV^\infty},\infty) \to (\widehat{\C} \,\sm \tilde T,\infty)$ which is conformal on $\cV^\infty$ such that
    \[
    \mathfrak{h}|_{\overline{\cV^\infty}\,\sm \{\infty\}} = g\circ \psi_\infty|_{\overline{\cV^\infty}\,\sm \{\infty\}};\ and
    \]
    \item  $\psi_\infty|_{\overline{\cV^\infty}} = \mathfrak{h} \circ \eta\vert_{\,\overline{\cV^\infty}}$; in particular, 
    \[
    g\vert_{\C \,\sm \tilde T}\equiv \mathfrak{h}\circ\eta\circ(\mathfrak{h}\vert_{\overline{\cV^0}\,\sm \{0\}})^{-1}.
    \]    
    \end{enumerate}
\end{thm}

\begin{proof}
We start by recalling that the factor Bowen--Series map $F_G:\overline{\D}\setminus T\to\overline{\D}$ restricts to $\partial T$ as an orientation-reversing involution. Hence, the conformal combination $g \colon \C \, \sm \tilde T \to \C$ is an orientation-reversing involution on $\partial \widetilde{T}$.
Since $\partial T$ is a piecewise non-singular real-analytic curve, the same is true for $\partial\widetilde{T}$. Further, the set $\mathfrak{s}$ of singularities of $\partial\widetilde{T}$ is finite, $g(\mathfrak{s})=\mathfrak{s}$, and each point of $\mathfrak{s}$ is either a fixed point or of period two under~$g$.

Let $\cV$ be the unbounded simply connected domain $\widehat{\C}\,\setminus\,\clo{\widetilde{T}}$; in this way $\overline{\cV}\setminus\{\infty\}$ is the domain of definition of $g$. Further let $\partial^0 \cV:=\partial\cV\setminus\mathfrak{s}$. We note that $g$ extends as an analytic diffeomorphism in a pinched neighborhood of $\partial^0\cV$ (i.e., to an open set disjoint from $\mathfrak{s}$ and containing~$\partial^0 \cV$). 

Consider two copies of $\cV\sqcup\partial^0\cV$, and weld them via the orientation-reversing boundary involution $g:\partial^0\cV\to\partial^0\cV$ (cf. \cite[Chapter~2, Page~117]{AS60}). In this way, we identify a point $x \in \partial^0\cV$ in one copy with the point $g(x) \in \partial^0 \cV$ in the second copy. After the uniformization, we obtain a welded genus--zero Riemann surface, which we denote by $\Sigma$. This surface $\Sigma$ consists of two simply connected domains $\Sigma^\pm$ corresponding to the two copies of $\cV$ under the uniformization and their common boundary. Let us denote the conformal maps $\phi_\pm \colon \Sigma^\pm \to \cV$ on each of the copies that arise from uniformization (see Figure~\ref{welding_fig}). The maps $\phi_\pm$ extend homeomorphically to the closures, and the boundary extensions satisfy the relation $\phi_+=g\circ\phi_-$ on $\partial\Sigma^-=\partial\Sigma^+$.

Since $g\vert_{\partial\widetilde{T}}$ is an involution, we obtain a conformal involution \begin{equation*}
    \iota:\Sigma\to\Sigma,\quad 
    \begin{cases}
     \phi_-^{-1}\circ\phi_+,\quad \mathrm{on}\ \overline{\Sigma^+}  \\
     \phi_+^{-1}\circ\phi_-,\quad \mathrm{on}\ \Sigma^-.
    \end{cases}
\end{equation*} 
This involution $\iota$ exchanges the two conformal copies $\Sigma^\pm$ of $\cV$ in $\Sigma$, and preserves their common boundary $\partial \Sigma^+ = \partial \Sigma^-$. In fact, $\iota\vert_{\partial \Sigma^+}$ is conjugate to $g\vert_{\partial\widetilde{T}}$ via the uniformizing map $\phi_+$. 

\vskip -4mm
\begin{figure}[h!]
\includegraphics[width=0.9\textwidth]{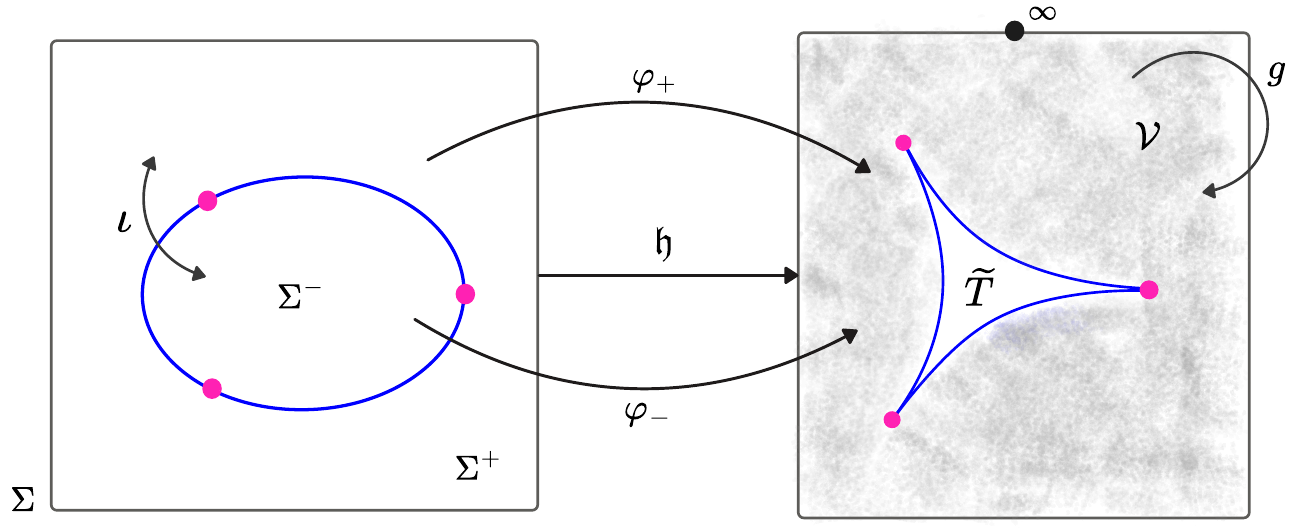}
\caption{Illustrated is the welding construction of Theorem~\ref{Thm:Analytic}.}
\label{welding_fig}
\end{figure}

Define a meromorphic map $\mathfrak{h}:\Sigma\to\widehat{\C}$ as
$\mathfrak{h}\equiv \phi_-$ on $\Sigma^-$ and $\mathfrak{h}\equiv g \circ \phi_+$ on $\Sigma^+\,\setminus\,\{\phi_+^{-1}(\infty)\}$. The fact that the piecewise definitions agree along $\partial \Sigma^+ = \partial \Sigma^-$ follows from the fact that the welding was done via the map $g$; i.e., from the relation $\phi_+=g\circ\phi_-$ on $\partial\Sigma^-=\partial\Sigma^+$.

We will argue that $\Sigma$ is a punctured sphere; i.e., the boundary components of $\Sigma$ corresponding to the singular points of $\partial\widetilde{T}$ are points. To this end, we need to take a careful look at the action of $g$ near the points of $\mathfrak{s}$. 

Let $p$ be a fixed point in $\mathfrak{s}$. Note that $p$ is the result of the welding of a parabolic fixed point of the factor Bowen--Series map $F_G$ and a repelling/parabolic fixed point of $f$.
The discussion of Section~\ref{fbs_para_fixed_local_dyn_subsec} now implies the following structure.
Let $\alpha$ be a simple closed curve surrounding $p$ intersecting $\partial\widetilde{T}$ at two points $q, r$ such that $g(q)=r$ (see Figure~\ref{mating_parabolic_fixed_fig}, right). Denote the disk bounded by $\alpha$ as $B_\alpha$, and set $\widetilde{T}_\alpha:=\widetilde{T}\cap B_\alpha$ (shaded in blue). The local preimage of $\alpha$ under $g$ is a simple arc $\beta$ connecting $q, r$. Further, $g^{-1}(\widetilde{T}_\alpha)$ is the union of two topological triangles (shaded in orange) and $g:g^{-1}(\widetilde{T}_\alpha)\to\widetilde{T}_\alpha$ is a $2:1$ covering map. On the other hand, $g^{-1}(B_\alpha\setminus\overline{\widetilde{T}_\alpha})$ is a Jordan domain (shaded in green) which maps univalently onto $B_\alpha\setminus\overline{\widetilde{T}_\alpha}$ under $g$.

\begin{figure}[ht!]
\includegraphics[width=0.97\textwidth]{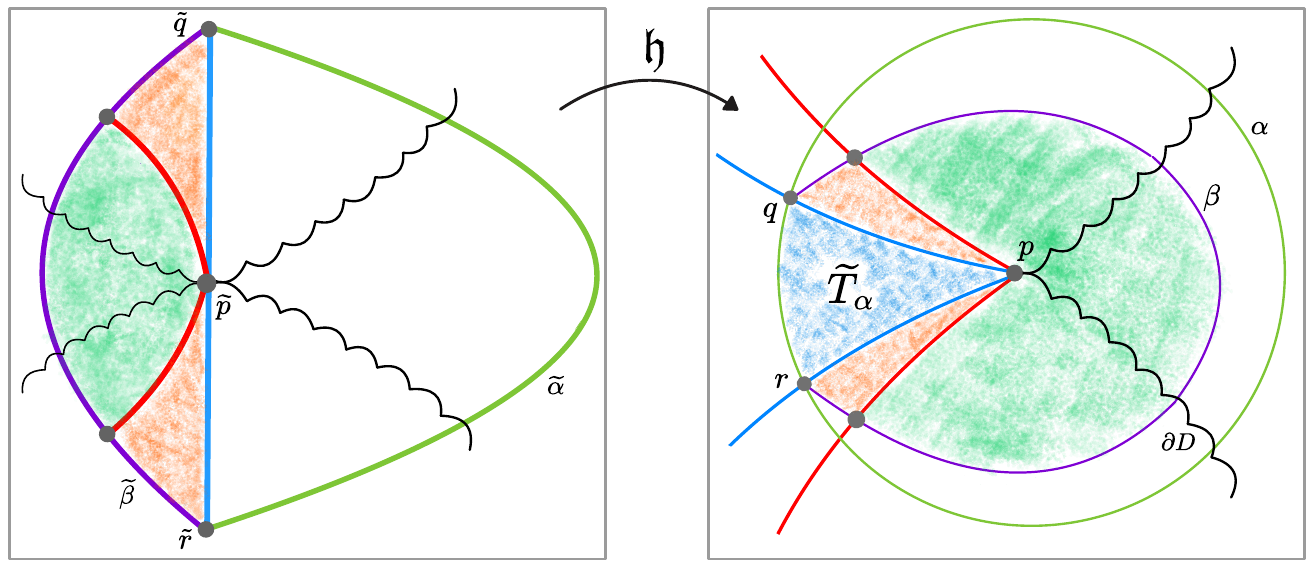}
\caption{Depicted is the local dynamics of a conformal mating near a parabolic fixed point.}
\label{mating_parabolic_fixed_fig}
\end{figure}

The construction of $\Sigma$ welds the arc $\arc{q,p}\setminus\{p\}\subset\partial \widetilde{T}_\alpha$ (respectively, $\arc{r,p}\setminus\{p\}\subset\partial \widetilde{T}_\alpha$) in the first copy of $\cV\cup\partial^0\cV$ with the arc $\arc{r,p}\setminus\{p\}\subset\partial \widetilde{T}_\alpha$ (respectively, $\arc{q,p}\setminus\{p\}\subset\partial \widetilde{T}_\alpha$) in the second copy of $\cV\cup\partial^0\cV$. We denote the image of $\alpha$ (belonging to the first copy of $\cV\cup\partial^0\cV$) and $\beta$ (belonging to the second copy of $\cV\cup\partial^0\cV$) in $\Sigma$ as $\widetilde{\alpha}, \widetilde{\beta}$, respectively (see Figure~\ref{mating_parabolic_fixed_fig}, left). Note that the point $p$ itself corresponds to a boundary component of $\Sigma$. In particular, the Jordan curve $\widetilde{\alpha}\cup\widetilde{\beta}$ encloses an annulus $\widetilde{A}$ in $\Sigma$. The map $\mathfrak{h}$ is now easily seen to be a holomorphic double covering from $\widetilde{A}$ onto the (degenerate) annulus $B_\alpha\setminus\{p\}$. Since $B_\alpha\setminus\{p\}$ is biholomorphic to a punctured disk, it has infinite modulus. It follows that the modulus of $\widetilde{A}$ is also infinite. Standard results for Riemann surfaces now tell us that the boundary component of $\Sigma$ associated with the parabolic fixed point $p\in\mathfrak{s}$ reduces to a point and hence the welding extends to the boundary component, which we denote by $\widetilde{p}$ (cf. \cite[Theorem~1]{Oik63}).

A similar argument can be used to show that the boundary components of $\Sigma$ associated with the parabolic $2-$cycles in $\mathfrak{s}$ are also points.

It follows that $\Sigma$ is a punctured sphere. The map $\iota$ extends continuously to the punctures, and hence is a conformal map on $\Cc$. Therefore, $\iota$ is a M{\"obius} map. Every M{\"obius} involution is M{\"obius} conjugate to $\eta(z) = 1/z$ once we send two fixed points of $\iota$ to $\pm 1$. Hence, we can assume that the involution $\iota$ is $\eta(z) = 1/z$. The map $\mathfrak{h}:\Sigma\to\widehat{\C}$ extends to all but one punctures (recall that $g$ has an essential singularity coming from that of $f$), and we can further pre-compose $\mathfrak{h}$ with a M{\"o}bius map to require that $\mathfrak{h}:\C\to\widehat{\C}$ is a meromorphic map with a simple pole at $0$. The required domains $\cV^0$, $\cV^\infty$ are the images of $\Sigma^\pm$ after all these coordinate changes.
\end{proof}

\subsection{The line complex of $\mathfrak{h}$ for maps in the Speiser class: an example}\label{speiser_subsec}

In this subsection, we will describe a procedure that shows how to construct the meromorphic map $\mathfrak{h} \colon \C \to \Cc$ based on the mapping properties of the starting map $f$ and the group $G$. For this, we will use the concept of \emph{line complexes}. Below, we explain the key steps in the construction of line complexes, and refer the reader to \cite[Chapter 7]{GO} for further in-depth discussion.{We first recall the general construction of line complexes. We then illustrate the construction in the specific case where $f(z)=\frac{\pi}{2}\sin z \in \mathscr K_{\attr}(2)$ is combined with the modular group $G\in\mathscr G_{\attr}(2)$.} In the general case, the construction is analogous.

\subsubsection{{Line complexes: general procedure.}} Let $f$ be a transcendental, entire or meromorphic, function with finitely many singular values, that is, $\# S(f) < \infty$. For such $f$, let $\gamma \subset \Cc$ be an oriented simple closed curve passing through all points in $S(f) \cup \{\infty\}$. In this way, $\Cc \sm \gamma$ consists of two regions, $H_i$ (interior) and $H_e$ (exterior). The curve $\gamma$ can be viewed as a graph with vertex set $S(f) \cup \{\infty\}$, and the faces being $H_{i/e}$. Consider the graph $\Gamma$ dual to $\gamma$; we can view $\Gamma$ as embedded into $\C$. The vertex set of $\Gamma$ consists of two points $p_{i/e}$ with $p_i \in H_i$ (internal vertex) and $p_e \in H_e$ (external vertex). The faces of $\Gamma$ correspond to the vertices of $\gamma$, and each edge of $\Gamma$ intersects exactly one edge of $\gamma$ exactly once. 

Let $\tG := f^{-1}(\Gamma) \subset \C$ be the full preimage of $\Gamma$. The graph $\tG$ is unbounded, each connected component $\widetilde V$ of $\C \sm \tG$ is mapped over the corresponding component $V$ of $\C \sm \Gamma$ and $f \colon \widetilde V \to V$ is a proper map of finite degree (if $V$ contains a critical value), or of infinite degree (if $V \subset \C$ is either unbounded or contains a finite asymptotic value); in the last case, $f \colon \widetilde V \to V$ is a universal covering. The graph $\tG$ is bipartite: each edge is mapped $1:1$ over an edge connecting $p_i$ and $p_e$. 

Following \cite{GO}, the \emph{(marked) line complex} of $f$ is the graph $\tG = f^{-1}(\Gamma)$ together with the labels of the singular values along the curve $\gamma$. For a fixed finite set of singular values, the marked line complex defines the {corresponding meromorphic map} $f$ uniquely up to pre-composition with an affine map (see \cite{GO} and \cite{LE}). Equivalently, if we do not fix the position of singular values, but fix their labeling along $\gamma$, the map $f$ is uniquely defined up to pre- and post-composition with quasiconformal homeomorphisms of $\C$ {preserving the labeling} (compare the discussion in \cite{LE}).     

\subsubsection{{Example: the sine map and the modular group.}} The map $f(z)=\frac{\pi}{2}\sin z$ lies in $\mathscr K_{\attr}(2)$, and $S(f) = \{-\pi/2, \pi/2\}$, $\Crit(f) = \{\pi/2 + k\pi, \,k\in \mathbb Z\}$. {Throughout this example, $G$ denotes the modular group, equipped with its factor Bowen--Series map $F_G$.} The line complex of $f$, with the choice of $\Gamma$ indicated, is shown in Figure~\ref{Fig:LCSine}. Assume that the marked super-attracting immediate basin is the Fatou component $U$ containing $-\pi/2$.

\begin{figure}[ht]
\captionsetup{width=0.98\linewidth}
    \centering
    \includegraphics[width=1.\linewidth]{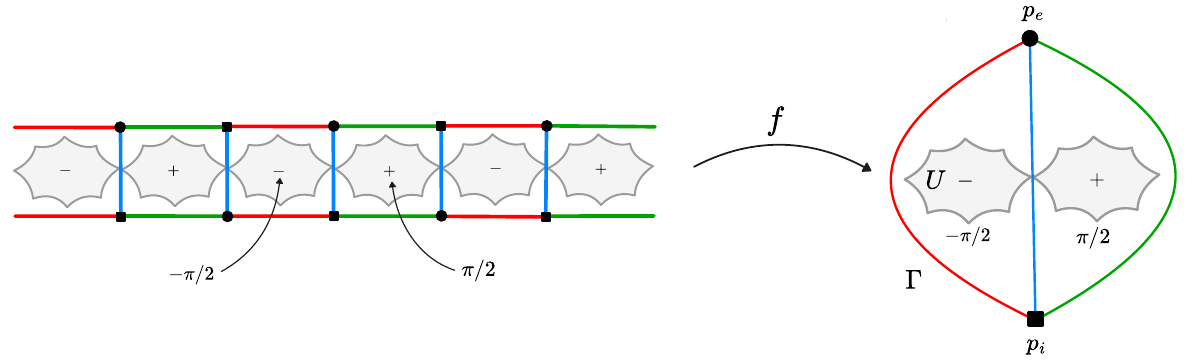}
    \caption{The line complex of $f(z) = \frac{\pi}{2}\sin(z)$. The graph $\Gamma$ consists of three edges (colored in green, blue, and red) connecting a pair of vertices $p_i$, $p_e$. The Fatou components containing the singular values $\pm \pi/2$ (on the right) and the critical points (on the left) are shown in gray.}
    \label{Fig:LCSine}
\end{figure}

In preparation for the construction of the map $g$, let us modify the graph $\Gamma$, and hence its pullback line complex, by introducing into $\Gamma$ two extra edges that surround the unique fixed point $0$ on $\partial U$; see Figure~\ref{Fig:LCext}.   

\begin{figure}[ht]
\captionsetup{width=0.98\linewidth}
    \centering
    \includegraphics[width=1.\linewidth]{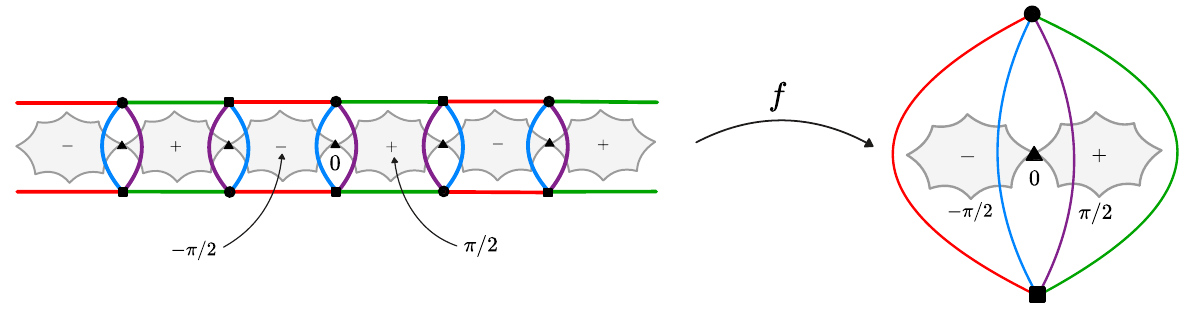}
    \caption{The modification of the line complex of $f(z) = \frac{\pi}{2}\sin(z)$.}
    \label{Fig:LCext}
\end{figure}

After the combination (given by Theorem~\ref{Thm:mating}), we obtain the function $g \colon \C \sm \widetilde T \to \C$, where the dynamics in the basin $U$ of $f$ was replaced by the dynamics of the factor Bowen--Series map $F_G \colon \overline{\disk} \sm T \to \overline{\disk}$ coming from the modular group (see the top of Figure~\ref{Fig:g}). {Here $\phi\colon U\to\disk$ is the Riemann map used in Theorem~\ref{Thm:mating}, $\psi$ is the David extension of the boundary conjugacy with $F_G$, and $\Psi$ is the David straightening map. We normalize $\phi$ and $\psi$ so that $\psi(\phi(-\pi/2))=0$. For the modular group, $n=3$ and $p=1$. Hence, the factor Bowen--Series map has one interior critical value and one singular boundary point. After an affine conjugation, we may assume that the critical values of $g$ are precisely $\pm\pi/2$. The remaining singular value of $\mathfrak h$ comes from the singular boundary point of the factor tile, and in this normalization it is $\Psi(0)$.} Thus $S(\mathfrak{h})=\{\frac{\pi}{2}, -\frac{\pi}{2}, \Psi(0)\}$.

\begin{figure}[ht]
\captionsetup{width=0.98\linewidth}
    \centering
    \includegraphics[width=1.\linewidth]{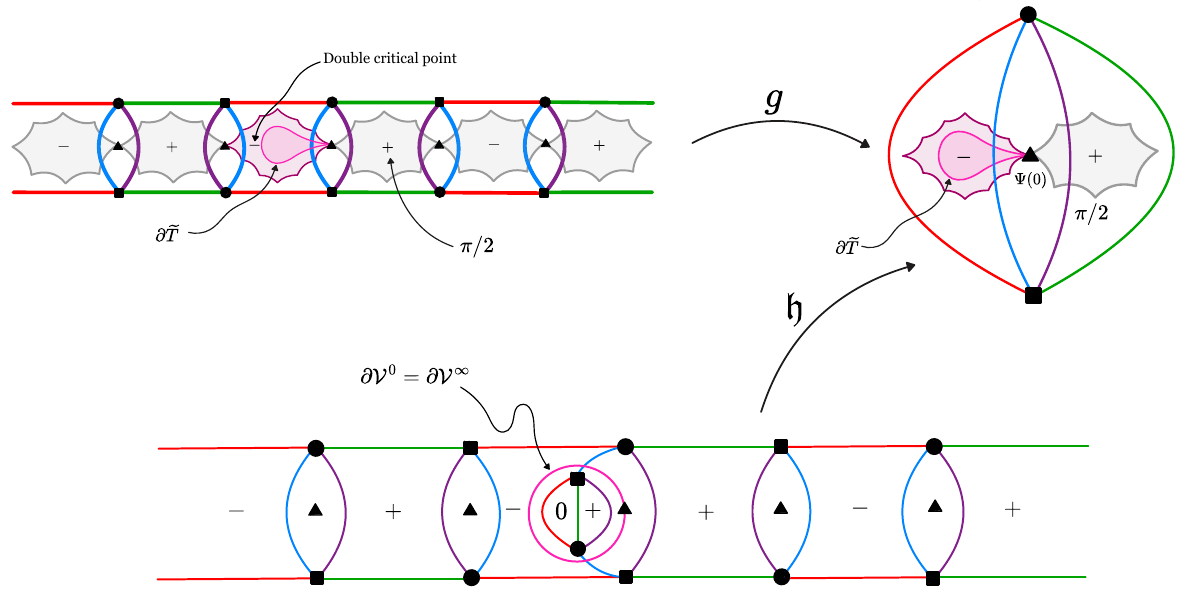}
    \caption{{Top:} the map $g$ obtained after the conformal combination of the map dynamics and the group dynamics. {Bottom: the domains $\mathcal V^\infty$ and $\mathcal V^0$ used, via Theorem~\ref{Thm:Analytic}, to construct the meromorphic map $\mathfrak h$.}}
    \label{Fig:g}
\end{figure}

The map $\mathfrak{h}$ is constructed as follows (compare Figure~\ref{Fig:g}, bottom); here we will use the conclusion of Theorem~\ref{Thm:Analytic}. In Figure~\ref{Fig:g} (bottom), the outside of the pink circle is the domain $\mathcal V^\infty$, the inside is the domain $\mathcal V^0$. {On $\mathcal V^\infty$, Theorem~\ref{Thm:Analytic} gives a conformal map $\psi_\infty\colon(\overline{\mathcal V^\infty},\infty)\to(\widehat{\C}\sm\tilde T,\infty)$ such that $\mathfrak h=g\circ\psi_\infty$ on $\overline{\mathcal V^\infty}\sm\{\infty\}$. Thus the part of the line complex of $\mathfrak h$ contained in $\mathcal V^\infty$ is conformally identified with the corresponding part of the line complex of $g$ in $\C\sm\tilde T$} (see the top left, outside of the tear drop-shaped pink region). In $\mathcal V^0$, the map $\mathfrak{h}|_{\mathcal V^0}$ sends $(\mathcal V^0, 0)$ conformally over $(\Cc \sm \tilde T, \infty)$. Therefore, inside the pink circle {$\mathfrak h^{-1}(\Gamma)\cap\mathcal V^0$ is a conformal copy of $\Gamma\sm\tilde T$} (i.e., the red, purple, and green edges, as well as two pieces of blue edges). The resulting line complex of $\mathfrak{h}$ is shown in Figure~\ref{Fig:LCh}.
\begin{rem}{
In this example, $\mathfrak h$ has three singular values, which is one more than $f$. In general, however, the number of new singular values is not determined by $d$ alone. For $G\in\mathscr G_{\attr}(d)$ one has $d+1=n(G)p(G)$: the number $p(G)$ counts the singular boundary points of the factor tile, while $n(G)$ controls the interior ramification. Thus the group side contributes the boundary singular values coming from these $p(G)$ points and, when $n(G)>1$, the interior critical value of the factor Bowen--Series map. These contributions need not be disjoint from the singular values inherited from $f$; the final number of singular values of $\mathfrak h$ is the cardinality of the resulting union, which is still a finite number as long as $\#S(f)$ is finite. In the modular group example, $p(G)=1$, and the interior critical value is normalized to coincide with $-\pi/2$.}
\end{rem}

\begin{figure}[ht]
\captionsetup{width=0.98\linewidth}
    \centering
    \includegraphics[width=1.\linewidth]{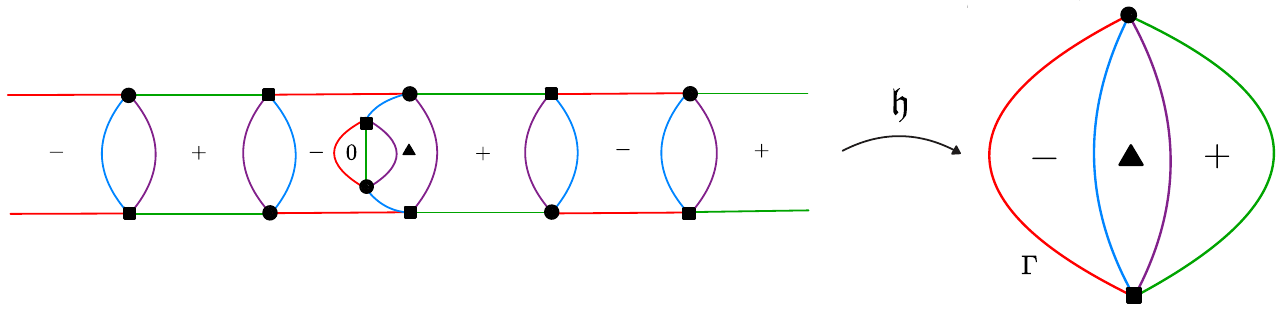}
    \caption{The line complex of the meromorphic map $\mathfrak{h}$ for the function $f(z) = \frac{\pi}{2}\sin(z) \in \mathscr{K}_{\text{attr}}(2)$ (with the marked immediate basing of $-\pi/2$) and the modular group $G \in \mathscr{G}_{\text{attr}}(2)$. The map $\mathfrak{h}$ has a unique simple pole at $0$ and three finite singular values (critical values) marked with $+, -, \blacktriangle$ on the right. The corresponding critical points are marked with the respective smaller symbols on the left. The component on the left marked with $0$ is mapped by $\mathfrak{h}$ univalently over the unique unbounded component of $\Cc \sm \Gamma$.}
    \label{Fig:LCh}
\end{figure}

\section{Combinations realized as correspondences}\label{corr_sec}

As in Section~\ref{analytic_description_sec}, let $g \colon\C\setminus\widetilde{T}\to\C$ be a conformal combination between a map $f\in\Kstar(d)$ and a group $G\in\Gstar(d)$ (more precisely, the factor Bowen--Series map $F_G$ associated with $G$). 
Further, let $\mathfrak{h}$ be the meromorphic function from Theorem~\ref{Thm:Analytic}.
Recall that the conformal combination $g$, the meromorphic function $\mathfrak{h}$, and the M{\"o}bius involution $\eta$ fit into following commutative diagram:
\begin{equation}
\begin{tikzcd}
	{\overline{\mathcal{V}^0}\,\sm\{0\}} && {\C \sm \tilde T} \\
	\\
	{\overline{\mathcal{V}^\infty}\,\sm\{\infty\}} && {\mathbb{C}}
	\arrow["{\mathfrak{h}}", from=1-1, to=1-3]
	\arrow["\eta",swap, from=1-1, to=3-1]
	\arrow["{\mathfrak{h}}"', from=3-1, to=3-3]
	\arrow["g"',swap, from=1-3, to=3-3]
\end{tikzcd}
\label{anal_char_cd}
\end{equation}
We define a holomorphic correspondence $\mathfrak{C}\subset \C_* \times \C_*$ of bi-degree $(\infty:\infty)$ as
\begin{equation}
(z, w) \in \mathfrak{C} \iff \frac{\mathfrak{h}(w)-\mathfrak{h}(\eta(z))}{w- \eta(z)}=0,
\label{corr_eq}
\end{equation}
where we use the notation $\C_*:=\C\,\setminus\{0\}$. 
The correspondence $\mathfrak{C}$ can also be thought of as a multi-valued map whose forward branches are given by
\begin{align}
z\mapsto w=\mathfrak{h}^{-1}(\mathfrak{h}(\eta(z)))\,\setminus\,\{\eta(z)\},
\label{forward_corr_eqn}
\end{align}
and backward branches are given by
\begin{align}
w\mapsto z=\eta(\mathfrak{h}^{-1}(\mathfrak{h}(w)))\,\setminus\,\{\eta(w)\}. 
\label{backward_corr_eqn}
\end{align}
Note that the local branches of the correspondence are defined by $\eta$ and the (local) deck transformations of $\mathfrak{h}$.

It is readily seen from~\eqref{anal_char_cd} and Definition~\ref{corr_eq} that for $z\in \overline{\cV^0}$,
\begin{equation}
 (z, w) \in \mathfrak{C} \iff \mathfrak{h}(w) = \mathfrak{h}(\eta(z)) = g(\mathfrak{h}(z)),\ \ w\neq \eta(z), 
 \label{in_V_0_equation}
\end{equation}
while for $z \in \cV^\infty$,
\begin{equation}
    (z,w)\in \mathfrak{C} \iff \mathfrak{h}(w) = \mathfrak{h}(\eta(z)) = g^{-1}(\mathfrak{h}(z)),\ \ w\neq \eta(z),
    \label{outside_V_0_equation}
\end{equation}
for some inverse branch of $g$. This shows that a forward branch of the multi-valued map sends $z$ to $w$ if and only if 
the conformal combination $g$ or one of its inverse branches sends $\mathfrak{h}(z)$ to $\mathfrak{h}(w)$. Succinctly, the meromorphic function $\mathfrak{h}$ lifts the forward/backward branches of $g$ to the multi-valued map $\mathfrak{C}$. It is thus not surprising that the correspondence $\mathfrak{C}$ can also be interpreted as a combination of $f$ and $G$.

\subsection{Dynamically invariant sets for $\mathfrak{C}$}\label{inv_partition_corr_subsec}
The dynamical plane of $g$ admits an invariant partition
$$
\C=\mathcal{E}\sqcup \cR,\quad \mathrm{where}\quad \mathcal{E}=\bigcup_{k=0}^\infty g^{-k}\big(\tilde U\big).
$$
The sets $\cE,\cR$ will be referred to as the \emph{expulsion set} and the \emph{resident set} of $g$ respectively. Indeed, all points in $\cE$ are eventually expelled from (the interior of) the domain of definition of $g$, while points of $\cR$ reside there perpetually. The common boundary of $\cE$ and $\cR$ (i.e., the set $\partial\mathcal{E}=\partial\cR$) is the \emph{Julia set} of $g$, and it is precisely the image of the Julia set of $f$ under the quasiconformal homeomorphism $\psi$.  

Define
$$
\pmb{\cE}:= \mathfrak{h}^{-1}(\cE),\quad \text{and} \quad \pmb{\cR} := \mathfrak{h}^{-1}(\cR)
$$
as the pull-backs of the expulsion and resident sets of the conformal mating $g$ under the meromorphic map $\mathfrak{h}$. The common boundary of $\pmb{\cE}$ and $\pmb{\cR}$ is called the \emph{Julia set} of the correspondence $\mathfrak{C}$.

\begin{prop}\label{invariance_prop}\upshape 
The following statements hold.

\begin{enumerate}
    \item $\eta(\pmb{\cE}) = \pmb{\cE}$, and $\eta(\pmb{\cR}) = \pmb{\cR}$.
    \item Given a point $(z, w) \in \mathfrak{C}$, we have $z\in \pmb{\cE}$ (respectively, $z\in \pmb{\cR}$) if and only if $w\in \pmb{\cE}$ (respectively, $w\in \pmb{\cR}$). In particular, both $\pmb{\cR}$ and $\pmb{\cE}$ are fully invariant under $\frakC$. 
\end{enumerate} 
\end{prop}
\begin{proof} 
We will prove the assertions for the set $\pmb{\cE}$. The case of $\pmb{\cR}$ is similar. Suppose that $z\in \pmb{\cE}$. If $z\in\overline{\cV^0}$, then by~\eqref{anal_char_cd} and the $g-$invariance of the expulsion set $\cE$ implies that $\mathfrak{h}(\eta(z))\in\cE$. Similarly, if $z\in\cV^\infty$, then by~\eqref{anal_char_cd} and the $g^{-1}-$invariance of the expulsion set $\cE$ implies that $\mathfrak{h}(\eta(z))\in\cE$. In both cases, $\eta(z)\in\pmb{\cE}$. As $\eta$ is an involution, it follows that $\eta(\pmb{\cE})=\pmb{\cE}$.

Now let $(z, w) \in \mathfrak{C}$. Suppose that $z\in\pmb{\cE}$. By the previous paragraph, we have that $\eta(z)\in\pmb{\cE}$. Since $\mathfrak{h}(w)=\mathfrak{h}(\eta(z))$, it follows that $w\in\pmb{\cE}$. Conversely, if $w\in\pmb{\cE}$, then the relation $\mathfrak{h}(w)=\mathfrak{h}(\eta(z))$ implies that $\eta(z)\in\pmb{\cE}$. By the $\eta-$invariance of $\pmb{\cE}$, we have that $z\in\pmb{\cE}$.
\end{proof}

Recall that all points in the expulsion set $\cE$ eventually map to $\tilde U$ under $g$, and then get expelled to $\overline{\tilde T}$. Since the action of $g$ on $\tilde U$ is conformally conjugate to the map $F_G$, the set $\mathfrak{h}^{-1}(\tilde U)$ plays a special role for the correspondence $\mathfrak{C}$.
Let $\{\tilde U_\lambda\}$ be the components of $\cE$ such that $g(\tilde U_\lambda) = \tilde U$, and $\tilde U_\lambda\neq \tilde U$. By~\eqref{anal_char_cd} and the properties of $\mathfrak{h}$ listed in Theorem~\ref{Thm:Analytic}, we have that
\begin{equation*}
\begin{split}
\mathfrak{h}^{-1}(\tilde U) = &\Big(\big(\mathfrak{h}\big|_{\overline{\cV^0}}\big)^{-1}(\tilde U\,\sm\tilde T)\Big)\bigcup\Big(\big(\mathfrak{h}\big|_{\overline{\cV^\infty}}\big)^{-1}(\tilde U)\Big)\\
=&\Big(\big(\mathfrak{h}\big|_{\overline{\cV^0}}\big)^{-1}(\tilde U\,\sm\tilde T)\Big)\bigcup\Bigg(\eta\Big(\big(\mathfrak{h}\big|_{\overline{\cV^0}}\big)^{-1}\Big(g^{-1}(\tilde U)\Big)\Big)\Bigg)\\
=&\Bigg(\big(\mathfrak{h}\big|_{\overline{\cV^0}}\big)^{-1}(\tilde U\,\sm\tilde T)\Bigg)\bigcup\Bigg(\eta\Big(\big(\mathfrak{h}\big|_{\overline{\cV^0}}\big)^{-1}(\tilde U\,\sm\tilde T)\Big)\Bigg)\bigsqcup\Bigg(\bigsqcup_{\lambda} \eta\Big(\big(\mathfrak{h}\big|_{\cV^0}\big)^{-1}(\tilde U_\lambda)\Big)\Bigg).    
\end{split}    
\end{equation*}
We need to focus on the subset
\begin{equation}
\pmb{\mathcal{W}}:= \Bigg(\big(\mathfrak{h}\big|_{\overline{\cV^0}}\big)^{-1}(\tilde U\,\sm\tilde T)\Bigg)\bigcup\Bigg(\eta\Big(\big(\mathfrak{h}\big|_{\overline{\cV^0}}\big)^{-1}(\tilde U\,\sm\tilde T)\Big)\Bigg)\subset \mathfrak{h}^{-1}(\tilde U)
\label{tiling_comp_eq}
\end{equation}
to extract a group action from the dynamics of the correspondence.

In the next two subsections, we will prove the following theorem.
 
\begin{thm}[The final correspondence] \label{corr_mating_thm}
The correspondence $\mathfrak{C}$ is a mating of the group $G$ and the entire function $f$ in the following sense.
\begin{enumerate}[leftmargin=8mm]
    \item The correspondence $\mathfrak{C}$ has a forward branch (respectively, a backward branch) carrying $\pmb{\cR}\cap \overline{\cV^0}$ (respectively, $\pmb{\cR}\cap \overline{\cV^\infty}$) onto itself, and this branch restricted to the corresponding domains is conformally conjugate to 
    \[
    f \colon \C \,\sm \bigcup_{k=0}^\infty f^{-k}(U) \to \C \,\sm \bigcup_{k=0}^\infty f^{-k}(U).
    \]
    
    \item The branches of $\mathfrak{C}$ preserving $\pmb{\cW}$ generate a group of conformal automorphisms $\pmb{G}_{\pmb{\cW}}$ acting properly discontinuously on $\pmb{\cW}$. The subgroup of $\pmb{G}_{\pmb{\cW}}$ preserving any component of $\pmb{\cW}$ is conformally conjugate to the action of $G$ on $\D$. In particular, the quotient of $\pmb{\cW}$ under the action of the group $\pmb{G}_{\pmb{\cW}}$ is biholomorphic to~$\mathbb{D}/G$.
\end{enumerate}
\end{thm}

\subsection{Appearance of the entire function $f$ in the correspondence $\mathfrak{C}$}\label{map_like_subsec}

\begin{proof}[Proof of Theorem~\ref{corr_mating_thm} (Part (1))]
By Proposition~\ref{invariance_prop}, we have that $\eta(\pmb{\cR}\cap\overline{\cV^0})=\pmb{\cR}\cap\overline{\cV^\infty}$. Hence, by definition of the correspondence,
    $$
    \Big(\mathfrak{h}\big|_{\overline
    {\cV^0}}\Big)^{-1}\circ \mathfrak{h}\circ \eta: \pmb{\cR}\cap \overline{\cV^0}\rightarrow\pmb{\cR}\cap \overline{\cV^0}
    $$
    is a forward branch of $\mathfrak{C}$. Evidently, the univalent map $\mathfrak{h}: \overline{\cV^0}\rightarrow \C\,\sm\tilde T$ carries $\pmb{\cR}\cap \overline{\cV^0}$ onto $\cR$, and conjugates the above forward branch to $g \colon \C \,\sm \bigcup_{k=0}^\infty g^{-k}(\tilde U) \to \C \,\sm \bigcup_{k=0}^\infty g^{-k}(\tilde U)$. By Theorem~\ref{Thm:mating}, the map $g \colon \C \,\sm \bigcup_{k=0}^\infty g^{-k}(\tilde U) \to \C \,\sm \bigcup_{k=0}^\infty g^{-k}(\tilde U)$ is conformally conjugate to $f \colon \C \,\sm \bigcup_{k=0}^\infty f^{-k}(U) \to \C \,\sm \bigcup_{k=0}^\infty f^{-k}(U)$. This establishes the statement about the forward branch of $\mathfrak{C}$.

    Note also that the involution $\eta$ restricts to a conformal conjugacy between the backward branch  
    $$
    \eta\circ \Big(\mathfrak{h}\big|_{\cV^0}\Big)^{-1}\circ \mathfrak{h}: \pmb{\cR}\cap\overline{\cV^\infty} \rightarrow \pmb{\cR}\cap\overline{\cV^\infty}
    $$
    of $\mathfrak{C}$ and the above forward branch of $\mathfrak{C}$. The result now follows.
\end{proof}

\subsection{Appearance of the group $G$ in the correspondence $\mathfrak{C}$}\label{group_like_subsec} 

The welding construction performed in Theorem~\ref{Thm:Analytic} and the definition of the set $\pmb{\cW}$ given in~\eqref{tiling_comp_eq} show that $\pmb{\cW}$ is an open set obtained by gluing two copies of $\tilde U\,\sm\tilde T$ along $\partial\tilde T$ via the involution $g \colon \partial\tilde T\to\partial\tilde T$. This fact can be used to see that $\pmb{\cW}$ is a disjoint union of $p$ simply connected domains (cf. \cite[Lemma~5.2, Figure~12]{MM}; see also the top-right in Figure~\ref{corr_group_action_fig}). The following lemma describes the topological structure of $\pmb{\cW}$. 
We recall the notation $n(\Sigma)$ from Definition~\ref{Def:OrbifoldClassForFBS}. 

\begin{lemma}\label{lifted_tiling_top_lem}
\noindent\begin{enumerate}[leftmargin=8mm]
\item The set $\displaystyle\clo{\pmb{\cW}}=\bigcup_{i=0}^{p-1} \overline{\pmb{\cW}_i}$ is connected, where each $\pmb{\cW}_i$ is a simply connected domain (in fact, a Jordan domain if $p\neq 1$) that is mapped as a degree $n=n(\Sigma)$ branched covering onto $\tilde U$ by $\mathfrak{h}$.
\item  If $p\neq 1$, then $\overline{\pmb{\cW}_i}\cap \overline{\pmb{\cW}}_{i+1}$ is a single point.
\item If $p\neq 1$, then $\overline{\pmb{\cW}_i}\cap \overline{\pmb{\cW}_j}=\emptyset$ if $\vert j-i\vert\neq1$, $i,j\in\Z/p\Z$.
\item If $p=1$, then the boundary $\pmb{\cW}$ is topologically a figure-eight curve.
\end{enumerate}
\end{lemma}
\begin{proof}
It follows from the construction of the conformal combination $g$ that $\clo\tilde U\,\setminus\, \tilde T$ is a connected set. In fact, it is the union of $p$ closed topological disks $Y_0,\cdots,Y_{p-1}$ (except in the case $p=1$, when $Y_0=\clo\tilde U\,\setminus\, \tilde T$ is the closure of a simply connected domain such that there is a unique boundary point with two accesses from the domain and all other boundary points have a unique access from the domain)  such that, after possibly renumbering, we have the following properties:
\begin{itemize}[leftmargin=8mm]
\item $\partial Y_i\cap\partial \tilde T$ is a non-singular real-analytic curve that is a Jordan curve for $p=1$ and an arc for $p\neq 1$,
\item $Y_i\cap Y_{i+1}$ is a single point on $\partial\tilde U$, and
\item $Y_i\cap Y_j=\emptyset$ if $\vert j-i\vert\neq1$, where $i,j\in\Z/p\Z$.
\end{itemize}
The inverse branch $(\mathfrak{h}\vert_{\overline{\cV^0}})^{-1}$ maps each $Y_i$ to a closed topological disk $\pmb{Y}_i\subset \overline{\cV^0}$ (except for $p=1$, when $\pmb{Y}_0$ is the closure of a simply connected domain such that there is a unique boundary point with two accesses from the domain and all other boundary points have a unique access from the domain). By construction, the boundary $\partial\cV^0$ is contained in the boundary of $\left(\mathfrak{h}\vert_{\overline{\cV^0}}\right)^{-1}(\clo \tilde U\,\sm\,\tilde T)=\bigcup_{i=1}^p \pmb{Y}_i$. More precisely, $\bigcup_{i=1}^p \pmb{Y}_i$ is a connected set  containing a relative neighborhood of $\partial\cV^0\setminus\{x_1,\cdots,x_p\}$ in $\overline{\cV^0}$, where $\{x_1,\cdots,x_p\}\subset\partial\cV^0$ is $\eta-$invariant.
By definition of $\pmb{\cW}$, we have that
\begin{equation*}
\clo \pmb{\mathcal{W}} = \Bigg(\big(\mathfrak{h}\big|_{\overline{\cV^0}}\big)^{-1}(\clo\tilde U\,\sm\tilde T)\Bigg)\bigcup\Bigg(\eta\Big(\big(\mathfrak{h}\big|_{\overline{\cV^0}}\big)^{-1}(\clo\tilde U\,\sm\tilde T)\Big)\Bigg)
= \Big(\bigcup_{i=1}^p \pmb{Y}_i\Big)\bigcup\Bigg(\eta\Big(\bigcup_{i=1}^p \pmb{Y}_i\Big)\Bigg).
\end{equation*}
For $i\in\{0,\cdots,p-1\}$, we define $\pmb{\cW}_i$ to be 
$$
\pmb{\cW}_i:=\Int{\Bigg(\overline{\Int{\pmb{Y}_i}\bigsqcup\eta(\Int{\pmb{Y}_{i'}})}\Bigg)},
$$ 
where $i'\in\{0,\cdots,p-1\}$ is the unique index such that $g(\partial Y_i\cap\partial\tilde T)=\partial Y_{i'}\cap\partial\tilde T$.
The desired topological properties of $\pmb{\cW}_i$ now follows from the above description of $\pmb{\cW}_i$ and the $\eta-$invariance of $\{x_1,\cdots,x_p\}$.
\end{proof}

Having studied the topology of $\pmb{\cW}$, we now analyze the mapping structure of $\mathfrak{h}$ on $\pmb{\cW}$.

\begin{lemma}\label{h_branched_cover_lem}
For each $i\in\{0,\cdots,p-1\}$, the map $\mathfrak{h}:\pmb{\cW}_i\to\tilde U$ is a degree $n$ branched covering with a unique critical point of multiplicity $n-1$. The $\mathfrak{h}-$image of this critical point is the unique critical value of $g:\tilde U\,\sm\tilde T\to\tilde U$.  
\end{lemma}

\begin{proof}
By the proof of Lemma~\ref{lifted_tiling_top_lem}, we have that $\pmb{\cW}_i:=\Int{\Bigg(\overline{\Int{\pmb{Y}_i}\bigsqcup\eta(\Int{\pmb{Y}_{i'}})}\Bigg)}$, where $i'\in\{0,\cdots,p-1\}$ is the unique index such that $g(\partial Y_i\cap\partial\tilde T)=\partial Y_{i'}\cap\partial\tilde T$. By definition, $\mathfrak{h}$ carries $\Int{\pmb{Y}_i}$ univalently onto $\Int{Y_i}$, and maps the common boundary of $\Int{\pmb{Y}_i},\eta\big(\Int{\pmb{Y}_{i'}}\big)$ onto $\partial Y_i\cap\partial\tilde T$. On the other hand, $\mathfrak{h}$ acts as $g\circ\mathfrak{h}\vert_{\cV^0}\circ\eta$ on $\eta\big(\Int{(\pmb{Y}_{i'}})\big)$. Therefore, $\mathfrak{h}\big(\eta(\Int{\pmb{Y}_{i'}})\big)=g(\Int{Y_{i'}})$. As $g(\partial Y_i\cap\partial\tilde T)=\partial Y_{i'}\cap\partial\tilde T$ and $g \colon \tilde U\,\sm\tilde T \to \tilde U$ is conformally conjugate to the factor Bowen--Series map $F_G \colon\D\,\sm T_G\to\D$, it follows from the mapping properties of $F_G$ (cf. \cite[Proposition~2.7]{MM}) that $g\vert_{\Int{Y_{i'}}}$ is a degree $n$ branched covering over $\tilde U\,\sm Y_i$ and a degree $n-1$ branched covering over $\Big(\Int{Y_i}\bigcup\big(\partial Y_{i}\cap\partial\tilde T\big)\Big)$.  This proves that $\mathfrak{h}:\pmb{\cW}_i\to\tilde U$ is a degree $n$ branched covering. The unicriticality of this branched covering also follows from the above discussion and \cite[Proposition~2.7]{MM}. The last statement is a consequence of the fact that all critical points of $F_G$ are mapped to the same critical value.
\end{proof}

\begin{rem}\label{marking_rem}
If $n > 1$, for each $i$ we denote the unique critical point of $\mathfrak{h}\vert_{\pmb{\cW}_i}$ from the lemma above as $\pmb{p}_i \in \pmb{\mathcal W}_i$.

If $n=1$, then $\mathfrak{h}:\pmb{\cW}_i\to\tilde U$ is a conformal isomorphism; in 
particular, it has no critical points. In this case, we choose any point in $\tilde U$ (for definiteness, the point that corresponds to the origin under the conformal conjugacy between $g:\tilde U\,\sm\tilde T\to\tilde U$ and $F_G:\D\,\sm T\to \D$), and refer to its preimage under $\mathfrak{h}\vert_{\pmb{\cW}_i}$ as the marked point of $\pmb{\cW}_i$ and denote it by $\pmb{p}_i$.
\end{rem}

This is a good moment to point out that, unlike in the dynamical plane of the conformal combination $g$, no point is expelled in the iteration of the correspondence $\mathfrak{C}$. This is particularly relevant for the set $\pmb{\cW}$, where the dynamics of $\mathfrak{C}$ will mimic the full structure of the group $G$. 
Specifically, the open set $\pmb{\cW}$ is $\eta-$invariant.
Further, by Lemma~\ref{h_branched_cover_lem}, the map $\mathfrak{h} \colon \pmb{\cW}\to\tilde U$ is a degree $np$ branched covering. Since the branches of the correspondence $\mathfrak{C}$ are given by compositions of $\eta$ and the local deck transformations of $\mathfrak{h}$, we will now proceed to study the deck transformations of the branched covering $\mathfrak{h} \colon\pmb{\cW}\to\tilde U$ (these deck transformations preserve $\pmb{\cW}$ by definition) allowing us to completely describe the branches of $\mathfrak{C}$ preserving $\pmb{\cW}$.

\begin{lemma}\label{deck_tau_lem}
    There exists a conformal automorphism $\tau\colon\pmb{\cW}\to\pmb{\cW}$ such that
    $$
    \tau^{np}\equiv \text{id}, \hspace{2mm} and \hspace{2mm}\Big(\mathfrak{h}\big|_{\pmb{\cW}}\Big)^{-1}\bigg(\Big(\mathfrak{h}\big|_{\pmb{\cW}}\Big)(z)\bigg) = \{z, \tau(z), \cdots, \tau^{np-1}(z)\} \hspace{2mm}\forall z\in \pmb{\cW}. 
    $$
Hence, the conformal automorphisms $\tau\circ\eta,\cdots,\tau^{np-1}\circ\eta \colon \pmb{\cW}\to\,\pmb{\cW}$ define the forward branches of $\mathfrak{C}$ preserving $\pmb{\cW}$.
\end{lemma}

\begin{proof}
    The proof is similar to that of \cite[Proposition~5.4]{MM}. However, as the current setting is different from the one considered in \cite{MM}, we supply a proof. 
    
    Let $\displaystyle\pmb{\Phi}: \mathbb{D}\times\Z/p\Z\rightarrow \bigsqcup_{j\in\Z/p\Z}\pmb{\cW}_i$ be a conformal isomorphism such that for each $j\in \Z/p\Z$, we have $\pmb{\Phi}(\D\times\{j\})=\pmb{\cW}_j$, and $\pmb{\Phi}(0, j)$ is the marked point $\pmb{p}_j\in\pmb{\cW}_j$ (see Remark~\ref{marking_rem}). Further, let $\mathfrak{X}:\D\to\tilde U$ be the conformal isomorphism that conjugates $F_G\colon\D\,\sm T_G\to \D$ to $g\colon\tilde U\,\sm\tilde T\to\tilde U$.
     Then, after possibly pre-composing $\pmb{\Phi}$ with a rotation on each $\mathbb{D}\times\{j\}$, the map $\widetilde{\mathfrak{h}}:= \mathfrak{X}^{-1}\circ \mathfrak{h}\circ \pmb{\Phi}:\mathbb{D}\times\Z/p\Z\rightarrow \D$ takes the form
\begin{equation*}
    (w, j) \mapsto w^{n}, \hspace{2mm}w\in \mathbb{D}, \hspace{1mm}j\in \Z/p\Z.
\end{equation*}
It is readily seen that the map 
\begin{equation*}
\begin{split}
\quad \tilde \tau \colon \mathbb{D} \times \Z/p\Z &\rightarrow \mathbb{D} \times \Z/p\Z \\
\tilde \tau(w, j) &:= 
\begin{cases*}
    (w, j+1), & for $j \in \{0, \dots, p-2\}$, \\
    (e^{\frac{2i\pi}{n}}w, 0), & for $j = p-1$,
\end{cases*}
\end{split}
\end{equation*}
is an order $np$ conformal automorphism satisfying the condition
 $$
\widetilde{\mathfrak{h}}^{-1}(\widetilde{\mathfrak{h}}(w, j)) = \{(w, j), \widetilde{\tau}(w, j), \cdots, \widetilde{\tau}^{np-1}(w, j)\},\quad \forall\ (w, j) \in \mathbb{D}\times\Z/p\Z.
$$
The required automorphism $\tau$ is given by $\pmb{\Phi}\circ\tilde \tau\circ\pmb{\Phi}^{-1}$. The last statement follows from the above discussion and~\eqref{forward_corr_eqn}.
\end{proof}

\begin{cor}\label{corr_group_action_cor}
The branches of $\mathfrak{C}$ preserving $\pmb{\cW}$ generate the group $\langle\eta, \tau\rangle$ of conformal automorphisms of $\pmb{\cW}$.    
\end{cor}
\begin{proof}
    According to Lemma~\ref{deck_tau_lem}, the forward branches of $\mathfrak{C}$ preserving $\pmb{\cW}$ extend as the conformal automorphisms $\tau\circ\eta, \cdots, \tau^{np-1}\circ\eta:\pmb{\cW}\to\pmb{\cW}$. Now, the relation $\tau = (\tau^2\circ\eta)\circ(\tau\circ\eta)^{-1}$ implies that both $\tau,\eta$ can be written as compositions of $\left(\tau^j\circ\eta\right)^{\pm 1}$, $j\in\{1,\cdots,np-1\}$. Hence, the subgroup of $\Aut_{\mathrm{conf}}(\pmb{\cW})$ generated by $\tau\circ\eta, \cdots, \tau^{np-1}\circ\eta$ equals the one generated by $\eta, \tau$.
    The corollary follows.
\end{proof}

\begin{prop}\label{orbifold_isom_prop}
The following properties hold.
\noindent\begin{itemize}[leftmargin=8mm]
    \item The subgroup $\pmb{G}_0\leq\pmb{G}_{\pmb{\cW}}$ preserving the component $\pmb{\cW}_0$ of $\pmb{\cW}$ is conformally conjugate to the action of $G\vert_{\D}$.

    \item The group $\langle\eta, \tau\rangle$ acts properly discontinuously on $\pmb{\cW}$.

    \item The quotient orbifold $\pmb{\cW}/\langle\eta, \tau\rangle$ is biholomorphic to $\mathbb{D}/G$.
\end{itemize}
\end{prop}
\begin{proof}
The proof is analogous to that of \cite[Proposition~5.7]{MM} or \cite[Proposition~6.8]{MV25}. For completeness, we work out the details using the notation of the present paper.

Let $\pmb{\Phi}$ be as in the proof of Lemma~\ref{deck_tau_lem}. For brevity, we will use the notation $\pmb{\Phi}_j\equiv\pmb{\Phi}\vert_{\D\times\{j\}}:\D\to\D$, $j\in\Z/p\Z$. Then, $\pmb{\cW}_j=\pmb{\Phi}_j(\mathbb{D})$, $j\in\Z/p\Z$. 
Lemma~\ref{deck_tau_lem} also yields the following commutative diagram:
\[
 \begin{tikzcd}
\left(\D,0\right)    \arrow{d}[swap]{w\mapsto w^n} \arrow{r}{\pmb{\Phi}_j} &  \left(\pmb{\cW}_j,\pmb{p}_j\right)
\arrow{d}{\mathfrak{h}} \\
\left(\D,0\right)   \arrow{r}{\mathfrak{X}}  & \left(\tilde U,\mathfrak{h}(\pmb{p}_j)\right)
\end{tikzcd}
\]
where $\pmb{p}_j$ is the marked point in $\pmb{\cW}_j$ (see Remark~\ref{marking_rem} for the naming convention, and see Figure~\ref{corr_group_action_fig} that illustrates the proof). Note also that by the proof of Lemma~\ref{deck_tau_lem}, we have that
\begin{equation}
\pmb{\Phi}_{j+1}=\tau\circ\pmb{\Phi}_j,\quad j\in\{0,\cdots,p-2\}.
\label{phi_relation_eqn}
\end{equation}

Let us consider 
$$
\pmb{G}_0:=\{\kappa\in\langle\eta,\tau\rangle: \kappa(\pmb{\cW}_0)=\pmb{\cW}_0\};
$$
the stabilizer of $\pmb{\cW}_0$ in $\langle\eta,\tau\rangle$. Note that the cyclic group $\langle\tau\rangle$ acts transitively on the components of $\pmb{\cW}$. Thus, to establish the proposition, it is enough to prove that $\pmb{G}_0$ acts properly discontinuously on $\pmb{\cW}_0$ and that $\pmb{\cW}_0/\pmb{G}_0$ is biholomorphic to $\D/G$.

To this end, first observe that $\tau^{p}$ restricts to an order $n$ element of $\pmb{G}_0$ that fixes $\pmb{p}_0$, and has derivative $\omega =e^{2i\pi/n}$ at $\pmb{p}_0$. Hence, $\pmb{\Phi}_0$ conjugates $M_\omega(z)=\omega z$ to $\tau^p$.
In order to extend $\tau^p$ to a generating set of $\pmb{G}_0$, we will construct a fundamental domain and associated side-pairing transformations for the $\pmb{G}_0-$action on $\pmb{\cW}_0$. 

For this, let us
set $\pmb{T}_j:=\mathfrak{h}^{-1}(\tilde T)\cap\pmb{\cW}_j$, $j\in\Z/p\Z$. For $n=1$, the maps in the vertical arrows in the above commutative diagram are conformal. On the other hand, if $n\geq 3$, then the maps in the vertical arrows are degree $n$ branched coverings with unique critical points of multiplicity $n-1$ at $0$, $\pmb{p}_j$, respectively. It now follows from the definition of (factor) Bowen--Series maps that $\pmb{\Phi}_j^{-1}\big(\pmb{T}_j\big)=\Int{\Pi}$ (see Section~\ref{SSec:FactorBowenSeriesMaps}).

Recall that $\Gamma$ is the index $n$ subgroup of $G$ such that $G=\Gamma\rtimes \langle M_\omega\rangle$ (cf.\ Section~\ref{SSec:FactorBowenSeriesMaps}). The polygon $\Pi$ is a closed fundamental domain for $\Gamma$, and
$$
\Pi_G:=\{z\in\Pi:0\leq \arg{z}\leq 2\pi/n\}
$$
is a closed fundamental domain for $G$. Define $\widehat{\pmb{T}}_j:=\pmb{\Phi}_j(\Pi_G)\subset\pmb{T}_j$. By the commutative diagram, $\mathfrak{h}$ is univalent on $\Int{\widehat{\pmb{T}}_j}$, and maps the two geodesics on $\partial \widehat{\pmb{T}}_j$ emanating from $\pmb{p}_0$ to the arc $\mathfrak{X}([0,1))$. Further, after possibly normalizing $\pmb{\Phi}_j$ (specifically, pre-composing it with some iterate of $M_\omega$), we can assume that the hyperbolic geodesic $\partial\cV^0\cap\pmb{\cW}_j$ (of $\pmb{\cW}_j$) is one of the sides of $\partial\widehat{\pmb{T}}_j$.

Note that each component $\pmb{\cW}_j$ of $\pmb{\cW}$ is preserved by some $\tau^r\circ\eta$. The collection of these maps, conjugated by suitable powers of $\tau$, produce elements of $\pmb{G}_0$ that act as side-pairing transformations of the boundary of the $np-$gon $\pmb{T}_0$. Combined with the map $\tau^p$, (a subset of) these maps pair the sides of $\widehat{\pmb{T}}_0$. Finally, Relation~\eqref{phi_relation_eqn} and the above normalization of the conformal maps $\pmb{\Phi}_j$ can be used to see that $\pmb{\Phi}_0^{-1}$ conjugates these side-pairing transformations for $\widehat{\pmb{T}}_0$ to the side-pairing transformations for the fundamental domain $\Pi_G$ of $G$ (i.e., to the marked generators of $G$). In particular, $\widehat{\pmb{T}}_0$ is a fundamental domain for the action of the group $\pmb{G}_0$ on $\pmb{\cW}_0$, and $\pmb{\Phi}_0^{-1}$ conjugates $\pmb{\pmb{G}}_0\vert_{\pmb{\cW}_0}$ to $G\vert_{\D}$. 
It follows that the group $\langle\eta, \tau\rangle$ acts properly discontinuously on $\pmb{\cW}$, and that the conformal map $\pmb{\Phi}_0:\D\to\pmb{\cW}_0$ induces a biholomorphism between the orbifolds $\D/G$ and $\pmb{\cW}_0/\pmb{G}_0$.    
\end{proof}

\subsubsection{An illustration of the construction of Proposition~\ref{orbifold_isom_prop}}

Let $G$ be a Fuchsian group such that $\Sigma\cong\D/G$ is a genus zero (finite area) hyperbolic orbifold with two punctures, an order $2$ orbifold point, and an order $4$ orbifold point. Then, in the notation of Section~\ref{SSec:FactorBowenSeriesMaps}, the cyclic cover $\tilde\Sigma$ of $\Sigma$ is a genus zero (finite area) hyperbolic orbifold with five punctures, and four order $2$ orbifold points. The surface $\tilde\Sigma$ is a fourfold cover of $\Sigma$. Let $\Gamma$ be the index four subgroup of $G$ such that $G=\Gamma\rtimes\langle M_i\rangle$, where $M_i(z)=iz$. In this case, $n=4$ and $p=3$.  Figure~\ref{corr_group_action_fig} (top left) shows the ideal $12-$gon $\Pi$ that is a closed fundamental domain for $\Gamma$. Its subset
$$
\Pi_G:=\{z\in\Pi:0\leq \arg z\leq \pi/2\}
$$
is a closed fundamental domain for $G$. The region $\Pi_G$ is shaded in gray. Note that $\partial\Pi_G$ consists of three bi-infinite geodesics and a pair of geodesic rays emanating from the origin. The map $w\mapsto w^4$ semi-conjugates the Bowen--Series map $A_\Gamma^{\mathrm{BS}}:\overline{\D}\,\sm\Int(\Pi)\to\overline{\D}$ (associated with the fundamental domain $\Pi$ of the group $\Gamma$) to the factor Bowen--Series map $F_G:\overline{\D}\,\sm T_G\to\overline{\D}$, where $T_G$ is the image of $\Int{\Pi}$ under $w\mapsto w^4$ (the set $T$ is shaded in gray in the bottom left picture of Figure~\ref{corr_group_action_fig}).
\begin{figure}[h!]
\captionsetup{width=0.98\linewidth}
\begin{tikzpicture}
\node[anchor=south west,inner sep=0] at (0,0) {\includegraphics[width=1\textwidth]{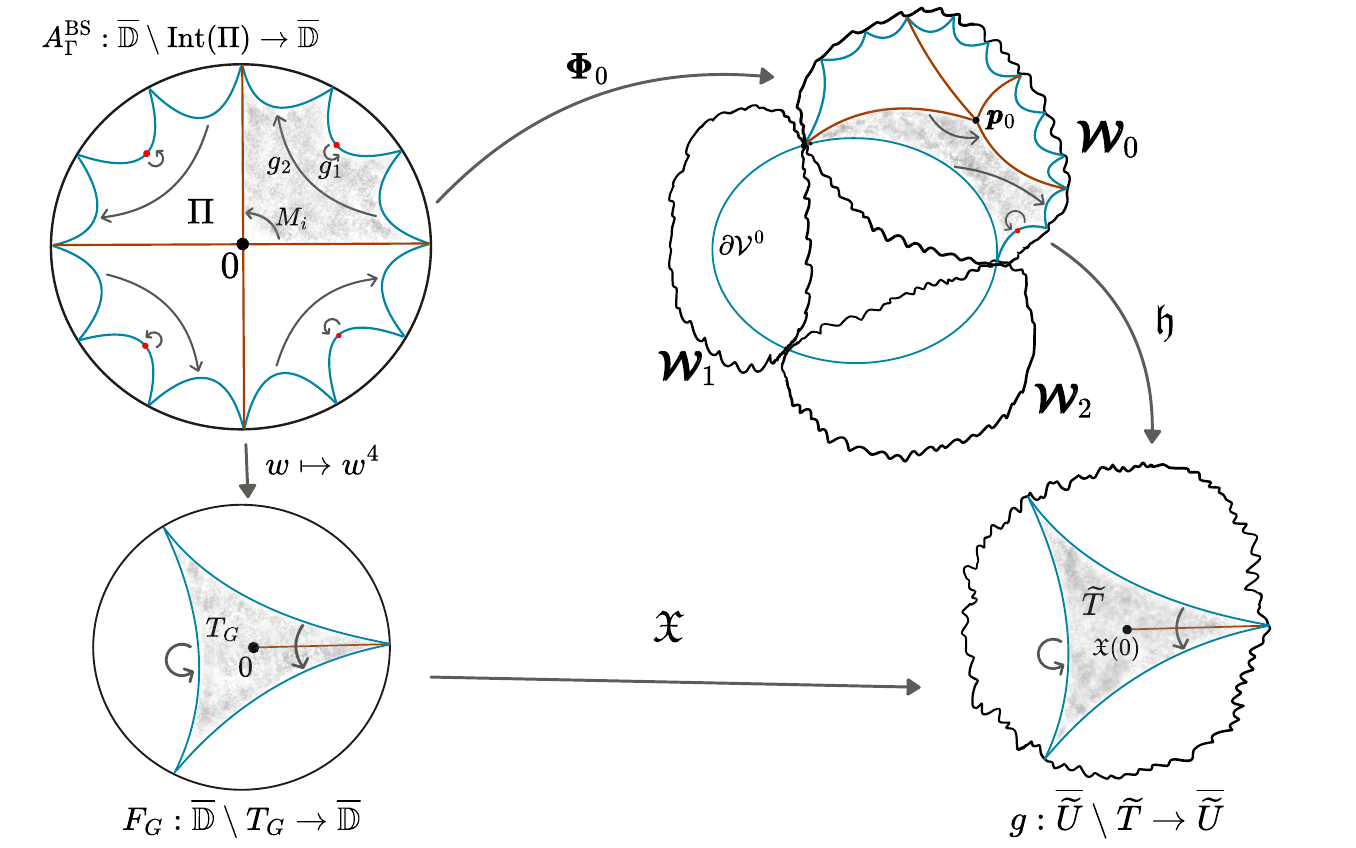}};
\end{tikzpicture}
\caption{Illustrated is the proof of Proposition~\ref{orbifold_isom_prop}.}
\label{corr_group_action_fig}
\end{figure}
In bottom right of Figure~\ref{corr_group_action_fig}, we depict the restriction $g:\clo\tilde U\,\sm\tilde T\to\clo\tilde U$ of the conformal combination $g$. The conformal map $\mathfrak{X}:\overline{\D}\to\clo\tilde U$ conjugates $F_G$ to $g$. Here, $\tilde T=\mathfrak{X}(T)$ is shaded in gray. In top right of Figure~\ref{corr_group_action_fig}, we show a part of the dynamical plane of the correspondence $\mathfrak{C}$. Since $p=3$, we have that $\pmb{\cW}=\pmb{\cW}_0\sqcup\pmb{\cW}_1\sqcup\pmb{\cW}_2$, where each $\pmb{\cW}_j$ maps as a degree $4$ branched cover onto $\clo\tilde U$ under $\mathfrak{h}$. Observe that $\pmb{\cW}$ is obtained by gluing two copies of $\tilde U\,\sm\tilde T$ along $\partial\tilde T$ via the identification map $g\colon\partial\tilde T\to\partial\tilde T$. The cyan oval curve in this picture is the Jordan curve $\partial\cV^0=\partial\cV^\infty$, its bounded (respectively, unbounded) complementary component in $\widehat{\C}$ is $\cV^0$ (respectively, $\cV^\infty$). The ideal $12-$gon shown in $\pmb{\cW}_0$ is $\pmb{T}_0=\mathfrak{h}^{-1}(\tilde T)\cap\pmb{\cW}_0$. The map $\pmb{\Phi}_0\colon(\D,0)\to(\pmb{\cW}_0,\pmb{p}_0)$ carries $\Int{\Pi}$ conformally onto $\pmb{T}_0$, where $\pmb{p}_0$ is the unique critical point (of multiplicity $3$) of $\mathfrak{h}\colon\pmb{\cW}_0\to\tilde U$. Further, $\pmb{\Phi}_0$ maps gray subset $\Pi_G$ of $\Pi$ conformally onto the gray subset $\widehat{\pmb{T}}_0$ of $\pmb{T}_0$.

We now construct a generating set for $\pmb{G}_0$.
\begin{enumerate}[leftmargin=8mm]
\item The map $\tau^3 \colon \pmb{\cW}_0\to\pmb{\cW}_0$, that sends one of the brown geodesic edges of $\partial\widehat{\pmb{T}}_0$ to the other, is conjugated to $M_i:\D\to\D$ under $\pmb{\Phi}_0^{-1}$.

\item The map $\eta$ preserves $\pmb{\cW}_1$, and induces an involution of the geodesic $\partial\cV^0\cap\pmb{\cW}_1$. By the equivariance properties of the maps $w\mapsto w^4$, $\mathfrak{X}$, $\mathfrak{h}$, and by the normalization that $\partial\cV^0\cap\pmb{\cW}_j$ is one of the sides of $\partial\widehat{\pmb{T}}_j$, we can deduce that $\pmb{\Phi}_1$ conjugates $g_1\in G$ (shown in top left of the figure) to $\eta\vert_{\pmb{\cW}_1}$. The relation $\pmb{\Phi}_1\equiv \tau\circ\pmb{\Phi}_0$ now implies that $\pmb{\Phi}_0$ conjugates $g_1\in G$ to $\tau^{-1}\circ\eta\circ\tau\in\pmb{G}_0$. As $g_1$ preserves one of the sides of $\Pi_G$, and since $\pmb{\Phi}_0(\Pi_G)=\widehat{\pmb{T}}_0$, it follows that this element of $\pmb{G}_0$ pairs one of the (cyan) geodesics on $\partial\widehat{\pmb{T}}_0$ with itself.

\item Finally, $\eta$ maps $\pmb{\cW}_0$ onto $\pmb{\cW}_2$, and carries $\partial\cV^0\cap\pmb{\cW}_0$ onto $\partial\cV^0\cap\pmb{\cW}_2$.
Hence $\tau^{-2}\circ\eta$ preserves $\pmb{\cW}_0$. Once again, the equivariance properties and normalization mentioned in the previous paragraph imply that $\pmb{\Phi}_2\circ g_2=\eta\circ\pmb{\Phi}_0$ (where the action of $g_2$ is shown in top left of Figure~\ref{corr_group_action_fig}). By the relation $\pmb{\Phi}_2\equiv \tau^2\circ\pmb{\Phi}_0$, we conclude that $\pmb{\Phi}_0$ conjugates $g_2\in G$ to $\tau^{-2}\circ\eta\in\pmb{G}_0$. As $g_2$ maps one of the sides of $\Pi_G$ to another, the element $\tau^{-2}\circ\eta$ of $\pmb{G}_0$ sends the (cyan) geodesic $\partial\cV^0\cap\pmb{\cW}_0$ to another (cyan) geodesic side of $\partial\widehat{\pmb{T}}_0$.
\end{enumerate}

This shows that the generators $\tau^3$, $\tau^{-1}\circ\eta\circ\tau$, and $\tau^{-2}\circ\eta$ of $\pmb{G}_0$ induce side-pairing transformations for the fundamental domain $\widehat{\pmb{T}}_0$, and $\pmb{\Phi}^{-1}$ conjugates these elements to the generators $M_i$, $g_1$, $g_2$ (respectively) of $G$ that are side-pairing transformations for the fundamental domain $\Pi_G$. Hence, $\pmb{\Phi}$ descends to a conformal isomorphism between the orbifolds $\D/G$ and $\pmb{\cW}_0/\pmb{G}_0$.

\begin{question}\label{pre_compactness_qstn}
Let $G_0\in\mathscr G_{\attr}$, $f\in\mathscr K_{\attr}$, and $U$ be the immediate basin of an attracting fixed point of $f$. Consider the collection of holomorphic correspondences 
$$
(z,w) \in \mathfrak{C}_G \ \iff\ \mathfrak{h}_G(w)=\mathfrak{h}_G(\eta(z))
$$ 
obtained by replacing $f\vert_{U}$ with $F_G$, for $G\in\mathrm{Teich}(G_0)$.
Is the family $\{\mathfrak{C}_G\}$ of correspondences, or equivalently the family $\{\mathfrak{h}_G\}$ of meromorphic functions, pre-compact in some natural topology in such a way that the limiting correspondences combine $f$ with Kleinian groups lying on the boundary of the Bers slice of $G_0$ ?
\end{question}
\noindent (See \cite{LMM26} for analogous pre-compactness results for algebraic correspondences.)

\end{document}